\documentclass[pdflatex,sn-mathphys]{sn-jnl}% Math and Physical Sciences

\usepackage{algorithm}               %format of the algorithm
\usepackage{algcompatible,amsmath,amsfonts,bm,amssymb}             %format of the algorithm
\usepackage{url}
\newtheorem{problem}{Problem}

\jyear{2021}%

\theoremstyle{thmstyleone}%
\newtheorem{theorem}{Theorem}%  meant for continuous numbers
\theoremstyle{thmstyletwo}%
\newtheorem{remark}{Remark}%
\newtheorem{lemma}{Lemma}%
\newtheorem{corollary}{Corollary}%

\theoremstyle{thmstylethree}%
\newtheorem{definition}{Definition}%

\begin{document}

\title[Efficient algorithms for Tucker decomposition]{Efficient algorithms for Tucker decomposition via approximate matrix multiplication}


\author[1]{\fnm{Maolin} \sur{Che}}\email{chncml@outlook.com and cheml@swufe.edu.cn}

\author[2]{\fnm{Yimin} \sur{Wei}}\email{ymwei@fudan.edu.cn}

\author[3]{\fnm{Hong} \sur{Yan}}\email{h.yan@cityu.edu.hk}

\affil[1]{\orgdiv{School of Mathematics}, \orgname{Southwestern University of Finance and Economics}, \orgaddress{\city{Chengdu}, \postcode{611130}, \country{P. R. of China}}}

\affil[2]{\orgdiv{School of Mathematical Sciences and Key Laboratory of Mathematics for Nonlinear Sciences}, \orgname{Fudan University}, \orgaddress{\city{Shanghai}, \postcode{200433}, \country{P. R. of China}}}

\affil[3]{\orgdiv{Department of Electrical Engineering and Center for Intelligent Multidimensional Data Analysis}, \orgname{City University of Hong Kong}, \orgaddress{\street{83 Tat Chee Avenue}, \city{Kowloon}, \country{Hong Kong}}}


\abstract{This paper develops fast and efficient algorithms for computing Tucker decomposition with a given multilinear rank. By combining random projection and the power scheme, we propose two efficient randomized versions for the truncated high-order singular value decomposition (T-HOSVD) and the sequentially T-HOSVD (ST-HOSVD), which are two common algorithms for approximating Tucker decomposition. To reduce the complexities of these two algorithms, fast and efficient algorithms are designed by combining two algorithms and approximate matrix multiplication. The theoretical results are also achieved based on the bounds of singular values of standard Gaussian matrices and the theoretical results for approximate matrix multiplication. Finally, the efficiency of these algorithms are illustrated via some test tensors from synthetic and real datasets.}




\maketitle

\section{Introduction}

It is natural to use higher-order tensors, multidimensional matrices, or multiway arrays to represent the data from a wide range of applications, such as machine learning, signal processing, higher-order statistics, computer vision, pattern recognition and recommendation system \cite{cichocki2016tensor,cichocki2017tensor,cichocki2015tensor,Grasedyck2013a,vervliet2014breaking}. A tensor is denoted by $\mathcal{A}\in\mathbb{R}^{I_1\times I_2\times \dots \times I_N}$ with entries given by
$a_{i_1i_2 \dots i_N}\in\mathbb{R}$, where $i_n=1,2,\dots,I_n$ and $n=1,2,\dots,N$. In particular, when $N=1$, a first-order tensor $\mathcal{A}$ is a vector of size $I_1$, and when $N=2$, a second-order tensor $\mathcal{A}$ is a matrix of size $I_1\times I_2$.

We first introduce some notations. We use lower case letters (e.g. $x,u,v$) for scalars, lower case bold letters (e.g. $\mathbf{x},\mathbf{u},
\mathbf{v}$) for vectors, bold capital letters (e.g. $\mathbf{A},\mathbf{B},\mathbf{C}$) for matrices, and calligraphic letters $\mathcal{A},\mathcal{B},\mathcal{C},\dots$ for tensors. This notation is consistently used for the lower-order parts of a given structure. For example, the entry with row index $i$ and column index $j$ in a matrix ${\bf A}$, i.e., $({\bf A})_{ij}$, is represented as $a_{ij}$ (also $(\mathbf{x})_i=x_i$ and $(\mathcal{A})_{i_1i_2\dots i_N}=a_{i_1i_2\dots i_N}$).

The symbols $\mathbf{A}^\top$, $\mathbf{A}^{\dag}$, $\|\mathbf{A}\|_F$ and $\|\mathbf{A}\|_2$ are, respectively, denoted by the transpose, the Moore-Penrose pseudoinverse, the Frobenius norm and the spectral norm of $\mathbf{A}\in\mathbb{R}^{I\times J}$. We use $\mathbf{I}_J$ to denote the identity matrix in $\mathbb{R}^{J\times J}$. The Kronecker product of matrices
$\mathbf{A}\in\mathbb{R}^{I\times J}$ and $\mathbf{B}\in\mathbb{R}^{K\times L}$ is represented as $\mathbf{A} \otimes \mathbf{B}$ . An orthogonal projection $\mathbf{P}\in\mathbb{R}^{I\times I}$ is defined as ${\bf P}^2={\bf P}$ and ${\bf P}^\top={\bf P}$. A matrix $\mathbf{Q}\in\mathbb{R}^{I\times K}\ (I>K)$ is orthonormal if $\mathbf{Q}^\top\mathbf{Q}=\mathbf{I}_K$. We use $\mathbb{S}_N$ to denote the $N$th-order symmetric group on the set $\{1,2,\dots,N\}$.  We use the interpretation of $O(\cdot)$ to refer to the class of functions the growth of which is bounded above and below up to a constant. The term ``for each $n$'' means $n=1,2,\dots, N$. The term ``for a given $n$'' means the integer $n$ satisfies $1\leq n\leq N$. For an event $\mathfrak{X}$, we use $\mathbb{P}(\mathfrak{X})$ and $\mathbb{E}(\mathfrak{X})$ to denote its probability and expectation, respectively.

As $N$ or some $I_n$'s increase, it is impossible to deal with tensors directly. Hence, it is necessary to use fewer parameters to effectively represent these large-scale tensors. Tensor decomposition is one of the most effective methods to represent large-scale data. Unlike low rank matrix approximations, there exist several methods for low rank approximations of tensors, such as CANDECOMP/PARAFAC (CP) decomposition \cite{carroll1970analysis}, Tucker decomposition \cite{tucker1966some}, Hierarchical Tucker decomposition \cite{kressner2014algorithm,grasedyck2010hierarchical}, tensor train decomposition \cite{oseledets2011tensor}, t-SVD \cite{kilmer2011factorization}, tensor ring decomposition \cite{sedighin2021adaptive,zhao2016tensor} and fully connected tensor network decomposition \cite{zheng2021nonlocal}. In this paper, we focus on the computation of low multilinear rank approximations, which can be viewed as a special case of Tucker decomposition.

With the desired multilinear rank, we consider the computation of Tucker decomposition with all the mode-$n$ factor matrices being orthonormal, which is also called the low rank approximation of a tensor in the Tucker format.

\begin{problem}
	\label{RALMRA:pro1}
	Suppose that $\mathcal{A}\in\mathbb{R}^{I_1\times I_2\times \dots\times I_N}$. For given $N$ positive integers $\{\mu_1,\mu_2,\dots,\mu_N\}$ with $\mu_n\ll I_n$, the goal is to find $N$ orthonormal matrices ${\bf Q}_{n}\in\mathbb{R}^{I_{n}\times \mu_{n}}$ such that
	\begin{equation*}
	a_{i_1i_2\dots i_N}\approx\sum_{j_1,j_2,\dots,j_N=1}^{I_1,I_2,\dots,I_N}a_{j_1j_2\dots j_N}(\mathbf{P}_{1})_{i_1j_1}(\mathbf{P}_{2})_{i_2j_2}\dots (\mathbf{P}_{N})_{i_Nj_N},
	\end{equation*}
	where $\mathbf{P}_{n}=\mathbf{Q}_{n}\mathbf{Q}_{n}^\top\in\mathbb{R}^{I_n\times I_n}$ is an orthogonal projection.
\end{problem}
Assume that $\{\mathbf{Q}_1,\mathbf{Q}_2,\dots,\mathbf{Q}_N\}$ is a solution of Problem \ref{RALMRA:pro1}. Then the tensor $\mathcal{A}\times_1(\mathbf{Q}_1\mathbf{Q}_1^\top)\times_2(\mathbf{Q}_2\mathbf{Q}_2^\top)\dots
\times_N(\mathbf{Q}_N\mathbf{Q}_N^\top)$ is called a low rank approximation in the Tucker format to $\mathcal{A}$. The developed algorithms for solving Problem \ref{RALMRA:pro1} can be divided into iterative algorithms and direct algorithms. Iterative algorithms include the higher-order orthogonal iteration \cite{lathauwer2000on}, the Newton-Grassmann method \cite{elden2009a}, the Riemannian trust-region method \cite{ishteva2011best}, the Quasi-Newton method \cite{savas2010quasi}, the semi-definite programming (SDP) \cite{navasca2009low}, and the Lanczos-type iteration \cite{goreinov2012wedderburn,savas2013krylov}. The truncated higher-order singular value decomposition (T-HOSVD) \cite{delathauwer2000a} and the sequentially truncated HOSVD (ST-HOSVD) \cite{vannieuwenhoven2012new} are two direct algorithms to find an approximate solution of Problem \ref{RALMRA:pro1}.

Recently, many researchers established the randomized variants of both T-HOSVD and ST-HOSVD. By using the randomized range finder algorithm \cite{halko2011finding} to find the factor matrix from the corresponding mode unfolding, Zhou {\it et al.} \cite{zhou2014decomposition} developed a distributed randomized Tucker decomposition approach for arbitrarily big tensors with relatively low multilinear ranks, and  Minster {\it et al.} \cite{minster2020randomized} presented randomized versions of T-HOSVD and ST-HOSVD, and gave a detailed probabilistic analysis of the error in using both algorithms. As shown in \cite{che2019randomized,che2020theory,minster2020randomized}, the adaptive randomized versions of T-HOSVD and ST-HOSVD are obtained by using the adaptive randomized range finder \cite{halko2011finding} for computing each mode factor matrix from the corresponding mode unfolding. By applying a randomized linear map to a tensor, Sun {\it et al.} \cite{sun2020low-rank} obtained a sketch that captures the important directions within each mode, and then described a new algorithm for computing a low rank approximation of the tensor in the Tucker format. The main difference between the work in \cite{che2019randomized} and that of \cite{minster2020randomized} is that the standard Gaussian vector is applied in \cite{minster2020randomized} and the Kronecker product of standard Gaussian vectors is exploited in \cite{che2019randomized}. After obtaining all the factor matrices, Sun {\it et al.} \cite{sun2020low-rank} further considered an efficient algorithm for computing the core tensor. Ahmadias {\it et al.} \cite{ahmadiasl2021randomized} provided more ways to approximate the first $\mu_n$ left singular vectors of the mode-$n$ unfolding. Che {\it et al.} \cite{che2021randomized} investigated the randomized versions of T-HOSVD and ST-HOSVD, where the Kronecker product of the standard Gaussian matrices are employed in the randomized range finder algorithm. When the singular values of each unfolding of the original tensor decay slowly, an efficient randomized algorithm combining the result of \cite{che2021randomized} and the power scheme is presented in \cite{che2020computation}. The storage of the Kronecker product of standard Gaussian matrices to project each mode unfolding is less than that by the standard Gaussian matrix and the Khatri-Rao product of standard Gaussian matrices to project the unfolding.

As shown in \cite{che2020computation,che2021an}, when the randomized algorithms are designed by combining random projection and power scheme to obtain all the factor matrices, the accuracy of these algorithms are competitive with that of some deterministic algorithms, such as HOOI, T-HOSVD and ST-HOSVD. However, these algorithms are slower than T-HOSVD and ST-HOSVD. Meanwhile, the randomized algorithms without the power scheme are faster than HOOI, T-HOSVD and ST-HOSVD, but may be worse than these algorithms for some test tensors. Hence, the main research is to design fast and efficient algorithms for computing the low multilinear rank approximation of a tensor with a known multilinear rank. Numerical examples illustrate that the proposed algorithms are faster than HOOI, T-HOSVD and ST-HOSVD, and their accuracy are competitive with that of these deterministic algorithms.

Without loss of generality, the processing order used in ST-HOSVD is set to $\{1,2,\dots,N\}$. We now give some differences between our algorithms and the existing randomized algorithms:
\begin{enumerate}
\item[(a)] In our algorithms, to generate $N$ standard Gaussian matrices needs $\sum_{n=1}^{N}O(I_n(\mu_n+K))$ operations; in the randomized variant of T-HOSVD \cite{minster2020randomized}, to generate $N$ standard Gaussian matrices requires $\sum_{n=1}^{N}O(I_1\dots I_{n-1}I_{n+1}\dots I_N(\mu_n+K))$ operations; and in the randomized variant of ST-HOSVD \cite{minster2020randomized}, to generate $N$ standard Gaussian matrices requires $\sum_{n=1}^{N}O(\mu_1\dots \mu_{n-1}I_{n+1}\dots I_N(\mu_n+K))$ operations. For each $n$, the random matrix used in \cite{che2019randomized} is the Khatri-Rao product of standard Gaussian matrices and to generate this random matrix needs $O((I_1+\dots+I_{n-1}+I_{n+1}+\dots+I_N)(\mu_n+K))$ operations for the T-HOSVD framework and $O((\mu_1+\dots+\mu_{n-1}+I_{n+1}+\dots+I_N)(\mu_n+K))$ operations for the ST-HOSVD framework. Meanwhile, the random matrix used in \cite{che2020computation,che2021randomized} is the Kronecker product of standard Gaussian matrices and to generate this matrix needs $O(\sum_{m=1,m\neq n}^{N}(I_mT_m))$ operations for the T-HOSVD framework and $O(\mu_1T_{n,1}+\dots+\mu_{n-1}I_{n,n-1}+I_{n+1}I_{n,n+1}+\dots+I_NT_{n,N})$ operations for the ST-HOSVD framework, where $T_{n,1}\dots T_{n,n-1}T_{n,n+1}\dots T_{n,N}$ is the smallest integer greater than $\mu_n+K$.
\item[(b)] For each $n$, in our algorithms, each factor matrix $\mathbf{Q}_n$ is obtained by applying the thin singular value decomposition (SVD) to the matrix $\mathbf{C}_n$, which is the same as the way to obtain the factor matrix $\mathbf{Q}_n$ in \cite{che2020computation,che2021an,che2021randomized}, and in randomized variants of T-HOSVD and ST-HOSVD, the matrix $\mathbf{Q}_n$ is obtained by this process: 1) to compute the thin QR decomposition: $\mathbf{Q}\mathbf{R}=\mathbf{C}_n$, 2) to compute the thin SVD: $\mathbf{U}{\bf \Sigma}\mathbf{V}^\top=\mathbf{Q}^\top\mathbf{A}_{(n)}$, and 3) to obtain the factor matrix $\mathbf{Q}_n=\mathbf{Q}\mathbf{U}(:,1:\mu_n)$.
\item[(c)] For each $n$, in Algorithms \ref{RALMRA:alg1} and \ref{RALMRA:alg2}, the matrix $\mathbf{C}_n$ is computed as $\mathbf{C}_n=(\mathbf{A}_{(n)}\mathbf{A}_{(n)}^\top)^q\mathbf{G}_{n,1}$, and when power iterations are involved in randomized variants of T-HOSVD and ST-HOSVD, the matrix $\mathbf{C}_n$ is computed as $\mathbf{C}_n=(\mathbf{A}_{(n)}\mathbf{A}_{(n)}^\top)^q\mathbf{A}_{(n)}\mathbf{G}_{n,2}$, which is the same as in \cite{che2020computation}, except for the way to form the matrix $\mathbf{G}_{n,2}$. Sometimes, the cost to generate $\mathbf{G}_{n,1}$ is far less than that to generate $\mathbf{G}_{n,2}$.
\item[(d)] By combining the randomized algorithms for low rank approximations \cite{martinsson2011randomized,rokhlin2010randomized} in the framework of T-HOSVD and ST-HOSVD, we obtain the upper bounds for the accuracy of the proposed algorithms. The theoretical bounds for the accuracy of randomized variants of T-HOSVD and ST-HOSVD are based on that for randomized SVD in \cite{zhang2018a}.
\end{enumerate}

The rest of the paper is organized as follows. In Section \ref{RALMRA:sec2}, we overview the basic definitions of tensors and some results for Tucker decomposition. In Sections \ref{RALMRA:sec3} and \ref{RALMRA:sec4}, we respectively describe randomized versions of T-HOSVD and ST-HOSVD by combining random projection, the power scheme and approximate matrix multiplication. The upper bounds of $\|\mathcal{A}-\widehat{\mathcal{A}}\|_F^2$ are obtained in Section \ref{RALMRA:sec5}, where $\widehat{\mathcal{A}}$ is a low rank approximation in the Tucker format obtained by the proposed algorithms. We illustrate the efficiency and accuracy of the proposed algorithms via some test tensors from synthetic and real datasets in Section \ref{RALMRA:sec6} and conclude this paper in Section \ref{RALMRA:sec7}.
\section{Preliminaries}
\label{RALMRA:sec2}

We recommend \cite{che2020theory,cichocki2016tensor,cichocki2017tensor,cichocki2015tensor,kolda2019tensor} for a thorough survey of basic definitions of tensors. For a given $n$, the mode-$n$ product of $\mathcal{A}\in \mathbb{R}^{I_1\times I_2\times \dots \times I_N}$ by ${\bf B}\in \mathbb{R}^{J_n\times I_n}$, denoted
by $\mathcal{A}\times_{n}{\bf B}$, is $\mathcal{C}\in\mathbb{R}^{I_1\times \dots\times I_{n-1}\times J_n\times I_{n+1}\times \dots \times I_{N}}$, where its entries are given by
$$c_{i_{1}\dots i_{n-1}ji_{n+1}\dots i_{m}}=\sum_{i_{n}=1}^{I_n}a_{i_{1}\dots i_{n-1}i_{n}i_{n+1}\dots i_{N}}b_{ji_{n}}.$$
Suppose that $\mathcal{A}\in \mathbb{R}^{I_1\times I_2\times \dots \times I_N}$, $\mathbf{F}\in \mathbb{R}^{J_n\times I_n}$, $\mathbf{G}\in \mathbb{R}^{J_m\times I_m}$ and $\mathbf{H}\in\mathbb{R}^{J'_n\times J_n}$. For each $m$ and $n$ with $m\neq n$, we have
\begin{equation*}
(\mathcal{A}\times_{n}\mathbf{F})\times_{m}\mathbf{G}
=(\mathcal{A}\times_{m}\mathbf{G})\times_{n}\mathbf{F}=\mathcal{A}\times_{n}\mathbf{F}\times_{m}\mathbf{G},\quad (\mathcal{A}\times_{n}\mathbf{F})\times_{n}\mathbf{H}=\mathcal{A}\times_{n}(\mathbf{H}\mathbf{F}).
\end{equation*}

For two tensors $\mathcal{A},\mathcal{B}\in \mathbb{R}^{I_1\times I_2\times\dots\times I_N}$, the {\it Frobenius norm} of a tensor $\mathcal{A}$ is given by  $\|\mathcal{A}\|_{F}=\sqrt{\langle\mathcal{A},\mathcal{A}\rangle}$ and the scalar product $\langle\mathcal{A},\mathcal{B}\rangle$ is defined as
$$\langle\mathcal{A},\mathcal{B}\rangle=\sum_{i_1,i_2,\dots,i_N=1}^{I_1,I_2,\dots,I_N}a_{i_{1}i_{2}\dots i_{N}}b_{i_{1}i_{2}\dots i_{N}}.$$

For a given $m\in\{1,2,\dots, M\}$ and $n$, the mode-$(n,m)$ product \cite{oseledets2011tensor} (called {\it tensor contraction}) of two tensors $\mathcal{A}\in \mathbb{R}^{I_1\times I_2\times \dots \times I_N}$ and $\mathcal{B}\in \mathbb{R}^{J_1\times J_2\times \dots\times J_M}$ with common modes $I_n=J_m$ produces an order $(M+N-2)$ tensor $$\mathcal{C}=\mathcal{A}\times_{n}^m\mathcal{B}\in\mathbb{R}^{I_1\times \dots \times I_{n-1}\times I_{n+1}\times\dots\times I_N\times J_1\times\dots\times J_{m-1}\times J_{m+1}\times \dots \times J_M},$$ where its entries are given by
\begin{equation*}
c_{i_1\dots i_{n-1}i_{n+1}\dots i_Nj_1\dots j_{m-1}i_{m+1}\dots j_N}=\sum_{i_n=1}^{I_n}a_{i_1\dots i_{n-1}i_ni_{n+1}\dots i_N}b_{j_1\dots j_{m-1}i_nj_{m+1}\dots\dots j_M}.
\end{equation*}

For a given $n$, the mode-$n$ unfolding matrix of $\mathcal{A}\in \mathbb{R}^{I_1\times I_2\times \dots \times I_N}$, denoted by ${\bf A}_{(n)}$, arranges the mode-$n$ fibers into the columns of a matrix. More specifically, a tensor element $(i_1,i_2,\dots,i_N)$ maps on a matrix element $(i_n,j)$, where
\begin{equation*}
\begin{aligned}
j=&i_1+(i_{2}-1)I_1+\dots+(i_{n-1}-1)I_1\dots I_{n-2}+(i_{n+1}-1)I_1\dots I_{n-1}\\
&+\dots+(i_N-1)I_1\dots I_{n-1}I_{n+1}\dots I_{N-1}.
\end{aligned}
\end{equation*}
\subsection{Tucker decomposition: a review}
For given $N$ positive integers $\{R_1,R_2,\dots,R_N\}$ with $R_n<I_n$, {\it Tucker decomposition} \cite{tucker1966some} of a tensor $\mathcal{A}\in \mathbb{R}^{I_1\times I_2\times \dots \times I_N}$ is defined as
\begin{equation}
\label{RALMRA:eqn1}
\mathcal{A}\approx\mathcal{G}\times_1{\bf U}^{(1)}\times_2{\bf U}^{(2)}\dots\times_N{\bf U}^{(N)},
\end{equation}
where ${\bf U}^{(n)}\in \mathbb{R}^{I_n\times R_n}$ are called the {\it mode-$n$ factor matrices} and $\mathcal{G}\in \mathbb{R}^{R_1\times R_2\times\dots\times R_N}$ is called the {\it core tensor}. In general, the matrix $\mathbf{U}_n$ is not be restricted to be orthonormal. Hence, the low rank approximation in the Tucker format from any solution of Problem \ref{RALMRA:pro1} is a special case of Tucker decomposition. Two other special cases are the nonnegative Tucker decomposition (cf. \cite{zhou2012fast}) and the higher-order interpolatory decomposition (cf. \cite{drineas2007a,saibaba2016hoid}).

For each $n$, Tucker decomposition is closely related to the mode-$n$ unfolding matrix $\mathbf{A}_{(n)}\in\mathbb{R}^{I_n\times I_1\dots I_{n-1}I_{n+1}\dots I_N}$. In particular, the relation (\ref{RALMRA:eqn1}) implies
\begin{equation*}
{\bf A}_{(n)}\approx{\bf U}^{(n)}{\bf G}_{(n)}({\bf U}^{(N)}\otimes\dots\otimes {\bf U}^{(n+1)}\otimes {\bf U}^{(n-1)}
\otimes\dots\otimes {\bf U}^{(1)})^{\top},
\end{equation*}
where $\mathbf{G}_{(n)}\in\mathbb{R}^{R_n\times R_1\dots R_{n-1}R_{n+1}\dots R_N}$.

It follows that the rank of ${\bf A}_{(n)}$ is less than or equal to $R_n$, as the mode-$n$ factor ${\bf U}^{(n)}\in\mathbb{R}^{I_n\times R_n}$  at most has rank  $R_n$. This motivates us to define the multilinear rank of $\mathcal{A}$ as the tuple
\begin{equation*}
\{R_1,R_2,\dots,R_N\},\quad\text{where the rank of }{\bf A}_{(n)}\text{ is equal to }R_n.
\end{equation*}

A classical algorithm for computing Tucker decomposition with all the factor matrices being orthonormal is named after the T-HOSVD \cite{kolda2019tensor}. This algorithm is not optimal in terms of giving the best fitting as measured by the norm of the difference, but it is a good starting point for the HOOI \cite{lathauwer2000on} to compute the Tucker decomposition with orthonormal factor matrices. Vannieuwenhoven {\it et al.} \cite{vannieuwenhoven2012new} presented an alternative strategy for truncating the HOSVD, which is denoted by ST-HOSVD. In contrast to the T-HOSVD algorithm that processes the modes independently, the ST-HOSVD algorithm processes the modes sequentially with a given processing order in $\mathbb{S}_N$.

\section{Randomized T-HOSVD with power scheme}
\label{RALMRA:sec3}
By using the Frobenius norm of a tensor, we can rewrite Problem \ref{RALMRA:pro1} as follows,
\begin{equation}\label{RALMRA:eqn4}
\begin{cases}
\max&\quad\left\|\mathcal{A}\times_1\mathbf{Q}_1^\top\times_2\mathbf{Q}_2^\top\dots
\times_N\mathbf{Q}_N^\top\right\|_F^2,\\
\text{subject to}&\quad\mathbf{Q}_n\in\mathbb{R}^{I_n\times \mu_n}\text{ is orthonormal with }n=1,2,\dots,N.
\end{cases}
\end{equation}
According to \cite{lathauwer2000on}, researching a solution for (\ref{RALMRA:eqn4}) is equivalent to find a tensor $\widehat{\mathcal{A}}\in\mathbb{R}^{I_1\times I_2\times\dots\times I_N}$ that minimizes the least-squares cost function $f(\widehat{\mathcal{A}})=\|\mathcal{A}-\widehat{\mathcal{A}}\|_F^2$, with the constraints ${\rm rank}(\widehat{\mathbf{A}}_{(n)})=\mu_n$, where $\widehat{\mathbf{A}}_{(n)}$ is the mode-$n$ unfolding of $\widehat{\mathcal{A}}$.

Let $\{\mathbf{Q}_1,\mathbf{Q}_2,\dots,\mathbf{Q}_N\}$ be a solution for (\ref{RALMRA:eqn4}). As shown in \cite{saibaba2016hoid,vannieuwenhoven2012new}, we have
\begin{equation}\label{RALMRA:eqn5}
\begin{aligned}
&\left\|\mathcal{A}-\mathcal{A}\times_1\left({\bf Q}_1{\bf Q}_1^\top\right)\times_2\left({\bf Q}_2{\bf Q}_2^\top\right)\dots\times_N\left({\bf Q}_N{\bf Q}_N^\top\right)\right\|_F^2\\
&\leq\sum_{n=1}^N\left\|\mathcal{A}-\mathcal{A}\times_n\left({\bf Q}_n{\bf Q}_n^\top\right)\right\|_F^2.
\end{aligned}
\end{equation}

Hence, an approximate solution $\{\mathbf{Q}_1,\mathbf{Q}_2,\dots,\mathbf{Q}_N\}$ for (\ref{RALMRA:eqn4}) can be obtained by finding an orthonormal matrix $\mathbf{Q}_n\in\mathbb{R}^{I_n\times \mu_n}$ such that
\begin{equation*}
\begin{aligned}
\mathbf{Q}_n=&\ {\rm argmin}\left\|\mathcal{A}-\mathcal{A}\times_n\left({\bf Q}{\bf Q}^\top\right)\right\|_F^2={\rm argmin}\left\|\mathbf{A}_{(n)}-\left({\bf Q}{\bf Q}^\top\right)\mathbf{A}_{(n)}\right\|_F^2,\\
&\ {\rm subject to}\quad \mathbf{Q}\in\mathbb{R}^{I_n\times \mu_n}\ \text{ is orthonormal},n=1,2,\dots,N.
\end{aligned}
\end{equation*}
\subsection{The proposed algorithm}

According to the above descriptions, for each $n$, the columns of the matrix $\mathbf{Q}_n$ are the first $\mu_n$ left singular vectors of $\mathbf{A}_{(n)}$. Note that computing the first $\mu_n$ left singular vectors is prohibitive for high dimensions. For each $n$, through Algorithm \ref{RALMRA:alg1}, we obtain the approximation for the first $\mu_n$ left singular vectors of $\mathbf{A}_{(n)}$. In detail, firstly, we use a standard Gaussian matrix $\mathbf{G}$ to project the matrix $(\mathbf{A}_{(n)}\mathbf{A}_{(n)}^\top)^q$, where $q\geq 1$ is a given integer; secondly, we obtain an orthonormal matrix $\mathbf{Q}_n$ from this projection matrix; and finally, we compute the core tensor $\mathcal{G}=\mathcal{A}\times_1\mathbf{Q}_1^\top\times_2\mathbf{Q}_2^\top\dots\times_N\mathbf{Q}_N^\top$.
\begin{algorithm}[htb]
	\caption{Solving Problem \ref{RALMRA:pro1} via random projection and power scheme}
	\label{RALMRA:alg1}
	\begin{algorithmic}[1]
		\STATEx {\bf Input}: A tensor $\mathcal{A}\in \mathbb{R}^{I_1\times I_2\times\dots\times I_N}$, the desired multilinear rank $\{\mu_1,\mu_2,\dots,\mu_N\}$, the oversampling parameter $K$ and an integer $q\geq 1$.
		\STATEx {\bf Output}: $N$ orthonormal matrices $\mathbf{Q}_n$ and a core tensor $\mathcal{G}$ such that $\mathcal{A}\approx\mathcal{G}\times_1\mathbf{Q}_1\times_2\mathbf{Q}_2\dots\times_N\mathbf{Q}_N$.
		\STATE For $n=1,2,\dots,N$
		\STATE Generate a standard Gaussian matrix $\mathbf{G}\in\mathbb{R}^{I_n \times (\mu_n+K)}$.
		\STATE Arrange the mode-$n$ unfolding $\mathbf{A}_{(n)}$ of $\mathcal{A}$.
            \STATE Compute $\mathbf{C}_n=(\mathbf{A}_{(n)}\mathbf{A}_{(n)}^\top)^q\mathbf{G}_n$.
		\STATE Find an orthonormal matrix $\mathbf{Q}_n\in\mathbb{R}^{I_n\times \mu_n}$ such that there exists $\mathbf{S}_n\in\mathbb{R}^{\mu_n\times (\mu_n+K)}$ for which
	\begin{equation*}
	\|\mathbf{Q}_n\mathbf{S}_n-\mathbf{C}_n\|_F\leq\Delta_{\mu_n+1}(\mathbf{C}_n)=\left(\sum_{k=\mu_n+1}^{\mu_n+K}\sigma_k(\mathbf{C}_n)^2\right)^{1/2},
	\end{equation*}
	where $\sigma_k(\mathbf{C}_n)$ is the $k$th singular value of $\mathbf{C}_n$.
		\STATE End for
            \STATE Compute the core tensor $\mathcal{G}=\mathcal{A}\times_1\mathbf{Q}_1^\top\times_2\mathbf{Q}_2^\top\dots\times_N\mathbf{Q}_N^\top$.
	\end{algorithmic}
\end{algorithm}

\begin{remark}
\label{RALMRA:remadd1}
	In Algorithm {\rm \ref{RALMRA:alg1}}, for each $n$, the matrix $\mathbf{C}_n$ can be computed by one of the following two ways:
	\begin{enumerate}
		\item[] Strategy A: We first compute $\mathbf{C}'=\mathbf{A}_{(n)}^\top\mathbf{G}_n$ and $\mathbf{C}_n=\mathbf{A}_{(n)}\mathbf{C}'$, and then compute $\mathbf{C}'=\mathbf{A}_{(n)}^\top\mathbf{C}_n$ and $\mathbf{C}_n=\mathbf{A}_{(n)}\mathbf{C}'$ with $(q-1)$ times, which requires $4qI_1\dots I_{n-1}I_n\dots I_{N}(\mu_n+K)$ operations.
		\item[] Strategy B: We first compute $\mathbf{C}'=\mathbf{A}_{(n)}\mathbf{B}_{(A)}^\top$ and let $\mathbf{C}_n=\mathbf{G}_n$, and then compute $\mathbf{C}_n=\mathbf{C}'\mathbf{C}_n$ with $q$ times, which requires $2I_1\dots I_{n-1}I_n^2\dots I_{N}+2qI_n^2(\mu_n+K)$ operations.
	\end{enumerate}
 In Problem {\rm \ref{RALMRA:pro1}}, for each $n$, we assume that $\mu_n\ll I_n$. Hence, for some small $q$ and $K$, we can ensure that $2q(\mu_n+K)<I_n$. Numerical examples illustrate that the case of $q=1$ and $K=10$ is suitable for Algorithm {\rm \ref{RALMRA:alg1}}. Hence, in the rest, we use Strategy A to compute the matrix $\mathbf{C}_n$.
\end{remark}

We now study the computational complexity of Algorithm {\rm \ref{RALMRA:alg1}}. For clarity, we assume $I_1\sim I_2\sim\dots\sim I_N\sim I$ and $\mu_1\sim\mu_2\sim\dots\sim\mu_N\sim\mu$ in complexity estimates \cite[Page A2]{goreinov2012wedderburn}, where $\mu_n$ is the number of the columns of $\mathbf{Q}_n$, and $I_n\sim I$ means $I_n=\alpha_n I$ for some constant $\alpha_n$. To count the floating-point operations of Algorithm \ref{RALMRA:alg1}, for each $n$, we evaluate the complexity for each step:
\begin{enumerate}
	\item[(a)] to form a standard Gaussian matrix $\mathbf{G}_n$ requires $I(\mu+K)$ operations;
	\item[(b)] it needs $O(I^N)$ operations to form the matrix $\mathbf{A}_{(n)}$;
	\item[(c)] it requires $4qI^N(\mu+K)$ operations to compute the matrix $\mathbf{C}_n$; and
	\item[(d)] computing $\mathbf{Q}_n$ demends $O(I(\mu+K)\mu)$ operations.
\end{enumerate}
Meanwhile, it costs $2I\mu(I^{N-1}+I^{N-2}\mu+\dots+I\mu^{N-2}+\mu^{N-1})$ operations to obtain the core tensor $\mathcal{G}$. Then, by summing up the computational complexity for each $n$, Algorithm \ref{RALMRA:alg1} requires
\begin{equation*}
\begin{aligned}
&N\cdot\left(I(\mu+K)+O(I^N)+4qI^N(\mu+K)+O(I\mu^2+IK\mu)\right)\\
&+2I\mu(I^{N-1}+I^{N-2}\mu+\dots+I\mu^{N-2}+\mu^{N-1})
\end{aligned}
\end{equation*}
operations for the tensor $\mathcal{A}$.

%\begin{remark}
%	For each $n$ in Algorithm {\rm\ref{RALMRA:alg1}}, when we use the functions ``{\rm ttt}'' and ``{\rm ttm}'' in {\rm( \cite{tensortool})} to obtain the matrix $\mathbf{A}_{(n)}\mathbf{A}_{(n)}^\top$ as follows:
%	\begin{equation*}
% \mathbf{B}_{(n)}\mathbf{B}_{(n)}^\top
% \leftarrow{\rm ttt}(\mathcal{B},\mathcal{B},[1,\dots,n-1,n+1,\dots,N]),
%	\end{equation*}
%	Hence, the operations to obtain the matrix $\mathbf{B}_{(n)}\mathbf{B}_{(n)}^\top$ are denoted by $T_{{\rm tenten}}$.
	
%	Under this case, Algorithm {\rm\ref{RALMRA:alg1}} requires
%	\begin{equation*}
%	N\cdot\left(I(\mu+K)+2qI^2(\mu+K)+O(I\mu^2+IK\mu)+T_{{\rm tenten}}\right)
%	\end{equation*}
%	operations for the tensor $\mathcal{A}$.
%\end{remark}
\subsection{Accelerating Algorithm \ref{RALMRA:alg1} with approximate matrix multiplication}

For each $n$, the matrix $\mathbf{C}_n$ in Algorithm \ref{RALMRA:alg1} is computed by $(\mathbf{B}_{(n)}\mathbf{B}_{(n)}^\top)^q\mathbf{G}$, which requires $4qI_1\dots I_{n-1}I_n\dots I_{N}(\mu_n+K)$ operations. These operations are enormous when all the $I_n$ are large. In this section, we introduce a fast sampling algorithm to approximate the matrix $\mathbf{C}_n$.

Assume that $0\leq p_1,p_2,\dots,p_I\leq 1$ and $\sum_{i=1}^Ip_i=1$, then we say $\mathbf{p}=(p_1,p_2,\dots,p_I)\in[0,1]^{I}$ is a probability distribution. For a given random variable $\xi\in\{1,2,\dots,I\}$, we say $\xi\sim{\rm MULTINOMIAL}(\mathbf{p})$ if $\mathbf{p}\in[0,1]^{I}$ is a probability distribution and $\mathbb{P}(\xi=i)=p_i$. We also say $\mathbf{S}\in\mathbb{R}^{I\times L}\sim{\rm RANDSAMPLE}(L,\mathbf{p})$ if $L$ is a positive integer, $\mathbf{p}\in[0,1]^{I}$ is a probability distribution, and the entries of $\mathbf{S}$ are defined as follows,
\begin{equation*}
s_{ij}=
\begin{cases}
\begin{aligned}
\frac{1}{\sqrt{Lp_i}}, &\quad {\rm if }\ \xi_j=i,\\
0, &\quad {\rm otherwise},
\end{aligned}
\end{cases}
\end{equation*}
where $\xi_j\sim{\rm MULTINOMIAL}(\mathbf{p})$, $i=1,2,\dots,I$ and $j=1,2,\dots,L$.

Drineas {\it et al.} \cite{drineas2006fast} investigated a famous randomized algorithm for approximating matrix multiplication, called the BasicMatrixMultiplication algorithm (Algorithm \ref{RALMRA:alg3}).
\begin{algorithm}[htb]
	\caption{BasicMatrixMultiplication algorithm \cite[Fig. 2]{drineas2006fast}}
	\label{RALMRA:alg3}
	\begin{algorithmic}[1]
		\STATEx {\bf Input}: Two matrices $\mathbf{A}\in\mathbb{R}^{I_1\times I_2}$ and $\mathbf{B}\in\mathbb{R}^{I_2\times I_3}$, the number of sampling $K$ such that $1\leq K\leq I_2$, and a probability distribution $\{p_{i}\}_{i=1}^{I_2}$.
		\STATEx {\bf Output}: Two matrices $\mathbf{C}\in\mathbb{R}^{I_1\times K}$ and $\mathbf{R}\in\mathbb{R}^{K\times I_3}$ such that $\mathbf{A}\mathbf{B}\approx\mathbf{C}\mathbf{R}$.
		\STATE Form $\mathbf{S}\in\mathbb{R}^{I_2\times K}\sim{\rm RANDSAMPLE}(K,\mathbf{p})$.
		\STATE Compute $\mathbf{C}=\mathbf{A}\mathbf{S}$ and $\mathbf{R}=\mathbf{S}^\top\mathbf{B}$.
	\end{algorithmic}
\end{algorithm}
\begin{remark}
	An important issue for Algorithm {\rm \ref{RALMRA:alg3}} is the choice of the probabilities $\{p_{i}\}_{i=1}^{I_2}$:
	\begin{enumerate}
		\item[{\rm (a)}] The probabilities $\{p_{i}\}_{i=1}^{I_2}$ are the optimal probabilities if
		\begin{equation*}
		p_i=\frac{\|\mathbf{A}(i,:)\|_2\|\mathbf{B}(:,i)\|_2}{\sum_{j=1}^{I_2}\|\mathbf{A}(j,:)\|_2\|\mathbf{B}(:,j)\|_2};
		\end{equation*}
		\item[{\rm (b)}] The probabilities $\{p_{i}\}_{i=1}^{I_2}$ are the nearly optimal probabilities if for some positive constant $\beta\leq1$,
		\begin{equation*}
		p_i\geq\beta\frac{\|\mathbf{A}(i,:)\|_2\|\mathbf{B}(:,i)\|_2}{\sum_{j=1}^{I_2}\|\mathbf{A}(j,:)\|_2\|\mathbf{B}(:,j)\|_2};
		\end{equation*}
		\item[{\rm (c)}] The probabilities $\{p_{i}\}_{i=1}^{I_2}$ are the uniform probabilities if $p_i=1/I_2$.
	\end{enumerate}
\end{remark}

There exist two situations when we count the computational complexity of Algorithm \ref{RALMRA:alg3}: for the (nearly) optimal probabilities, to construct $\mathbf{C}$ and $\mathbf{R}$ requires $O(K(I_1+I_2+I_3))$ operations, and for the uniform probabilities, it needs $O(K(I_1+I_3))$ operations to obtain $\mathbf{C}$ and $\mathbf{R}$.
\begin{remark}
	In the above remark, we give one type of the nearly optimal probabilities. Furthermore, if all the $p_i$ satisfy one of the following conditions
	\begin{equation*}
	\begin{aligned}
	&p_i\geq\beta\frac{\|\mathbf{A}(i,:)\|_2}{\sum_{j=1}^{I_2}\|\mathbf{A}(j,:)\|_2},\ p_i\geq\beta\frac{\|\mathbf{B}(:,i)\|_2}{\sum_{j=1}^{I_2}\|\mathbf{B}(:,j)\|_2},\\
 &p_i\geq\beta\frac{\|\mathbf{A}(i,:)\|_2^2}{\|\mathbf{A}\|_F^2},\ p_i\geq\beta\frac{\|\mathbf{B}(:,i)\|_2^2}{\|\mathbf{B}\|_F^2},\ \text{or }\ p_i\geq\beta\frac{\|\mathbf{A}(i,:)\|_2\|\mathbf{B}(:,i)\|_2}{\|\mathbf{A}\|_F\mathbf{B}\|_F},
	\end{aligned}
	\end{equation*}
	then the probabilities $\{p_{i}\}_{i=1}^{I_2}$ are the nearly optimal probabilities.
\end{remark}

In Algorithm \ref{RALMRA:alg3}, the strategy used to generate $\mathbf{S}$ is a random sampling. There exist other types of random sampling: including subspace sampling \cite{drineas2006subspace} and adaptive sampling \cite{deshpande2006matrix,wang2013improving}. Recently, many researchers study various block versions of Algorithm \ref{RALMRA:alg3}, and the interested readers can refer to \cite{chang2019random,eriksson2011importance,niu2023optimal,wu2020multilevel} for their details. Meanwhile, there exists some randomized algorithms for computing the approximate matrix multiplication based on random projection \cite{cohen2015optimal,eftekhari2015restricted,srinivasa2020localized}, including standard Gaussian, SpEmb (sparse subspace embedding) and SRFT (subsampled randomized Fourier transform) matrices.

By using  Algorithm \ref{RALMRA:alg3} to compute an approximation of $\mathbf{C}_n$ in Algorithm \ref{RALMRA:alg1}, we obtain another efficient algorithm to compute a low multilinear rank approximation of a tensor $\mathcal{A}\in \mathbb{R}^{I_1\times I_2\times\dots\times I_N}$, denoted by Algorithms \ref{RALMRA:alg4}.
\begin{algorithm}
	\caption{Accelerating Algorithm {\rm \ref{RALMRA:alg1}} with approximating matrix multiplication}
	\label{RALMRA:alg4}
	\begin{algorithmic}[1]
	\STATEx {\bf Input}: A tensor $\mathcal{A}\in \mathbb{R}^{I_1\times I_2\times\dots\times I_N}$, the
                 desired multilinear rank $\{\mu_1,\mu_2,\dots,\mu_N\}$, the oversampling parameters $K$, $N$ positive integers $\{T_1,T_2,\dots,T_N\}$, an integer $q\geq 1$, and $N$ probability distributions $\{p_{n,i}\}_{i=1}^{I_1\dots I_{n-1}I_{n+1}\dots I_N}$.
        \STATEx {\bf Output}: $N$ orthonormal matrices $\mathbf{Q}_n$ and a core tensor $\mathcal{G}$ such that
                 $\mathcal{A}\approx\mathcal{G}\times_1\mathbf{Q}_1\times_2\mathbf{Q}_2\dots\times_N\mathbf{Q}_N$.
	\STATE For $n=1,2,\dots,N$
	\STATE Generate $\mathbf{S}_n\in\mathbb{R}^{I_1\dots I_{n-1}I_{n+1}\dots I_N\times T_n}\sim{\rm RANDSAMPLE}            (T_n,\mathbf{p}_n)$.
	\STATE Arrange $\mathbf{A}_{(n)}$ as the mode-$n$ unfolding of $\mathcal{A}$.
        \STATE Obtain $\mathbf{C}_n'=\mathbf{A}_{(n)}\mathbf{S}_n$.
	\STATE Form a standard Gaussian matrix $\mathbf{G}_n\in\mathbb{R}^{I_n \times (\mu_n+K)}$.
	\STATE Compute $\mathbf{C}_n=(\mathbf{C}_n'\mathbf{C}_n'^\top)^q\mathbf{G}_n$.
	\STATE Find an orthonormal matrix $\mathbf{Q}_n\in\mathbb{R}^{I_n\times \mu_n}$ such that there exists $\mathbf{S}_n\in\mathbb{R}^{\mu_n\times (\mu_n+K)}$ for which
	\begin{equation*}
	\|\mathbf{Q}_n\mathbf{S}_n-\mathbf{C}_n\|_F\leq\Delta_{\mu_n+1}(\mathbf{C}_n).
	\end{equation*}
	\STATE End for
        \STATE Compute the core tensor
         $\mathcal{G}=\mathcal{A}\times_1\mathbf{Q}_1^\top\times_2\mathbf{Q}_2^\top\dots\times_N\mathbf{Q}_N^\top$.
	\end{algorithmic}
\end{algorithm}

To count the complexity of Algorithm {\rm \ref{RALMRA:alg4}}, we assume that $T_1\sim T_2\sim \dots\sim T_N\sim T$. Then, for each $n$, we have
\begin{enumerate}
        \item[(a)] it needs $O(I^N)$ operations to form the matrix $\mathbf{A}_{(n)}$;
	\item[(b)] for the (nearly) optimal probabilities, to construct $\mathbf{C}_n'$ requires $O((I+I^{N-1})T)$ operations, and for the uniform probabilities, it needs $O(IT)$ operations to obtain $\mathbf{C}_n'$;
	\item[(c)] to generate the matrix $\mathbf{G}_n\in\mathbb{R}^{I\times (\mu+K)}$ requires $I(\mu+K)$ operations;
	\item[(d)] it requires $4IT(\mu+K)$ operations to compute the matrix $\mathbf{C}_n$; and
	\item[(e)] computing $\mathbf{Q}_n$ requires $O(I(\mu+K)\mu)$ operations.
\end{enumerate}
Meanwhile, it costs $2I\mu(I^{N-1}+I^{N-2}\mu+\dots+I\mu^{N-2}+\mu^{N-1})$ operations to obtain the core tensor $\mathcal{G}$. Then, Algorithm \ref{RALMRA:alg4} requires
\begin{equation*}
\begin{aligned}
&N\cdot (O(I^N)+4IT(\mu+K)+I(\mu+K)+O(I\mu^2+IK\mu))\\
&+N\cdot O(IT)+2I\mu(I^{N-1}+I^{N-2}\mu+\dots+I\mu^{N-2}+\mu^{N-1})
\end{aligned}
\end{equation*}
operations for the uniform probabilities, and
\begin{equation*}
\begin{aligned}
&N\cdot (O(I^N)+4IT(\mu+K)+I(\mu+K)+O(I\mu^2+IK\mu))\\
&+N\cdot O((I+I^{N-1})T)+2I\mu(I^{N-1}+I^{N-2}\mu+\dots+I\mu^{N-2}+\mu^{N-1})
\end{aligned}
\end{equation*}
operations for the (nearly) optimal probabilities.
%\begin{remark}
%In Algorithms {\rm \ref{RALMRA:alg1}} and {\rm \ref{RALMRA:alg4}}, for each $n$, we assume that $\mu_n<\mu_n+K<T_n\leq \min\{I_n,I_1\dots I_{n-1}I_{n+1}\dots I_N\}$. Note that the size of the standard Gaussian matrix $\mathbf{G}_n$ in Algorithms {\rm \ref{RALMRA:alg1}} and {\rm \ref{RALMRA:alg4}} can be replaced by $I_n\times (\mu_n+K_n)$. However, the oversampling parameter
%\end{remark}
\section{Randomized ST-HOSVD with power scheme}
\label{RALMRA:sec4}
Let $\{\mathbf{Q}_1,\mathbf{Q}_2,\dots,\mathbf{Q}_N\}$ be a solution for (\ref{RALMRA:eqn4}). As shown in \cite{saibaba2016hoid,vannieuwenhoven2012new}, for a processing order $\{p_1,p_2,\dots,p_N\}\in\mathbb{S}_N$, we have
\begin{equation}\label{RALMRA:eqn6}
\begin{aligned}
&\mathcal{A}-\mathcal{A}\times_1\left({\bf Q}_1{\bf Q}_1^\top\right)\times_2\left({\bf Q}_2{\bf Q}_2^\top\right)\dots\times_N\left({\bf Q}_N{\bf Q}_N^\top\right)=\mathcal{A}-\mathcal{A}\times_{p_1}\left({\bf Q}_{p_1}{\bf Q}_{p_1}^\top\right)\\
&+\mathcal{A}\times_{p_1}\left({\bf Q}_{p_1}{\bf Q}_{p_1}^\top\right)-\mathcal{A}\times_{p_1}\left({\bf Q}_{p_1}{\bf Q}_{p_1}^\top\right)\times_{p_2}\left({\bf Q}_{p_2}{\bf Q}_{p_2}^\top\right)\\
&+\dots+\mathcal{A}\times_{p_1}\left({\bf Q}_{p_1}{\bf Q}_{p_1}^\top\right)\times_{p_2}\left({\bf Q}_{p_2}{\bf Q}_{p_2}^\top\right)\dots\times_{p_{N-1}}\left({\bf Q}_{p_{N-1}}{\bf Q}_{p_{N-1}}^\top\right)\\
&-\mathcal{A}\times_{p_1}\left({\bf Q}_{p_1}{\bf Q}_{p_1}^\top\right)\times_{p_2}\left({\bf Q}_{p_2}{\bf Q}_{p_2}^\top\right)\dots\times_{p_N}\left({\bf Q}_{p_N}{\bf Q}_{p_N}^\top\right).
\end{aligned}
\end{equation}
In particular, when $N=3$ and $\mathcal{A}\in\mathbb{C}^{I_1\times I_2\times I_3}$, the equality (\ref{RALMRA:eqn6}) appears in \cite{che2021an}; and when $\{p_1,p_2,\dots,p_N\}=\{1,2,\dots,N\}$, the equality (\ref{RALMRA:eqn6}) can be found in \cite{che2019randomized}. For each $n\geq 2$, we have
\begin{equation*}
\begin{aligned}
&\mathcal{A}\times_{p_1}\left({\bf Q}_{p_1}{\bf Q}_{p_1}^\top\right)\dots\times_{p_{n-1}}\left({\bf Q}_{p_{n-1}}{\bf Q}_{p_{n-1}}^\top\right)\\
&\quad-\mathcal{A}\times_{p_1}\left({\bf Q}_{p_1}{\bf Q}_{p_1}^\top\right)\dots\times_{p_{n-1}}\left({\bf Q}_{p_{n-1}}{\bf Q}_{p_{n-1}}^\top\right)\times_n\left({\bf Q}_{p_n}{\bf Q}_{p_n}^\top\right)\\
&\quad=\mathcal{A}\times_{p_1}\left({\bf Q}_{p_1}{\bf Q}_{p_1}^\top\right)\dots\times_{p_{n-1}}\left({\bf Q}_{p_{n-1}}{\bf Q}_{p_{n-1}}^\top\right)\times_n\left(\mathbf{I}_{I_{p_n}}-{\bf Q}_{p_n}{\bf Q}_{p_n}^\top\right)\\
&\quad=\left(\mathcal{A}\times_{p_1}{\bf Q}_{p_1}^\top\dots\times_{p_{n-1}}{\bf Q}_{p_{n-1}}^\top\times_n\left(\mathbf{I}_{I_{p_n}}-{\bf Q}_{p_n}{\bf Q}_{p_n}^\top\right)\right)\times_{p_1}{\bf Q}_{p_1}\dots\times_{p_{n-1}}{\bf Q}_{p_{n-1}},
\end{aligned}
\end{equation*}
which implies that
\begin{equation*}
\begin{aligned}
&\left\|\mathcal{A}\times_{p_1}\left({\bf Q}_{p_1}{\bf Q}_{p_1}^\top\right)\dots\times_{p_{n-1}}\left({\bf Q}_{p_{n-1}}{\bf Q}_{p_{n-1}}^\top\right)\times_n\left(\mathbf{I}_{I_{p_n}}-{\bf Q}_{p_n}{\bf Q}_{p_n}^\top\right)\right\|_F\\
&=\left\|\left(\mathcal{A}\times_{p_1}{\bf Q}_{p_1}^\top\dots\times_{p_{n-1}}{\bf Q}_{p_{n-1}}^\top\times_n\left(\mathbf{I}_{I_{p_n}}-{\bf Q}_{p_n}{\bf Q}_{p_n}^\top\right)\right)\times_{p_1}{\bf Q}_{p_1}\dots\times_{p_{n-1}}{\bf Q}_{p_{n-1}}\right\|_F\\
&\leq\left\|\left(\mathcal{A}\times_{p_1}{\bf Q}_{p_1}^\top\dots\times_{p_{n-1}}{\bf Q}_{p_{n-1}}^\top\times_n\left(\mathbf{I}_{I_{p_n}}-{\bf Q}_{p_n}{\bf Q}_{p_n}^\top\right)\right)\right\|_F.
\end{aligned}
\end{equation*}
Here the inequality holds by the fact $\|\mathbf{A}\mathbf{Q}\|_F\leq\|\mathbf{A}\|_F$ for $\mathbf{A}\in\mathbb{R}^{I\times J}$ and any orthonormal matrix $\mathbf{Q}\in\mathbb{R}^{J\times K}$ with $K\leq J$. Hence, we have
\begin{equation}\label{RALMRA:eqn7}
\begin{aligned}
&\left\|\mathcal{A}-\mathcal{A}\times_1\left({\bf Q}_1{\bf Q}_1^\top\right)\times_2\left({\bf Q}_2{\bf Q}_2^\top\right)\dots\times_N\left({\bf Q}_N{\bf Q}_N^\top\right)\right\|_F^2\\
&\leq\left\|\mathcal{A}\times_{p_1}\left(\mathbf{I}_{I_{p_1}}-{\bf Q}_{p_1}{\bf Q}_{p_1}^\top\right)\right\|_F^2
+\left\|\left(\mathcal{A}\times_{p_1}{\bf Q}_{p_1}^\top\right)\times_{p_2}\left(\mathbf{I}_{I_{p_2}}-{\bf Q}_{p_2}{\bf Q}_{p_2}^\top\right)\right\|_F^2\\
&+\dots+\left\|\left(\mathcal{A}\times_{p_1}{\bf Q}_{p_1}^\top\times_{p_2}\mathbf{Q}_{p_2}^\top\dots\times_{p_{N-1}}\mathbf{Q}_{p_{N-1}}^\top\right)\times_{p_N}\left(\mathbf{I}_{I_{p_N}}-{\bf Q}_{p_N}{\bf Q}_{p_N}^\top\right)\right\|_F^2.
\end{aligned}
\end{equation}

Another way to obtain an approximate solution $\{\mathbf{Q}_1,\mathbf{Q}_2,\dots,\mathbf{Q}_N\}$ for (\ref{RALMRA:eqn4}) is to solve the following $N$ subproblems with a given processing order $\{p_1,p_2,\dots,p_N\}\in\mathbb{S}_N$:
\begin{enumerate}
	\item[(a)] For given $\mathcal{A}\in\mathbb{R}^{I_1\times I_2\times \dots\times I_N}$ and $\mu_{p_1}<I_{p_1}$, to find an orthonormal matrix $\mathbf{Q}_{p_1}\in\mathbb{R}^{I_{p_1}\times \mu_{p_1}}$ such that
	\begin{equation*}
	\begin{aligned}
	\mathbf{Q}_{p_1}=&\ {\rm argmin}\left\|\mathcal{A}-\mathcal{A}\times_{p_n}\left({\bf Q}{\bf Q}^\top\right)\right\|_F^2,\\
	&\ {\rm s.t.}\quad \mathbf{Q}\in\mathbb{R}^{I_{p_1}\times \mu_{p_1}}\ \text{ is orthonormal};
	\end{aligned}
	\end{equation*}
	\item[(b)] For each $n\geq 2$, to find an orthonormal matrix $\mathbf{Q}_{p_n}\in\mathbb{C}^{I_{p_n}\times \mu_{p_n}}$ with $\mu_{p_n}< I_{p_n}$ such that
	\begin{equation*}
	\begin{aligned}
	\mathbf{Q}_n=&\ {\rm argmin}\left\|\mathcal{A}\times_{p_1}{\bf Q}_{p_1}^\top\dots\times_{p_{n-1}}{\bf Q}_{p_{n-1}}^\top\times_{p_n}\left(\mathbf{I}_{I_{p_n}}{\bf Q}{\bf Q}^\top\right)\right\|_F^2,\\
	&\ {\rm s.t.}\quad \mathbf{Q}\in\mathbb{R}^{I_{p_n}\times \mu_{p_n}}\ \text{ is orthonormal}.
	\end{aligned}
	\end{equation*}
\end{enumerate}

It is worthly noting that choosing the best processing order is is a very challenging problem for solving (\ref{RALMRA:eqn4}) efficiently, which is beyond the scrope of this paper. Here, we assume that the processing order is $\{1,2,\dots,N\}$ and then rewrite (\ref{RALMRA:eqn7}) as follows,
\begin{equation}\label{RALMRA:eqn23}
\begin{aligned}
&\left\|\mathcal{A}-\mathcal{A}\times_1\left({\bf Q}_1{\bf Q}_1^\top\right)\times_2\left({\bf Q}_2{\bf Q}_2^\top\right)\dots\times_N\left({\bf Q}_N{\bf Q}_N^\top\right)\right\|_F^2\\
&\leq\left\|\mathcal{A}\times_{1}\left(\mathbf{I}_{I_{1}}-{\bf Q}_{1}{\bf Q}_{1}^\top\right)\right\|_F^2
+\left\|\left(\mathcal{A}\times_{1}{\bf Q}_{1}^\top\right)\times_{2}\left(\mathbf{I}_{I_{2}}-{\bf Q}_{2}{\bf Q}_{2}^\top\right)\right\|_F^2\\
&+\dots+\left\|\left(\mathcal{A}\times_{1}{\bf Q}_{1}^\top\times_{2}\mathbf{Q}_{2}^\top\dots\times_{N-1}\mathbf{Q}_{N-1}^\top\right)\times_{N}\left(\mathbf{I}_{I_N}-{\bf Q}_{N}{\bf Q}_{N}^\top\right)\right\|_F^2.
\end{aligned}
\end{equation}
\subsection{Framework of the proposed algorithm}
\label{RALMRA:sec3:1}

According to the above description, we obtain Algorithm \ref{RALMRA:alg2} for estimating the first $\mu_n$ left singular vectors of $\mathbf{A}_{(n)}$. Comparison with Algorithm \ref{RALMRA:alg1}, in Algorithm \ref{RALMRA:alg2}, after obtaining $\mathbf{Q}_n$, we update $\mathcal{A}$ as $\mathcal{A}=\mathcal{A}\times_n\mathbf{Q}_n^\top$. Hence, when all the matrices $\mathbf{Q}_n$ are obtained by Algorithm \ref{RALMRA:alg1}, we need to compute the core tensor $\mathcal{G}$ as $\mathcal{G}=\mathcal{A}\times_1\mathbf{Q}_1^\top\times_2\mathbf{Q}_2^\top\dots\times_N\mathbf{Q}_N^\top$; and when all the matrices $\mathbf{Q}_n$ are obtained by Algorithm \ref{RALMRA:alg2}, the core tensor $\mathcal{G}$ is also obtained.  Another difference between Algorithms {\rm\ref{RALMRA:alg1}} and {\rm\ref{RALMRA:alg2}} is that Algorithm {\rm\ref{RALMRA:alg1}} can be easily implemented in a parallel computing environment.
\begin{algorithm}[htb]
\caption{Another algorithm for solving Problem \ref{RALMRA:pro1} via random projection and power scheme}
\label{RALMRA:alg2}
\begin{algorithmic}[1]
\STATEx {\bf Input}: A tensor $\mathcal{A}\in \mathbb{R}^{I_1\times
         I_2\times\dots\times I_N}$, the desired multilinear rank $\{\mu_1,\mu_2,\dots,\mu_N\}$, the oversampling parameter $K$ and an integer $q\geq 1$.
\STATEx {\bf Output}: $N$ orthonormal matrices $\mathbf{Q}_n$ and a core tensor
         $\mathcal{G}$ such that $\mathcal{A}\approx\mathcal{G}\times_1\mathbf{Q}_1\times_2\mathbf{Q}_2\dots\times_N\mathbf{Q}_N$.
\STATE For $n=1,2,\dots,N$
\STATE   Generate a standard Gaussian matrix $\mathbf{G}\in\mathbb{R}^{I_n \times         (\mu_n+K)}$.
\STATE   Form the mode-$n$ unfolding $\mathbf{A}_{(n)}$ of $\mathcal{A}$.
\STATE   Compute $\mathbf{C}_n=
         (\mathbf{A}_{(n)}\mathbf{A}_{(n)}^\top)^q\mathbf{G}_n$.
\STATE   Find an orthonormal matrix $\mathbf{Q}_n\in\mathbb{R}^{I_n\times \mu_n}$         such that there exists $\mathbf{S}_n\in\mathbb{R}^{\mu_n\times
         (\mu_n+K)}$ for which
	\begin{equation*}
	\|\mathbf{Q}_n\mathbf{S}_n-\mathbf{C}_n\|_F\leq\Delta_{\mu_n+1}
         (\mathbf{C}_n).
	\end{equation*}
\STATE   Compute $\mathcal{A}=\mathcal{A}\times_n\mathbf{Q}_n^\top$.
\STATE End for
\STATE   Set the core tensor $\mathcal{G}$ as $\mathcal{G}=\mathcal{A}$.
\end{algorithmic}
\end{algorithm}

%From the assumptions of Problem {\rm \ref{RALMRA:pro1}}, for some small $q$ and $K$, we can ensure that $2q(\mu_n+K)<I_n$.
To count the floating-point operations of Algorithm \ref{RALMRA:alg2}, we evaluate the complexity for each $n$:
\begin{enumerate}
	\item[(a)] to form a standard Gaussian matrix $\mathbf{G}_n$ requires $I(\mu+K)$ operations;
	\item[(b)] it needs $O(\mu^{n-1}I^{N-n+1})$ operations to form the matrix $\mathbf{A}_{(n)}$;
	\item[(c)] it demands $4I(\mu+K+T)(\mu+K)$ operations to compute the matrix $\mathbf{C}_n$;
	\item[(d)] computing $\mathbf{Q}_n$ requires $O(I(\mu+K)\mu)$ operations;
	\item[(e)] updating the tensor $\mathcal{A}$ requires $2I^{N-n+1}\mu^n$ operations.
\end{enumerate}
Then Algorithm \ref{RALMRA:alg2} requires
\begin{equation*}
\begin{aligned}
&N\cdot\left(I(\mu+K)+O(I\mu^2+IK\mu)\right)\\
&+\sum_{n=1}^N\left(O(\mu^{n-1}I^{N-n+1})+4qI^{N-n+1}\mu^{n-1}(\mu+K)+2I^{N-n+1}\mu^{n-1}\right)
\end{aligned}
\end{equation*}
operations for the tensor $\mathcal{A}$.

\subsection{Accelerating Algorithm \ref{RALMRA:alg2} with approximate matrix multiplication}

Similar to the relationship between Algorithms \ref{RALMRA:alg1} and \ref{RALMRA:alg4}, Algorithm \ref{RALMRA:alg5} is obtained by using Algorithm \ref{RALMRA:alg3} to compute an approximation of $\mathbf{C}_n$ in Algorithm \ref{RALMRA:alg2}.
\begin{algorithm}
	\caption{Accelerating Algorithm \ref{RALMRA:alg2} with approximating matrix multiplication}
	\label{RALMRA:alg5}
	\begin{algorithmic}[1]
		\STATEx {\bf Input}: A tensor $\mathcal{A}\in \mathbb{R}^{I_1\times I_2\times\dots\times I_N}$, the desired multilinear rank $\{\mu_1,\mu_2,\dots,\mu_N\}$, $N$ positive integers $\{T_1,T_2,\dots,T_N\}$, the oversampling parameters $K$, an integer $q\geq 1$, and $N$ probability distributions $\{p_{n,i}\}_{i=1}^{\mu_1\dots \mu_{n-1}I_{n+1}\dots I_N}$.
		\STATEx {\bf Output}: $N$ orthonormal matrices $\mathbf{Q}_n$ and a core tensor $\mathcal{G}$ such that
                 $\mathcal{A}\approx\mathcal{G}\times_1\mathbf{Q}_1\times_2\mathbf{Q}_2\dots\times_N\mathbf{Q}_N$.
		\STATE For $n=1,2,\dots,N$
		\STATE Generate $\mathbf{S}_n\in\mathbb{R}^{\mu_1\dots \mu_{n-1}I_{n+1}\dots I_N\times T_n}\sim{\rm RANDSAMPLE}(T_n,\mathbf{p}_n)$.
		\STATE Arrange $\mathbf{A}_{(n)}$ as the mode-$n$ unfolding of $\mathcal{A}$.
        \STATE Obtain $\mathbf{C}_n'=\mathbf{A}_{(n)}\mathbf{S}_n$.
		\STATE Form a standard Gaussian matrix $\mathbf{G}_n\in\mathbb{R}^{I_n \times (\mu_n+K)}$.
		\STATE Compute $\mathbf{C}_n=(\mathbf{C}_n'\mathbf{C}_n'^\top)^q\mathbf{G}_n$.
		\STATE Find an orthonormal matrix $\mathbf{Q}_n\in\mathbb{R}^{I_n\times \mu_n}$         such that there exists $\mathbf{S}_n\in\mathbb{R}^{\mu_n\times
         (\mu_n+K)}$ for which
	\begin{equation*}
	\|\mathbf{Q}_n\mathbf{S}_n-\mathbf{C}_n\|_F\leq\Delta_{\mu_n+1}
         (\mathbf{C}_n).
	\end{equation*}
		\STATE Compute $\mathcal{A}=\mathcal{A}\times_n\mathbf{Q}_n^\top$.
		\STATE End for
            \STATE Set the core tensor $\mathcal{G}$ as $\mathcal{G}=\mathcal{A}$.
	\end{algorithmic}
\end{algorithm}

\begin{remark}\label{RALMRA:remadd}
In Algorithms {\rm \ref{RALMRA:alg1}}, {\rm \ref{RALMRA:alg4}}, {\rm \ref{RALMRA:alg2}} and {\rm \ref{RALMRA:alg5}}, for each $n$, the matrix $\mathbf{Q}_n\in\mathbb{R}^{I_n\times \mu_n}$ can be obtained by this process: we first compute the SVD of $\mathbf{C}_n$ as $\mathbf{C}_n=\mathbf{U}_n\bm{\Sigma}_n\mathbf{V}_n^\top$, and we then set $\mathbf{Q}_n=\mathbf{U}_n(:,1:\mu_n)$, where $\mathbf{U}_n\in\mathbb{R}^{I_n\times (\mu_n+K)}$ is orthonormal, $\mathbf{V}_n\in\mathbb{R}^{(\mu_n+K)\times (\mu_n+K)}$ is orthogonal, and $\bm{\Sigma}_n\in\mathbb{R}^{(\mu_n+K)\times (\mu_n+K)}$ is diagonal and its diagonal entries are the singular values of $\mathbf{C}_n$.
\end{remark}

To count the complexity of Algorithm \ref{RALMRA:alg5},  we assume that $T_1\sim T_2\sim \dots\sim T_N\sim T$. For each $n$, we have
\begin{enumerate}
	\item[(a)] for the (nearly) optimal probabilities, to construct $\mathbf{C}_n'$ requires $O((I+I^{N-n}\mu^{n-1})T)$ operations, and for the uniform probabilities, it needs $O(IT)$ operations to obtain $\mathbf{C}_n'$;
	\item[(b)] to generate the matrix $\mathbf{G}_n\in\mathbb{R}^{I\times (\mu+K)}$ requires $I(\mu+K)$ operations;
	\item[(c)] it requires $4I(\mu+K+T)(\mu+K)$ operations to compute the matrix $\mathbf{C}_n$;
	\item[(d)] computing $\mathbf{Q}_n$ requires $O(I\mu^2+IK\mu)$ operations and updating the tensor $\mathcal{A}$ requires $2I^{N-n+1}\mu^n$ operations.
\end{enumerate}
Then, Algorithm \ref{RALMRA:alg5} requires
\begin{equation*}
N\cdot (O(IL)+I(\mu+K)+2IT+2qI^2(\mu+K)+O(I\mu^2+IK\mu))+2\cdot \sum_{n=1}^{N}I^{N-n+1}\mu^n
\end{equation*}
operations with the uniform probabilities, and
\begin{equation*}
\begin{aligned}
&N\cdot (I(\mu+K)+2IT+2qI^2(\mu+K)+O(I\mu^2+IK\mu)\\
&+ \sum_{n=1}^{N}\left(2I^{N-n+1}\mu^n+O(LI+LI^{N-n}\mu^{n-1})\right)
\end{aligned}
\end{equation*}
operations with the (nearly) optimal probabilities.
\begin{remark}
In both Algorithms {\rm \ref{RALMRA:alg4}} and {\rm \ref{RALMRA:alg5}}, choosing the right sampling distribution requires insight and ingenuity, which is out of scope of this work. Here, the sampling distribution used in these two algorithms is chosen as the uniform sampling distribution.
\end{remark}
\begin{figure}[htb]
	\setlength{\tabcolsep}{4pt}
	\renewcommand\arraystretch{1}
	\centering
	\begin{tabular}{c}
		\includegraphics[width=3.8in]{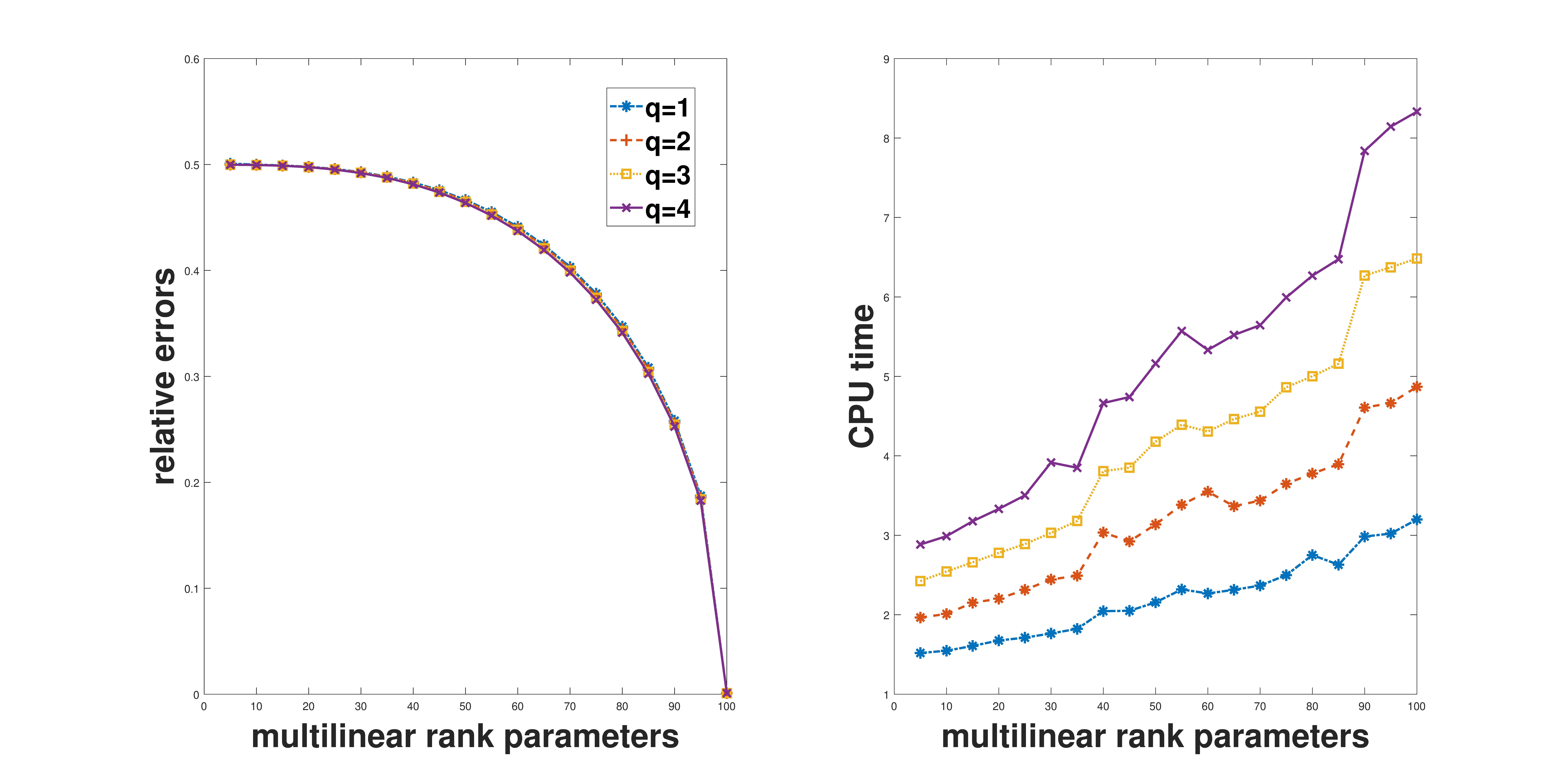}\\
		\includegraphics[width=3.8in]{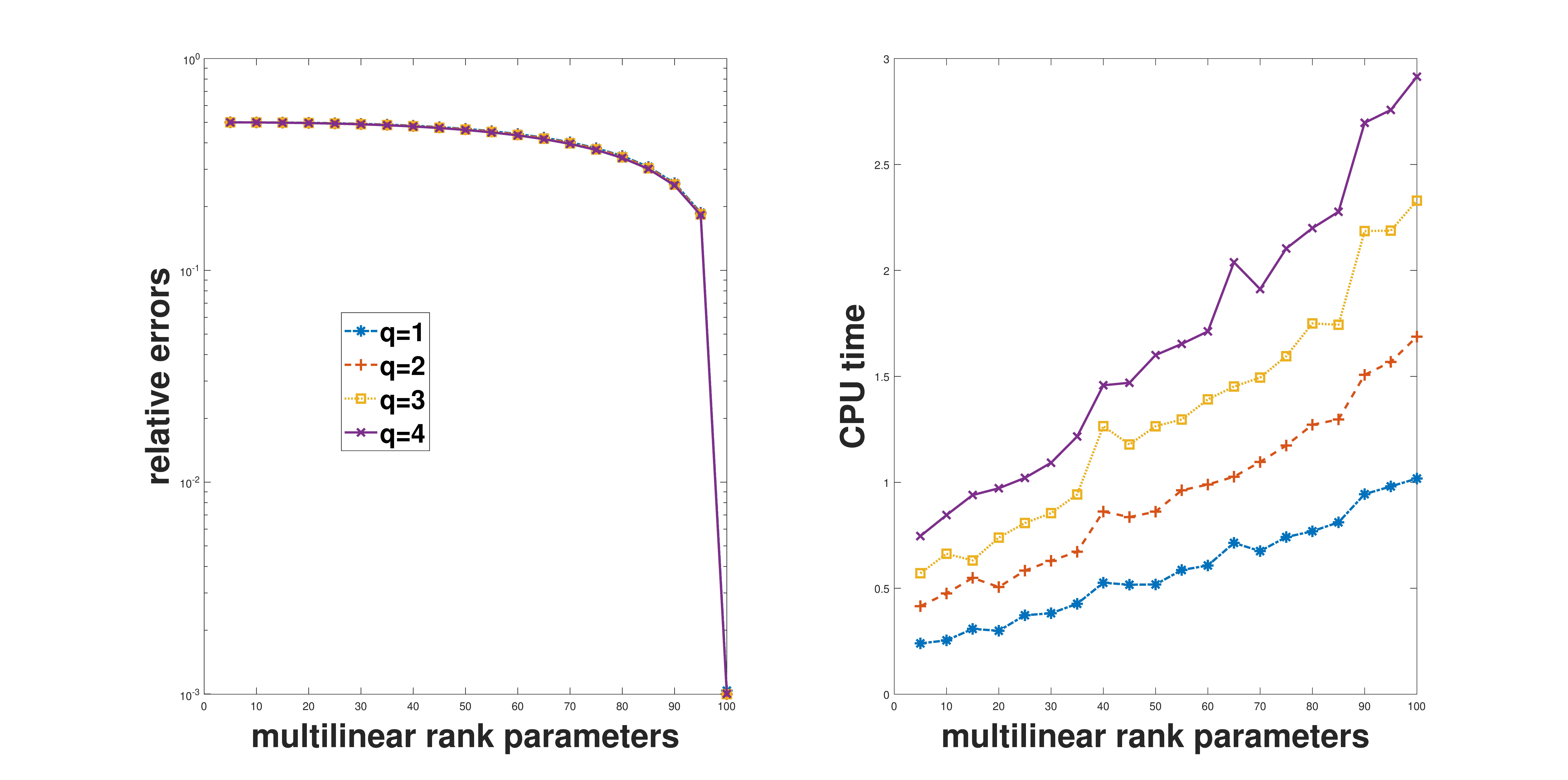}\\
	\end{tabular}
	\caption{Fixing $K=10$, results of applying Algorithms \ref{RALMRA:alg1} and \ref{RALMRA:alg2} with $q=1,2,3,4$ to $\mathcal{A}$: top for Algorithm \ref{RALMRA:alg1} and bottom for Algorithm \ref{RALMRA:alg2}.}\label{RALMRA:fig1}
\end{figure}

\subsection{Parameters selection}
\label{RALMRA:sec3:3}
In this section, we consider to choose all the parameters $K$, $q$ and $T_n$ in these proposed algorithms. We define the relative error (RE) is defined as
\begin{equation}\label{RALMRA:eqn8}
{\rm RE}=\|\mathcal{A}-\widehat{\mathcal{A}}\|_F/\|\mathcal{A}\|_F,
\end{equation}
with $\widehat{\mathcal{A}}=\mathcal{A}\times_1(\mathbf{Q}_1\mathbf{Q}_1^\top)\times_2(\mathbf{Q}_2\mathbf{Q}_2^\top)\dots\times_N(\mathbf{Q}_N\mathbf{Q}_N^\top)$, where for $N$ given positive integers $\mu_n< I_n$, an optimal solution $\{\mathbf{Q}_1,\mathbf{Q}_2,\dots,\mathbf{Q}_N\}$ for (\ref{RALMRA:eqn4}) is obtained from by applying the numerical algorithms to $\mathcal{A}\in\mathbb{R}^{I_1\times I_2\times \dots\times I_N}$.

We now illustrate the efficiencies of Algorithms \ref{RALMRA:alg1}, \ref{RALMRA:alg4}, \ref{RALMRA:alg2} and \ref{RALMRA:alg5} via a synthetic tensor $\mathcal{A}\in\mathbb{R}^{600\times 600\times 600}$ \cite{sun2020low-rank}, which is defined as
\begin{equation}\label{RALMRA:eqn9}
\mathcal{A}=\mathcal{B}+\frac{\gamma\|\mathcal{B}\|_F}{\sqrt{400^3}}\mathcal{E}
\end{equation}
with $\mathcal{B}=\mathcal{C}\times_1\mathbf{A}_1\times_2\mathbf{A}_2\times_3\mathbf{A}_3$, where the entries of $\mathcal{C}\in\mathbb{R}^{100\times 100\times 100}$ are drawn i.i.d. from the uniform distribution $U(0,1)$, the matrix $\mathbf{A}_n$ is an orthonormal basis for the column space of a standard Gaussian matrix $\mathbf{B}_n\in\mathbb{R}^{600\times 100}$, and the entries of $\mathcal{E}\in\mathbb{R}^{600\times 600\times 600}$ are an independent family of standard normal random variables. Here we set $\gamma=0.001$ and the desired multilinear rank $\{\mu,\mu,\mu\}$ as $\mu=5,10,\dots,100$.

When we fix $K=10$, by applying Algorithms \ref{RALMRA:alg1} and \ref{RALMRA:alg2} with different $q$ to the tensor $\mathcal{A}$, the values of RE and CPU time are shown in Figure \ref{RALMRA:fig1}, and when we fix $q=1$, by applying Algorithms \ref{RALMRA:alg1} and \ref{RALMRA:alg2} with different $K$ to the tensor $\mathcal{A}$, the results are shown in Figure \ref{RALMRA:fig2}. It implies from Figures \ref{RALMRA:fig1} and \ref{RALMRA:fig2} that in terms of the relative error, for $K=10$, Algorithms \ref{RALMRA:alg1} and \ref{RALMRA:alg2} with different $q$ are comparable, and for $q=1$, Algorithms \ref{RALMRA:alg1} and \ref{RALMRA:alg2} with different $K$ are similar.

\begin{figure}[htb]
	\centering
	\begin{tabular}{c}
		\includegraphics[width=3.8in]{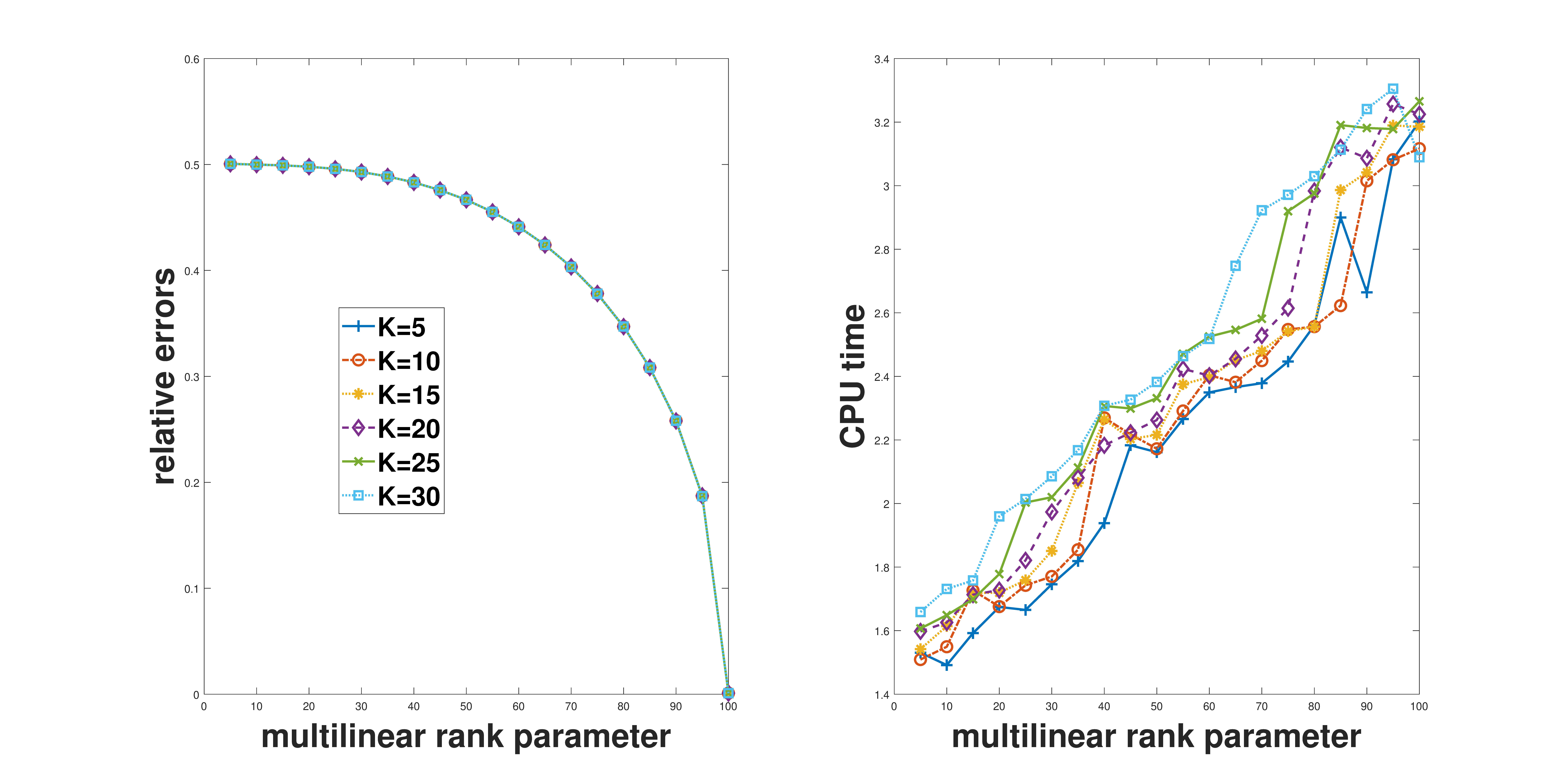}\\
		\includegraphics[width=3.8in]{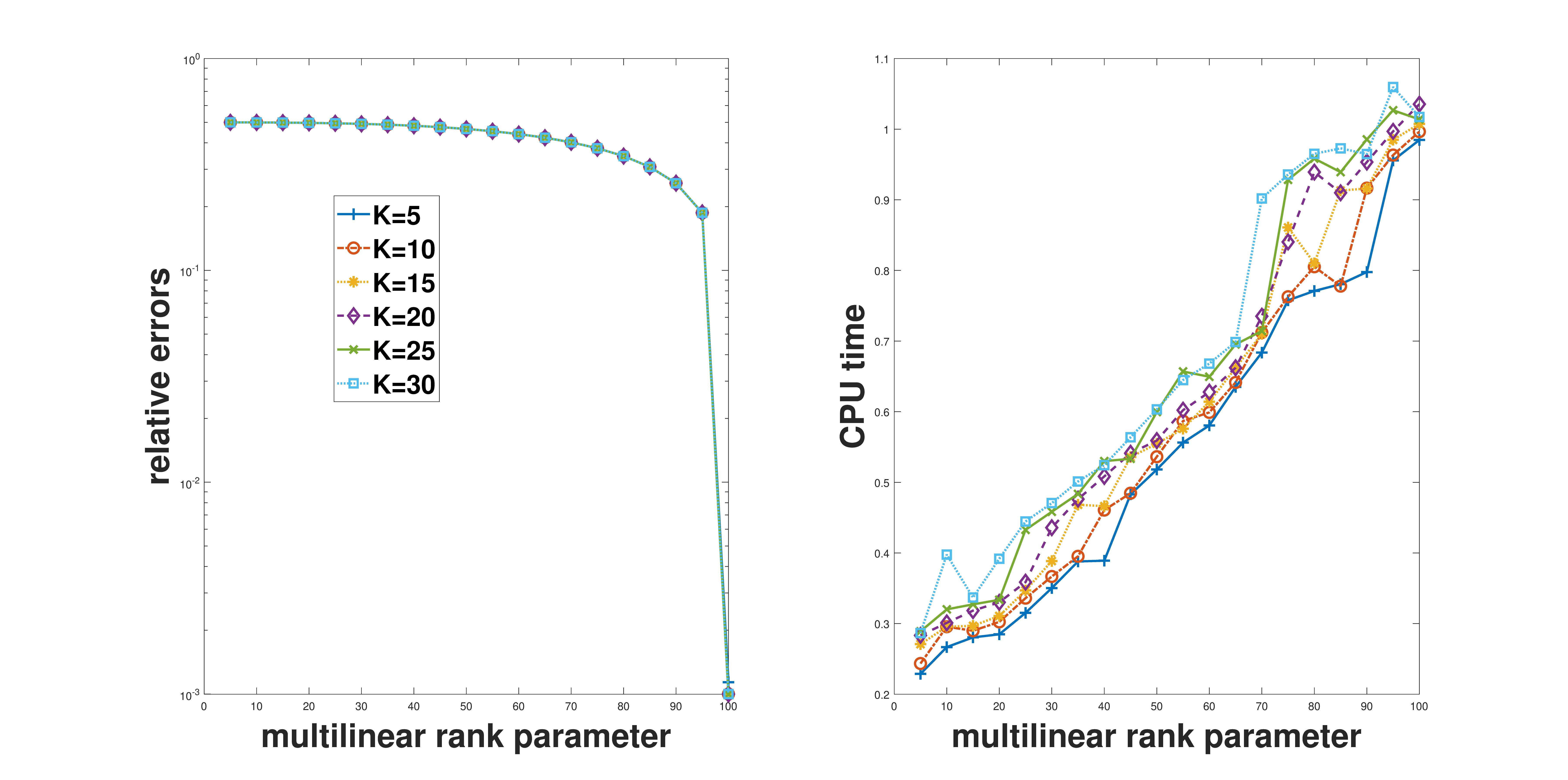}\\
	\end{tabular}
	\caption{Fixing $q=1$, results of applying applying Algorithms \ref{RALMRA:alg1} and \ref{RALMRA:alg2} with $K=5,10,15,20,25,30$ to $\mathcal{A}$: top for Algorithm \ref{RALMRA:alg1} and bottom for Algorithm \ref{RALMRA:alg2}.}\label{RALMRA:fig2}
\end{figure}

In practice, for each $n$, we set $T_n={\rm ceil}(\alpha I_1\dots I_{n-1}I_{n+1}\dots I_N)$ for Algorithm \ref{RALMRA:alg4} and $T_n={\rm ceil}(\alpha \mu_1\dots \mu_{n-1}I_{n+1}\dots I_N)$ for Algorithm \ref{RALMRA:alg5}, where $0<\alpha<1$ is a given number and in MAMLAB, ${\rm ceil}(x)$ rounds the elements of $x\in\mathbb{R}$ to the nearest integers towards infinity. When we fix $q=1$ and $K=10$, by applying Algorithms \ref{RALMRA:alg4} and \ref{RALMRA:alg5}
with different $\alpha$ to the tensor $\mathcal{A}$, the results are shown in Figure \ref{RALMRA:fig3}. In terms of the relative error, Algorithm \ref{RALMRA:alg4} with different $\alpha$ are similar, and Algorithm \ref{RALMRA:alg5} with $\alpha\geq 0.15$ are better than Algorithm \ref{RALMRA:alg5} with $\alpha< 0.15$. Meanwhile, for each $\mu$, as $\alpha$ increases, the values of CPU time for Algorithms \ref{RALMRA:alg4} and \ref{RALMRA:alg5} increases.

\begin{figure}[htb]
	\setlength{\tabcolsep}{4pt}
	\renewcommand\arraystretch{1}
	\centering
	\begin{tabular}{c}
		\includegraphics[width=3.8in]{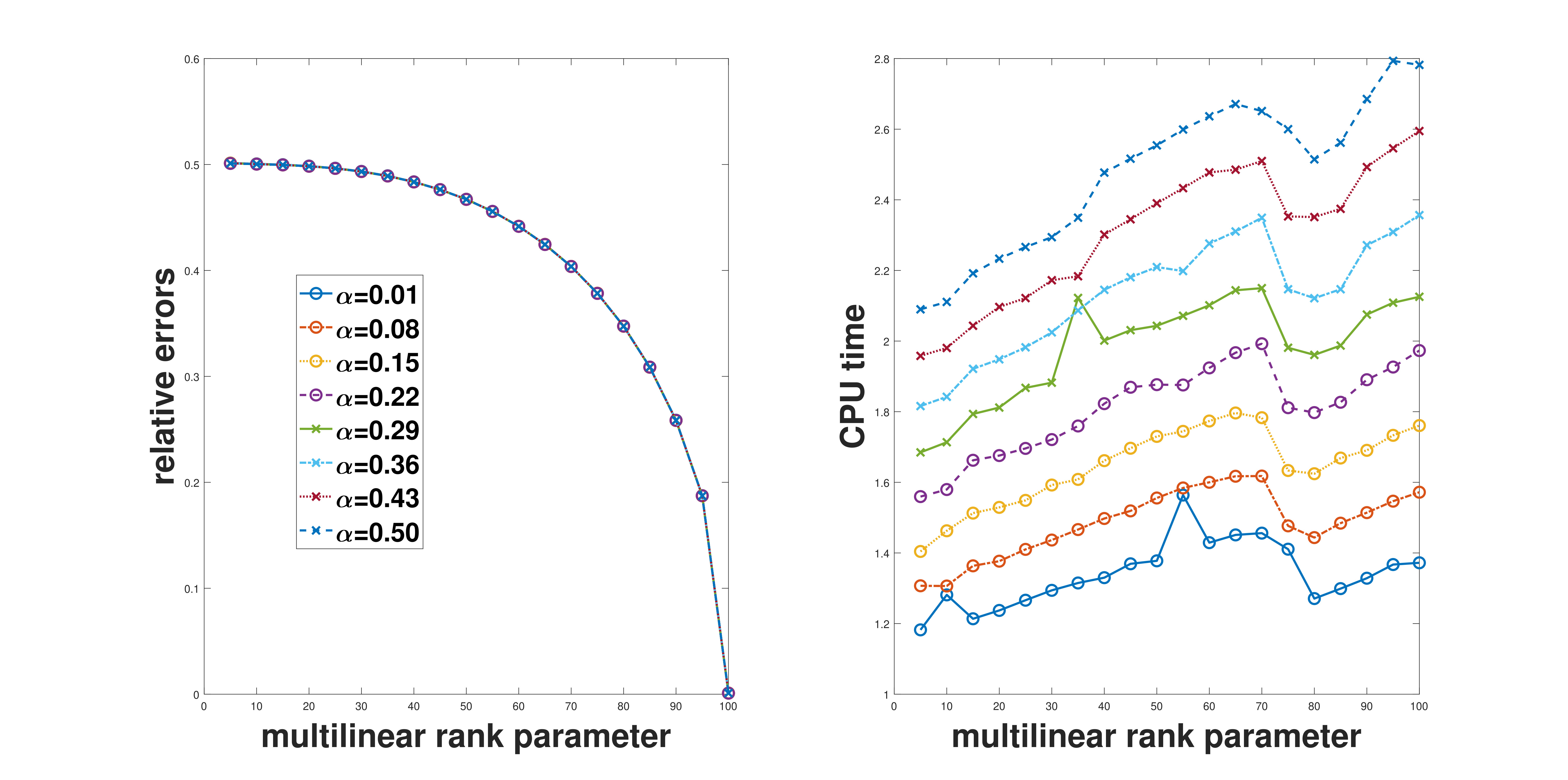}\\
		\includegraphics[width=3.8in]{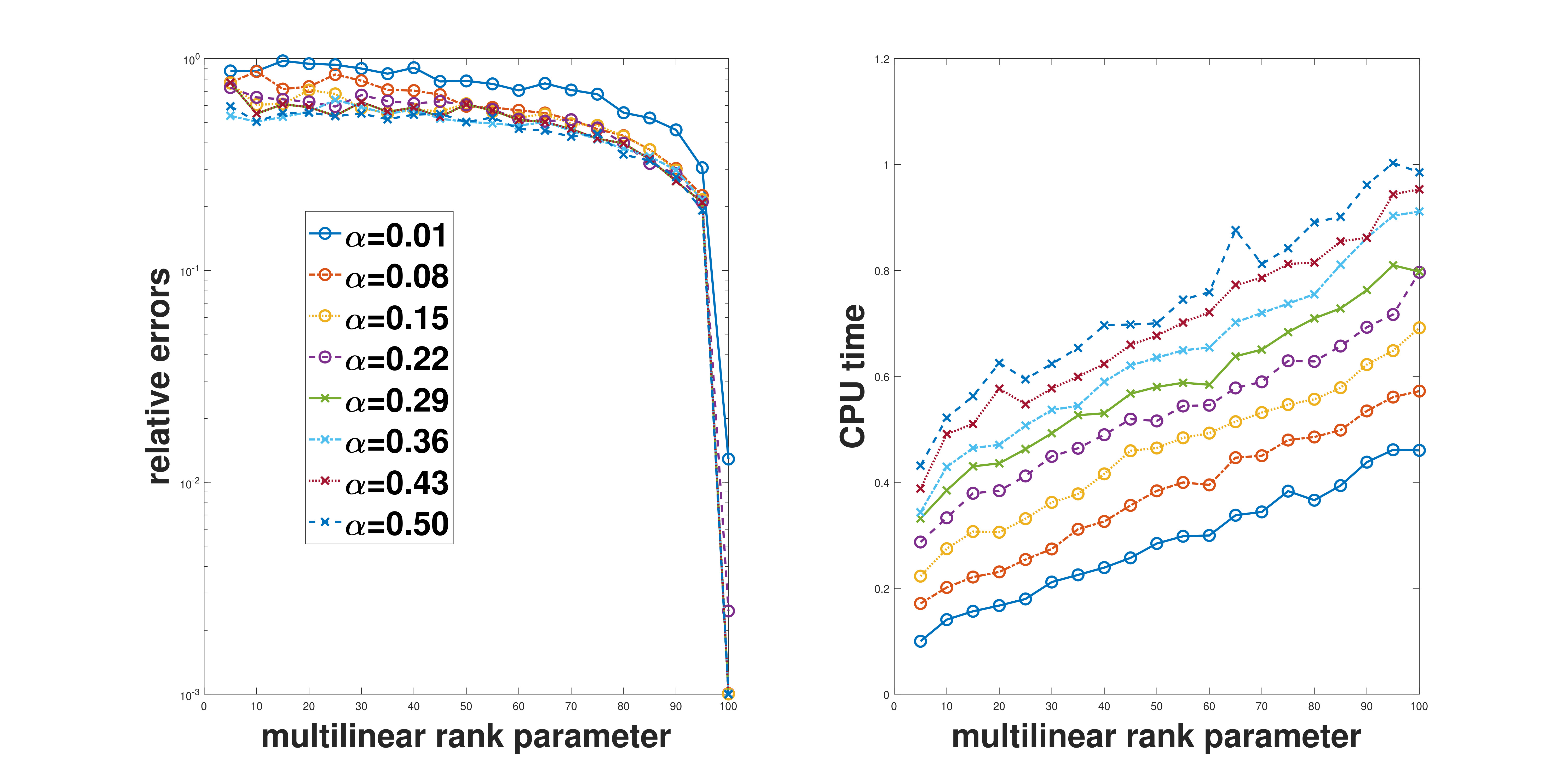}\\
	\end{tabular}
	\caption{Fixing $K=10$ and $q=1$, results of applying Algorithms \ref{RALMRA:alg4} and \ref{RALMRA:alg5} with $\alpha=0.01,0.08,\dots,0.5$ to $\mathcal{A}$: top for Algorithm \ref{RALMRA:alg4} and bottom for Algorithm \ref{RALMRA:alg5}.}\label{RALMRA:fig3}
\end{figure}

Finally, when we fix $q=1$, $K=10$ and $\alpha=0.2$, Figure \ref{RALMRA:fig4} illustrates the results by applying Algorithms \ref{RALMRA:alg1}, \ref{RALMRA:alg4}, \ref{RALMRA:alg3} and \ref{RALMRA:alg5} to $\mathcal{A}$. In terms of the relative error, these algorithms are comparable, and Algorithm \ref{RALMRA:alg2} is slight worse than other algorithms; in terms of CPU time, Algorithm \ref{RALMRA:alg5} is the fastest one, and Algorithms \ref{RALMRA:alg2} and \ref{RALMRA:alg5} are faster than Algorithms \ref{RALMRA:alg1} and \ref{RALMRA:alg4}.

\begin{figure}[htb]
	\centering
	\includegraphics[width=3.8in]{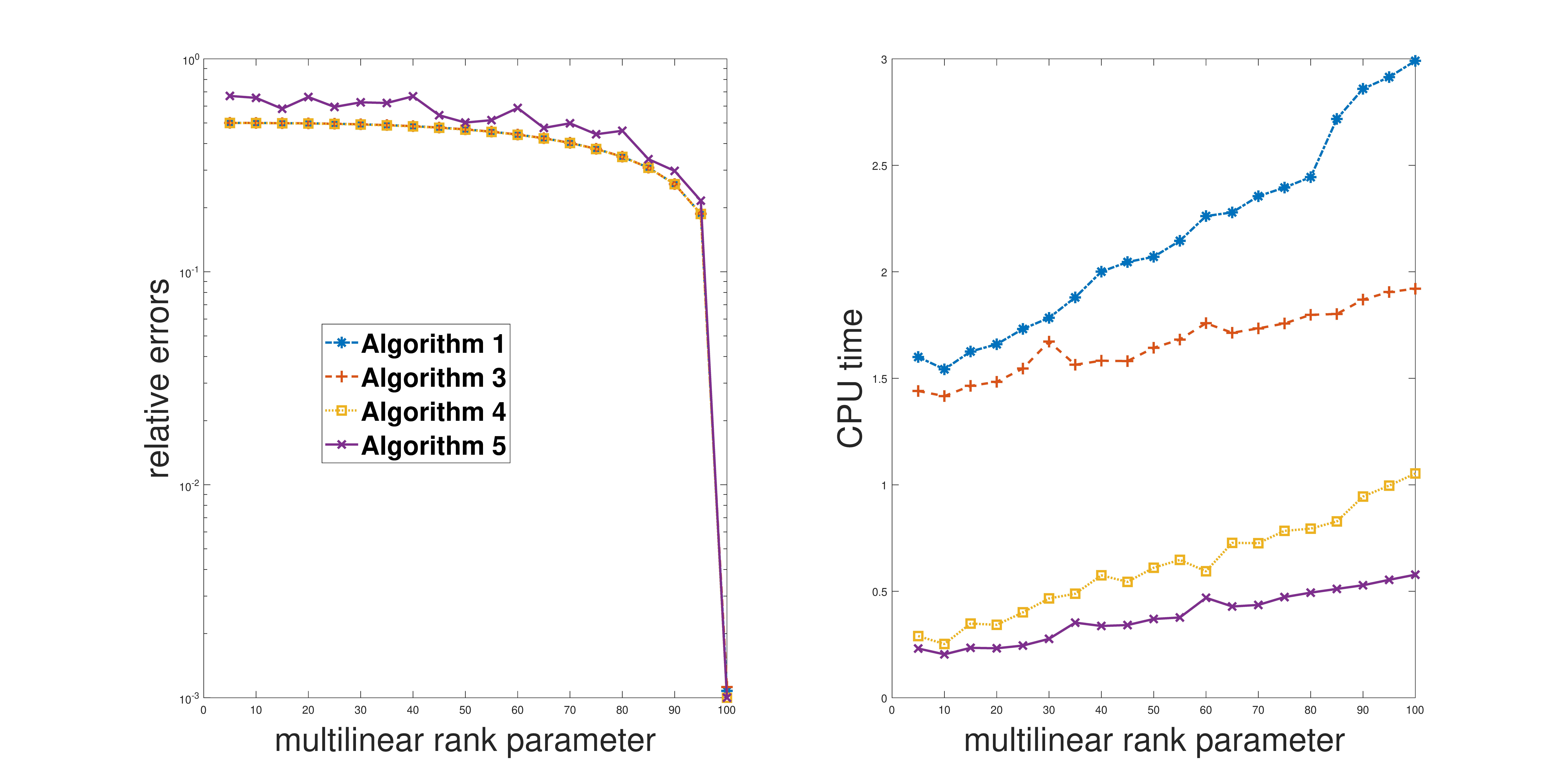}\\
	\caption{Fixing $q=1$, $K=10$ and $\alpha=0.2$, results of applying Algorithms \ref{RALMRA:alg1}, \ref{RALMRA:alg4}, \ref{RALMRA:alg2} and \ref{RALMRA:alg5} to $\mathcal{A}$.}\label{RALMRA:fig4}
\end{figure}

In the rest, we set $K=10$ and $q=1$ in Algorithms {\rm\ref{RALMRA:alg1}}, {\rm\ref{RALMRA:alg4}}, {\rm\ref{RALMRA:alg2}} and {\rm\ref{RALMRA:alg5}}. We also set $\alpha=0.2$ in Algorithms {\rm\ref{RALMRA:alg4}} and {\rm\ref{RALMRA:alg5}}. In Section {\rm \ref{RALMRA:sec5}}, we compare Algorithms {\rm\ref{RALMRA:alg2}} and {\rm\ref{RALMRA:alg5}} with the exisiting deterministic and randomized algorithms via several test tensors from both synthetic and real datasets.

\section{Theoretical analysis}
\label{RALMRA:sec5}
In this section, we will estimate the upper bounds of $\|\mathcal{A}-\widehat{\mathcal{A}}\|_F$, with $\widehat{\mathcal{A}}
=\mathcal{A}\times_1(\mathbf{Q}_1\mathbf{Q}_1^\top)\times_2(\mathbf{Q}_2\mathbf{Q}_2^\top)
\dots\times_N(\mathbf{Q}_N\mathbf{Q}_N^\top)$, where $\{\mathbf{Q}_1,\mathbf{Q}_2,\dots,\mathbf{Q}_N\}$ is obtained by applying the proposed algorithms to $\mathcal{A}\in\mathbb{R}^{I_1\times I_2\times \dots\times I_N}$ with a given multilinear rank $\{\mu_1,\mu_2,\dots,\mu_N\}$. Without loss of generality, for each $n$, we let $I_n'=\min\{I_n,I_1\dots I_{n-1}I_{n+1}\dots I_N\}$, $I_n''=\min\{I_n,T_n\}$, $I_n'''=\min\{I_n,\mu_1\dots \mu_{n-1}I_{n+1}\dots I_N\}$ and $I_n''''=\min\{\mu_1\dots I_{n-1}I_{n+1}\dots I_N,T_n\}$.

For the cases of Algorithms \ref{RALMRA:alg1} and \ref{RALMRA:alg4}, it follows from (\ref{RALMRA:eqn5}) that
\begin{equation}\label{RALMRA:eqn25}
\|\mathcal{A}-\widehat{\mathcal{A}}\|_F\leq\sum_{n=1}^N\|\mathcal{A}-\mathcal{A}\times_n(\mathbf{Q}_n\mathbf{Q}_n^\top)\|_F.
\end{equation}
Then, for each $n$, we need to give an upper bound for $\|\mathcal{A}-\mathcal{A}\times_n(\mathbf{Q}_n\mathbf{Q}_n^\top)\|_F$.

Meanwhile, for the cases of Algorithms \ref{RALMRA:alg2} and \ref{RALMRA:alg5}, it follows from (\ref{RALMRA:eqn23}) that
\begin{equation}\label{RALMRA:eqn26}
\|\mathcal{A}-\widehat{\mathcal{A}}\|_F\leq\sum_{n=1}^N\|(\mathcal{A}\times_{1}{\bf Q}_{1}^\top\times_{2}\mathbf{Q}_{2}^\top\dots\times_{n-1}\mathbf{Q}_{n-1}^\top)\times_{n}(\mathbf{I}_{I_n}-{\bf Q}_{n}{\bf Q}_{n}^\top)\|_F.
\end{equation}
Then for each $n$, we will consider the upper bound for $\|(\mathcal{A}\times_{1}{\bf Q}_{1}^\top\times_{2}\mathbf{Q}_{2}^\top\dots\times_{n-1}\mathbf{Q}_{n-1}^\top)\times_{n}(\mathbf{I}_{I_n}-{\bf Q}_{n}{\bf Q}_{n}^\top)\|_F$.
\subsection{Some necessary lemmas}

A matrix $\mathbf{G}\in\mathbb{R}^{I\times J}$ is a standard Gaussian matrix \cite{tropp2017practical} if the entries form an independent family of standard normal random variables, that is, independent and identically distributed (i.i.d.) Gaussian random variables of zero mean and unit variance.
\begin{lemma}{{\bf (\cite{coakley2011fast})}}
	\label{RALMRA:lem1}
	Let $\mathbf{G}\in\mathbb{R}^{L\times I}$ be a standard Gaussian matrix with $L<I$. For $\gamma>1$, if
	\begin{equation}\label{RALMRA:eqn2}
	1-\frac{1}{4(\gamma^2-1)\sqrt{\pi I\gamma^2}}\left(\frac{2\gamma^2}{e^{\gamma^2-1}}\right)^2
	\end{equation}
	is nonnegative, then, the largest singular value of $\mathbf{G}$ is at most $\sqrt{2I}\gamma$ with probability not less than the amount in {\rm (\ref{RALMRA:eqn2})}.
\end{lemma}
\begin{lemma}{{\bf (\cite{chen2005condition})}}
	\label{RALMRA:lem2}
	Let $I$ and $L$ be positive integers with $L\leq I$ and
	\begin{equation}\label{RALMRA:eqn3}
	1-\frac{1}{\sqrt{2\pi(I-L+1)}}\left(\frac{e}{(I-L+1)\beta}\right)^{I-L+1}
	\end{equation}
	being nonnegative, where $\beta>1$. Suppose that $\mathbf{G}\in\mathbb{R}^{L\times J}$ is a standard Gaussian matrix. Then, the smallest singular value of $\mathbf{G}$ is at least $1/(\sqrt{I}\beta)$ with probability not less than the amount in {\rm (\ref{RALMRA:eqn3})}.
\end{lemma}

A standard Gaussian matrix is a special case of sub-Gaussian matrices \cite{litvak2005smallest,rudelson2009smallest}, whose entries are sub-Gaussian variables. A sub-Gaussian variable is an important class of random variables that have strong tail decay properties.
\begin{definition}{{\bf (\cite{shabat2016randomized})}}
	\label{SUB:def4}
	A real valued random variable $X$ is called a sub-Gaussian random variable if there exist $b>0$ such that for all $t>0$ we have $\mathbb{E}(e^{tX})\leq e^{b^2t^2/2}$. A random variable $X$ is centered if $\mathbb{E}(X)=0$.
\end{definition}

\begin{definition}
	Assume that $s\geq1$, $a_1>0$ and $a_2>0$. The set $\mathbb{A}(s,a_1,a_2,K,I)$ consists of all $K\times I$ random matrices $\mathbf{G}$ whose entries are the centered  independent identically distributed real valued random variables satisfying the following conditions: {\rm (a)} moments: $\mathbb{E}(\lvert g_{ij}\rvert^3)\leq s^3$; {\rm (b)} norm: $\mathbb{P}(\|\mathbf{G}\|_2>a_1\sqrt{I})\leq e^{-a_2I}$; {\rm (c)} variance: $\mathbb{E}(\lvert g_{ij}\rvert^2)\leq 1$.
\end{definition}

It is proven in \cite{litvak2005smallest} that if $\mathbf{A}$ is sub-Gaussian, then $\mathbf{G}\in\mathbb{A}(s,a_1,a_2,I,J)$. For a Gaussian matrix with zero mean and unit variance, we have $s=(4/\sqrt{2\pi})^{1/3}$.
\begin{theorem}{{\bf (\cite{litvak2005smallest})}}
	\label{RALMRA:thm1}
	Suppose that $\mathbf{G}\in \mathbb{R}^{K\times I}$ is sub-Gaussian with $K\leq I$, $s\geq1$ and $a_2>0$. Then $\mathbb{P}(\|\mathbf{G}\|_2>a_1\sqrt{I})\leq e^{-a_2I}$,
	where $a_1=6s\sqrt{a_2+4}$.
\end{theorem}
\begin{theorem}{{\bf (\cite{litvak2005smallest})}}
	\label{RALMRA:thm2}
	Let $s\geq1$, $a_1>0$ and $a_2>0$. Suppose that $\mathbf{G}\in\mathbb{A}(s,a_1,a_2,K,I)$ with $I>(1+1/\ln(K))K$. Then, there exist positive constants $c_1$ and $c_2$ such that
	\begin{equation*}
	\mathbb{P}(\sigma_K(\mathbf{G})\leq c_1\sqrt{I})\leq e^{-I}+e^{-c''I/(2s^6)}+e^{-a_2I}\leq e^{-c_2I}.
	\end{equation*}
\end{theorem}

\begin{lemma}{{\bf (\cite{che2021randomized})}}
	\label{RALMRA:lem3}
	Let $K$, $I$ and $J$ be positive integers with $K<J\leq I$. For a given matrix $\mathbf{A}\in\mathbb{R}^{I\times J}$, there exist an orthonormal matrix $\mathbf{Q}\in\mathbb{R}^{I\times K}$ and $\mathbf{S}\in\mathbb{R}^{K\times J}$ such that
	\begin{equation*}
	\|\mathbf{Q}\mathbf{S}-\mathbf{A}\|_F\leq\Delta_{K+1}(\mathbf{A})
	=\left(\sum_{i=K+1}^J\sigma_i(\mathbf{A})^2\right)^{1/2},
	\end{equation*}
	where $\sigma_i(\mathbf{A})$ is the $i$th singular value of $\mathbf{A}$ and $j=1,2,\dots,J$.
\end{lemma}
\begin{lemma}{{\bf (\cite{che2021randomized})}}
	\label{RALMRA:lem4}
	Let $\mathbf{A}\in\mathbb{R}^{I\times J}$ and $\mathbf{B}\in\mathbb{R}^{J\times S}$ with $S\leq\min\{I,J\}$. Then for all $k=1,2,\dots,\min\{I,J,S\}-1,\min\{I,J,S\}$, we have
	\begin{equation*}
	\sum_{i=k}^{S}\sigma_i(\mathbf{A}\mathbf{B})^2\leq
	\|\mathbf{B}\|_2^2\sum_{i=k}^{\min\{I,J\}}\sigma_i(\mathbf{A})^2.
	\end{equation*}
\end{lemma}
\begin{lemma}
	\label{RALMRA:lem5}
	Suppose that $\mathbf{A}\in\mathbb{R}^{I_1\times I_2}$, $\mathbf{B}\in\mathbb{R}^{I_2\times I_3}$, $K$ is a positive integer such that $1\leq K\leq I_2$, and $\{p_i\}_{i=1}^{I_2}$ are such that $p_i\geq 0$, $\sum_{i=1}^{I_2}p_i=1$, and for some positive constant $\beta\leq1$,
	\begin{equation*}
	p_i\geq\beta\frac{\|\mathbf{A}(i,:)\|_2\|\mathbf{B}(:,i)\|_2}{\sum_{j=1}^{I_2}\|\mathbf{A}(j,:)\|_2\|\mathbf{B}(:,j)\|_2}.
	\end{equation*}
	Construct $\mathbf{C}$ and $\mathbf{R}$ with Algorithm {\rm \ref{RALMRA:alg3}} and let $\mathbf{C}\mathbf{R}$ be an approximation to $\mathbf{AB}$. Let $\delta\in(0,1)$ and $\eta=1+\sqrt{(8/\delta)\log(1/\delta)}$, then with probability at least $1-\delta$,
	\begin{equation*}
	\|\mathbf{A}\mathbf{B}-\mathbf{C}\mathbf{R}\|_F\leq \frac{\eta}{\sqrt{\beta K}}\|\mathbf{A}\|_F\|\mathbf{B}\|_F.
	\end{equation*}
	\end{lemma}
\begin{lemma}
	\label{RALMRA:lem6}
	Suppose that $\mathbf{A}\in\mathbb{R}^{I_1\times I_2}$, $\mathbf{B}\in\mathbb{R}^{I_2\times I_3}$, $K$ is a positive integer such that $1\leq K\leq I_2$, and $\{p_i\}_{i=1}^{I_2}$ are such that $p_i=1/I_2$. Construct $\mathbf{C}$ and $\mathbf{R}$ with Algorithm {\rm \ref{RALMRA:alg3}} and let $\mathbf{C}\mathbf{R}$ be an approximation to $\mathbf{AB}$. Let $\delta\in(0,1)$ and $\gamma=1+\frac{I_2}{\sqrt{K}}\sqrt{8\log(1/\delta)}\max_i\|\mathbf{A}(:,i)\|_2\|\mathbf{B}(i,:)\|_2$, then with probability at least $1-\delta$,
	\begin{equation*}
	\|\mathbf{A}\mathbf{B}-\mathbf{C}\mathbf{R}\|_F \leq\sqrt{\frac{I_2}{K}}\left(\sum_{i=1}^{I_2}\|\mathbf{A}(:,i)\|_2^2\|\mathbf{B}(i,:)\|_2^2\right)^{1/2}+\gamma.
	\end{equation*}
\end{lemma}
\subsection{Analyzing Algorithm \ref{RALMRA:alg1}}

When the matrix $\mathbf{Q}_n$ is obtained by Algorithm \ref{RALMRA:alg1} with given $n$, the matrix $\mathbf{Q}_n\mathbf{Q}_n^\top\mathbf{A}_{(n)}$ is a good approximation to the matrix $\mathbf{A}_{(n)}$, provided that there exist matrices $\mathbf{G}_n\in\mathbb{R}^{I_n\times (\mu_n+K)}$ and $\mathbf{R}_n\in\mathbb{R}^{\mu_n\times (\mu_n+K)}$ such that: (a) $\mathbf{Q}_n$ is orthonormal; (b) $\mathbf{Q}_n\mathbf{R}_n$ is a good approximation of $(\mathbf{A}_{(n)}\mathbf{A}_{(n)}^\top)^q\mathbf{G}_n$; and (c) there exists a matrix $\mathbf{F}\in\mathbb{R}^{(\mu_n+K)\times \mu_1\dots \mu_{n-1}I_{n+1}\dots I_{N}}$ such that $\|\mathbf{F}\|_2$ is not too large and $(\mathbf{A}_{(n)}\mathbf{A}_{(n)}^\top)^q\mathbf{G}_n\mathbf{F}_n$ is a good approximation of $\mathbf{A}_{(n)}$.
\begin{lemma}
	\label{RALMRA:lem7}
	For a given $n$, suppose that $\mathbf{A}_{(n)}\in\mathbb{R}^{I_n\times I_1\dots I_{n-1}I_{n+1}\dots I_{N}}$, $\mathbf{Q}_n\in\mathbb{R}^{I_n\times \mu_n}$ is orthonormal, $\mathbf{R}_n\in\mathbb{R}^{\mu_n\times (\mu_n+K)}$, $\mathbf{F}\in\mathbb{R}^{(\mu_n+K)\times I_1\dots I_{n-1}I_{n+1}\dots I_{N}}$ and $\mathbf{G}_n\in\mathbb{R}^{I_n\times (\mu_n+K)}$ with $\mu_n< \mu_n+K< I_n'$. For a given positive integer $q>0$, we have
	\begin{equation}
	\label{RALMRA:eqn10}
	\begin{aligned}
	\|\mathbf{A}_{(n)}-\mathbf{Q}_n\mathbf{Q}_n^\top\mathbf{A}_{(n)}\|_F^2
	&\leq 2\|(\mathbf{A}_{(n)}\mathbf{A}_{(n)}^\top)^q\mathbf{G}_n\mathbf{F}_n-\mathbf{A}_{(n)}\|_F^2\\
&\quad	+2\|\mathbf{F}\|_2^2\|\mathbf{Q}_n\mathbf{R}_n-(\mathbf{A}_{(n)}\mathbf{A}_{(n)}^\top)^q\mathbf{G}_n\|_F^2.
	\end{aligned}
	\end{equation}
\end{lemma}
\begin{remark}
	When the Frobenius norm is replaced by the spectral norm, Lemma {\rm \ref{RALMRA:lem7}} is similar to Lemma 4.3 in {\rm\cite{martinsson2011randomized}}, Lemma 4.4 in {\rm \cite{shabat2016randomized}} and Lemma 3.1 in {\rm\cite{rokhlin2010randomized}}. The proof is similar to that of Lemma 4.5 in {\rm \cite{che2020computation}} and Lemma 5.4 in {\rm\cite{che2021randomized}}.
\end{remark}
\begin{proof}
	The proof is straightforward, but tedious, as follows. By using the triangular inequality, we have
	\begin{equation}\label{RALMRA:eqn11}
	\begin{aligned}
	&\|\mathbf{A}_{(n)}-\mathbf{Q}_n\mathbf{Q}_n^\top\mathbf{A}_{(n)}\|_F^2\\
 &\leq\|\mathbf{A}_{(n)}-(\mathbf{A}_{(n)}\mathbf{A}_{(n)}^\top)^q\mathbf{G}_n\mathbf{F}_n\|_F^2\\
	&+\|(\mathbf{A}_{(n)}\mathbf{A}_{(n)}^\top)^q\mathbf{G}_n\mathbf{F}_n
	-\mathbf{Q}_n\mathbf{Q}_n^\top(\mathbf{A}_{(n)}\mathbf{A}_{(n)}^\top)^q\mathbf{G}_n\mathbf{F}_n\|_F^2\\
	&+\|\mathbf{Q}_n\mathbf{Q}_n^\top(\mathbf{A}_{(n)}\mathbf{A}_{(n)}^\top)^q\mathbf{G}_n\mathbf{F}_n
	-\mathbf{Q}_n\mathbf{Q}_n^\top\mathbf{A}_{(n)}\|_F^2.
	\end{aligned}
	\end{equation}
	Note that we have
	\begin{equation*}
	\begin{aligned}
&\|\mathbf{Q}_n\mathbf{Q}_n^\top(\mathbf{A}_{(n)}\mathbf{A}_{(n)}^\top)^q\mathbf{G}_n\mathbf{F}_n
	-\mathbf{Q}_n\mathbf{Q}_n^\top\mathbf{A}_{(n)}\|_F^2\\
 &\leq
	\|(\mathbf{A}_{(n)}\mathbf{A}_{(n)}^\top)^q\mathbf{G}_n\mathbf{F}_n-\mathbf{A}_{(n)}\|_F^2\|\mathbf{Q}_n\mathbf{Q}_n^\top\|_2^2,
	\end{aligned}
	\end{equation*}
	which implies that
	\begin{equation}\label{RALMRA:eqn12}
	\|\mathbf{Q}_n\mathbf{Q}_n^\top(\mathbf{A}_{(n)}\mathbf{A}_{(n)}^\top)^q\mathbf{G}_n\mathbf{F}_n
	-\mathbf{Q}_n\mathbf{Q}_n^\top\mathbf{A}_{(n)}\|_F^2
	\leq\|(\mathbf{A}_{(n)}\mathbf{A}_{(n)}^\top)^q\mathbf{G}_n\mathbf{F}_n
	-\mathbf{A}_{(n)}\|_F^2.
	\end{equation}
	For the second term in the right-hand side of (\ref{RALMRA:eqn11}), we have
	\begin{equation*}
	\begin{aligned}
	&\|(\mathbf{A}_{(n)}\mathbf{A}_{(n)}^\top)^q\mathbf{G}_n\mathbf{F}_n
	-\mathbf{Q}_n\mathbf{Q}_n^\top(\mathbf{A}_{(n)}\mathbf{A}_{(n)}^\top)^q\mathbf{G}_n\mathbf{F}_n\|_F^2\\
	&\quad\quad\leq\|(\mathbf{A}_{(n)}\mathbf{A}_{(n)}^\top)^q\mathbf{G}_n
	-\mathbf{Q}_n\mathbf{Q}_n^\top(\mathbf{A}_{(n)}\mathbf{A}_{(n)}^\top)^q\mathbf{G}_n\|_F^2\|\mathbf{F}_n\|_2^2.
	\end{aligned}
	\end{equation*}
	From the triangular inequality, we have
	\begin{equation*}
	\begin{aligned}
	&\|(\mathbf{A}_{(n)}\mathbf{A}_{(n)}^\top)^q\mathbf{G}_n
	-\mathbf{Q}_n\mathbf{Q}_n^\top(\mathbf{A}_{(n)}\mathbf{A}_{(n)}^\top)^q\mathbf{G}_n\|_F^2\leq\|(\mathbf{A}_{(n)}\mathbf{A}_{(n)}^\top)^q\mathbf{G}_n-\mathbf{Q}_n\mathbf{R}_n\|_F^2\\
	&\quad\quad+
	\|\mathbf{Q}_n\mathbf{R}_n-\mathbf{Q}_n\mathbf{Q}_n^\top\mathbf{Q}_n\mathbf{R}_n\|_F^2+\|\mathbf{Q}_n\mathbf{Q}_n^\top\mathbf{Q}_n\mathbf{R}_n
	-\mathbf{Q}_n\mathbf{Q}_n^\top(\mathbf{A}_{(n)}\mathbf{A}_{(n)}^\top)^q\mathbf{G}_n\|_F^2.
	\end{aligned}
	\end{equation*}
	By the fact that $\mathbf{Q}_n^\top\mathbf{Q}_n=\mathbf{I}_{\mu_n}$, we have $\|\mathbf{Q}_n\mathbf{R}_n-\mathbf{Q}_n\mathbf{Q}_n^\top\mathbf{Q}_n\mathbf{R}_n\|_F^2=0$ and
	\begin{equation*}
	\begin{aligned}
	\|\mathbf{Q}_n\mathbf{Q}_n^\top\mathbf{Q}_n\mathbf{R}_n
	-\mathbf{Q}_n\mathbf{Q}_n^\top(\mathbf{A}_{(n)}\mathbf{A}_{(n)}^\top)^q\mathbf{G}_n\|_F^2\leq
	\|\mathbf{Q}_n\mathbf{R}_n
	-(\mathbf{A}_{(n)}\mathbf{A}_{(n)}^\top)^q\mathbf{G}_n\|_F^2.
	\end{aligned}
	\end{equation*}
	Therefore,
	\begin{equation}\label{RALMRA:eqn13}
	\begin{aligned}
	&\|(\mathbf{A}_{(n)}\mathbf{A}_{(n)}^\top)^q\mathbf{G}_n
	-\mathbf{Q}_n\mathbf{Q}_n^\top(\mathbf{A}_{(n)}\mathbf{A}_{(n)}^\top)^q\mathbf{G}_n\|_F^2\\
 &\leq 2\|\mathbf{F}_n\|_2^2\|\mathbf{Q}_n\mathbf{R}_n
	-(\mathbf{A}_{(n)}\mathbf{A}_{(n)}^\top)^q\mathbf{G}_n\|_F^2.
	\end{aligned}
	\end{equation}
	Combining (\ref{RALMRA:eqn11}), (\ref{RALMRA:eqn12}) and (\ref{RALMRA:eqn13}) yields (\ref{RALMRA:eqn10}).
\end{proof}

The upper bound for $\|\mathbf{Q}_n\mathbf{R}_n-(\mathbf{A}_{(n)}\mathbf{A}_{(n)}^\top)^q\mathbf{G}_n\|_F$ is obtained by combining Lemmas \ref{RALMRA:lem1}, \ref{RALMRA:lem3} and \ref{RALMRA:lem4} together, summarized in the following theorem.
\begin{theorem}\label{RALMRA:thm3}
	For a given $n$, suppose that $\mathbf{A}_{(n)}\in\mathbb{R}^{I_n\times I_1\dots I_{n-1}I_{n+1}\dots I_{N}}$ and $\mathbf{G}_n\in\mathbb{R}^{I_n\times (\mu_n+K)}$ is a  standard Gaussian matrix with $\mu_n< \mu_n+K< I_n'$. For $\beta>1$ and $\gamma>1$, let
	\begin{equation}\label{RALMRA:eqn14}
	1-\frac{1}{4(\gamma^2-1)\sqrt{\pi I_n\gamma^2}}\left(\frac{2\gamma^2}{e^{\gamma^2-1}}\right)^2
	\end{equation}
	be nonnegative. For a given positive integer $q>0$, there exist an orthonormal matrix $\mathbf{Q}_n\in\mathbb{R}^{I_n\times \mu_n}$ and $\mathbf{R}\in\mathbb{R}^{\mu_n\times (\mu_n+K)}$ such that
	\begin{equation*}
	\begin{aligned}
	&\|\mathbf{Q}_n\mathbf{R}_n-(\mathbf{A}_{(n)}\mathbf{A}_{(n)}^\top)^q\mathbf{G}_n\|_F\leq
	\sqrt{2I_n}\gamma\Delta_{\mu_n+1}(\mathbf{A}_{(n)},q),
	\end{aligned}
	\end{equation*}
	with a probability at least the amount in {\rm (\ref{RALMRA:eqn14})}, where $$\Delta_{\mu_n+1}(\mathbf{A}_{(n)},q)=\left(\sum_{i=\mu_n+1}^{I_n'}\sigma_i(\mathbf{A}_{(n)})^{4q}\right)^{1/2}.$$
\end{theorem}

According to the properties of the matrix $\mathbf{G}_n$, for each $n$, we estimate $\|(\mathbf{A}_{(n)}\mathbf{A}_{(n)}^\top)^q
\mathbf{G}_n\mathbf{F}_n-\mathbf{A}_{(n)}\|_F$ and $\|\mathbf{F}_n\|_2$.
\begin{theorem}\label{RALMRA:thm4}
	For a given $n$, suppose that $\mathbf{A}_{(n)}\in\mathbb{R}^{I_n\times I_1\dots I_{n-1}I_{n+1}\dots I_{N}}$ and $\mathbf{G}_n\in\mathbb{R}^{I_n\times (\mu_n+K)}$ is a randomly chosen standard Gaussian matrix with $\mu_n< \mu_n+K< I_n'$. For $\beta>1$ and $\gamma>1$, if
	\begin{equation}\label{RALMRA:eqn15}
 \begin{aligned}
	&1-\frac{1}{4(\gamma^2-1)\sqrt{\pi \max\{I_n'-\mu_n,\mu_n+K\}\gamma^2}}\left(\frac{2\gamma^2}{e^{\gamma^2-1}}\right)^2\\
 &-
	\frac{1}{\sqrt{2\pi(K+1)}}\left(\frac{e}{K+1}\right)^{K+1}
 \end{aligned}
	\end{equation}
	is nonnegative, then, there exists a matrix $\mathbf{F}\in\mathbb{R}^{(\mu_n+K)\times \mu_1\dots \mu_{n-1}I_{n+1}\dots I_{N}}$ such that
	\begin{equation*}
	\begin{aligned}
	&\|(\mathbf{A}_{(n)}\mathbf{A}_{(n)}^\top)^q
	\mathbf{G}_n\mathbf{F}_n-\mathbf{A}_{(n)}\|_F\\
	&\leq\frac{\sqrt{2\max\{I_n'-\mu_n,\mu_n+K\}(\mu_n+K)}\gamma\beta}{\sigma_{\mu_n}(\mathbf{A}_{(n)})^{2q-1}}\Delta_{\mu_n+1}(\mathbf{A}_{(n)},q)
	+\Delta_{\mu_n+1}(\mathbf{A}_{(n)}),
	\end{aligned}
	\end{equation*}
	and
	\begin{equation*}
	\|\mathbf{F}_n\|_2\leq\frac{\sqrt{\mu_n+K}\beta}{\sigma_{\mu_n}(\mathbf{A}_{(n)})^{2q-1}}
	\end{equation*}
	with probability not less than the amount in {\rm (\ref{RALMRA:eqn15})}.
\end{theorem}
\begin{proof}
	Suppose that $\mathbf{A}_{(n)}=\mathbf{U}_n{\bf \Sigma}_n\mathbf{V}_n^\top$, where $\mathbf{U}_n\in\mathbb{R}^{I_n\times I_n'}$ and $\mathbf{V}_n\in\mathbb{R}^{I_1\dots I_{n-1}I_{n+1}\dots I_N\times I_n'}$ are orthonornal, and the diagonal entries of the diagonal matrix ${\bf \Sigma}_n\in\mathbb{R}^{I_n'\times I_n'}$ are nonnegative and arranged in descending order. Then, we have
	\begin{equation*}
	(\mathbf{A}_{(n)}\mathbf{A}_{(n)}^\top)^q=\mathbf{U}_n{\bf \Sigma}_n^{2q}\mathbf{U}_n^\top.
	\end{equation*}
	Assume that given $\mathbf{U}_n^\top$ and $\mathbf{G}_n$, suppose that
	\begin{equation*}
	\mathbf{U}_n^\top\mathbf{G}_n=\begin{pmatrix}
	\mathbf{H}\\
	\mathbf{R}
	\end{pmatrix},
	\end{equation*}
	where $\mathbf{H}\in\mathbb{R}^{\mu_n\times(\mu_n+K)}$ and $\mathbf{R}\in\mathbb{R}^{(I_n'-\mu_n)\times (\mu_n+K)}$. Since $\mathbf{G}_n$ is a standard Gaussian matrix, and $\mathbf{V}$ is an orthogonal matrix, then $\mathbf{U}_n^\top\mathbf{G}_n$ is also a standard Gaussian matrix. Therefore, $\mathbf{H}$ and $\mathbf{R}$ are also standard Gaussian matrices. Let $\mathbf{P}\in\mathbb{R}^{(\mu_n+K)\times I_n}$ be defined by
	\begin{equation*}
	\mathbf{P}=\begin{pmatrix}
	\mathbf{H}^\dag\bm{\Sigma}_{n,1}^{-2q+1}&\mathbf{0}_{(\mu_n+K)\times (I_n-\mu_n)}
	\end{pmatrix},
	\quad
	\bm{\Sigma}_{n,1}=\bm{\Sigma}_n(1:\mu_n,1:\mu_n).
	\end{equation*}
	Define $\mathbf{F}_n=\mathbf{P}\mathbf{V}_n^\top$. From Lemma \ref{RALMRA:lem2}, we estimate $\|\mathbf{F}_n\|_2$ as
	\begin{equation*}
	\|\mathbf{F}_n\|_2=\|\mathbf{P}\mathbf{V}_n^\top\|_2=
	\left\|\mathbf{H}^\dag\bm{\Sigma}_{n,1}^{-2q+1}\right\|_2
	\leq\frac{1}{\sigma_{\mu_n}(\mathbf{A}_{(n)})^{2q-1}}\frac{1}{\sigma_{\min}(\mathbf{H})}
	\leq\frac{\sqrt{\mu_n+K}\beta}{\sigma_{\mu_n}(\mathbf{A}_{(n)})^{2q-1}}
	\end{equation*}
	with a probability at least
	\begin{equation*}
	1-\frac{1}{\sqrt{2\pi(K+1)}}\left(\frac{e}{K+1}\right)^{K+1},
	\end{equation*}
	where $\beta>1$.
	Now, we can bound $\|(\mathbf{A}_{(n)}\mathbf{A}_{(n)}^\top)^q
	\mathbf{G}_n\mathbf{F}_n-\mathbf{A}_{(n)}\|_F$. Note that we have
	\begin{equation*}
 \begin{aligned}
	&(\mathbf{A}_{(n)}\mathbf{A}_{(n)}^\top)^q
	\mathbf{G}_n\mathbf{F}_n-\mathbf{A}_{(n)}
	\\
 &=\mathbf{U}_n\bm{\Sigma}_n\left(\bm{\Sigma}_n^{2q-1}
	\begin{pmatrix}
	\mathbf{H}\\
	\mathbf{R}\end{pmatrix}\begin{pmatrix}
	\mathbf{H}^\dag\bm{\Sigma}_{n,1}^{-2q+1}&\mathbf{0}_{(\mu_n+K)\times (I_n-\mu_n)}
	\end{pmatrix}-\mathbf{I}_{I_n}\right)\mathbf{V}_n^{\top}.
 \end{aligned}
	\end{equation*}
	Let $\bm{\Sigma}_{n,2}={\bf \Sigma}_n((\mu_n+1):J_n,(\mu_n+1):J_n)$. Then
	\begin{equation*}
	\begin{aligned}
	&\bm{\Sigma}_n\left(\bm{\Sigma}_n^{2q-1}
	\begin{pmatrix}
	\mathbf{H}\\
	\mathbf{R}\end{pmatrix}\begin{pmatrix}
	\mathbf{H}^\dag\bm{\Sigma}_{n,1}^{-2q+1}&\mathbf{0}_{(\mu_n+K)\times (I_n-\mu_n)}
	\end{pmatrix}-\mathbf{I}_{I_n}\right)\\
	&\quad\quad=\bm{\Sigma}_n\begin{pmatrix}
	\bm{\Sigma}_{n,1}^{2q-1}&\mathbf{0}_{\mu_n\times (I_n-\mu_n)}\\
	\mathbf{0}_{(I_n-\mu_n)\times \mu_n}&\bm{\Sigma}_{n,2}^{2q-1}
	\end{pmatrix}
	\begin{pmatrix}
	\mathbf{0}_{\mu_n\times \mu_n}&\mathbf{0}_{\mu_n\times (I_n-\mu_n)}\\
	\mathbf{R}\mathbf{H}^{\dag}\bm{\Sigma}_{n,1}^{-2q+1}&-\mathbf{I}_{I_n-\mu_n}
	\end{pmatrix}\\
	&\quad\quad=\begin{pmatrix}
	\mathbf{0}_{\mu_n\times \mu_n}&\mathbf{0}_{\mu_n\times (I_n-\mu_n)}\\
	\bm{\Sigma}_{n,2}^{2q}\mathbf{R}\mathbf{H}^{\dag}\bm{\Sigma}_{n,1}^{-2q+1}&-\bm{\Sigma}_{n,2}
	\end{pmatrix}.
	\end{aligned}
	\end{equation*}
	The norm of the last term is:
	\begin{equation*}
	\left\|\begin{pmatrix}
	\mathbf{0}_{\mu_n\times \mu_n}&\mathbf{0}_{\mu_n\times (I_n-\mu_n)}\\
	\bm{\Sigma}_{n,2}^{2q}\mathbf{R}\mathbf{H}^{\dag}\bm{\Sigma}_{n,1}^{-2q+1}&-\bm{\Sigma}_{n,2}
	\end{pmatrix}\right\|_F
	\leq \left\|\bm{\Sigma}_{n,2}^{2q}\mathbf{R}\mathbf{H}^{\dag}\bm{\Sigma}_{n,1}^{-2q+1}\right\|_F
	+\|\bm{\Sigma}_{n,2}\|_F.
	\end{equation*}
	Moreover,
	\begin{equation*}
	\begin{aligned}
	\left\|\bm{\Sigma}_{n,2}^{2q}\mathbf{R}\mathbf{H}^{\dag}\bm{\Sigma}_{n,1}^{-2q+1}\right\|_F&\leq
	\left\|\bm{\Sigma}_{n,1}^{-2q+1}\right\|_2\left\|\mathbf{H}^{\dag}\right\|_2\|\mathbf{R}\|_2
	\left\|\bm{\Sigma}_{n,2}^{2q}\right\|_F\\
	&\leq\frac{1}{\sigma_{\mu_n}(\mathbf{A}_{(n)})^{2q-1}}\left\|\mathbf{H}^{\dag}\right\|_2\|\mathbf{R}\|_2\Delta_{\mu_n+1}(\mathbf{A}_{(n)},q),
	\end{aligned}
	\end{equation*}
	with $\Delta_{\mu_n+1}(\mathbf{A}_{(n)},q)=\|\bm{\Sigma}_{n,2}^{2q}\|_F$. We also know that
	$$\|\mathbf{R}\|_2\leq \sqrt{2\max\{I_n'-\mu_n,\mu_n+K\}}\gamma$$
	with a probability at least $$1-\frac{1}{4(\gamma^2-1)\sqrt{\pi \max\{I_n'-\mu_n,\mu_n+K\}\gamma^2}}\left(\frac{2\gamma^2}{e^{\gamma^2-1}}\right)^2.$$
	Therefore,
	\begin{equation*}
	\begin{aligned}
	&\left\|(\mathbf{A}_{(n)}\mathbf{A}_{(n)}^\top)^q\mathbf{G}\mathbf{F}-\mathbf{A}_{(n)}\right\|_F\\
	&\leq\frac{1}{\sigma_{\mu_n}(\mathbf{A}_{(n)})^{2q-1}}\left\|\mathbf{H}^{\dag}\right\|_2\|\mathbf{R}\|_2\Delta_{\mu_n+1}(\mathbf{A}_{(n)},q)
	+\|\bm{\Sigma}_{n,2}\|_F\\
	&\quad\leq\frac{\sqrt{2\max\{I_n'-\mu_n,\mu_n+K\}(\mu_n+K)}\gamma\beta}{\sigma_{\mu_n}(\mathbf{A}_{(n)})^{2q-1}}\Delta_{\mu_n+1}(\mathbf{A}_{(n)},q)
	+\Delta_{\mu_n+1}(\mathbf{A}_{(n)}).
	\end{aligned}
	\end{equation*}
	with a probability at least the amount in (\ref{RALMRA:eqn15}).
\end{proof}

\begin{theorem}\label{RALMRA:thm5}
	For a given $n$, let $\mu_n$, $I_n$ and $K$ are positive integers such that $\mu_n< \mu_n+K< I_n'$. For $\beta>1$ and $\gamma>1$, suppose that
	\begin{equation}\label{RALMRA:eqn16}
	\begin{aligned}
	1-&\frac{1}{4(\gamma^2-1)\sqrt{\pi \max\{I_n'-\mu_n,\mu_n+K\}\gamma^2}}\left(\frac{2\gamma^2}{e^{\gamma^2-1}}\right)^2\\
	&-
	\frac{1}{\sqrt{2\pi(K+1)}}\left(\frac{e}{K+1}\right)^{K+1}-\frac{1}{4(\gamma^2-1)\sqrt{\pi I_n\gamma^2}}\left(\frac{2\gamma^2}{e^{\gamma^2-1}}\right)^2
	\end{aligned}
	\end{equation}
	is nonnegative. For $\mathbf{A}_{(n)}\in\mathbb{R}^{I_n\times I_1\dots I_{n-1}I_{n+1}\dots I_{N}}$ and an integer $q\geq 1$, the orthonormal matrix $\mathbf{Q}_n$ are obtained by Algorithm {\rm\ref{RALMRA:alg1}}. Then
	\begin{equation*}
	\|\mathbf{A}_{(n)}-\mathbf{Q}_n\mathbf{Q}_n^\top\mathbf{A}_{(n)}\|_F
	\leq 2\cdot\left(\lambda_n\Delta_{\mu_n+1}(\mathbf{A}_{(n)},q)
	+\Delta_{\mu_n+1}(\mathbf{A}_{(n)})\right)
	\end{equation*}
	holds with probability at least the amount in {\rm (\ref{RALMRA:eqn16})}, where
	\begin{equation*}
	\lambda_n=\frac{(\sqrt{\max\{I_n'-\mu_n,\mu_n+K\}}+\sqrt{I_n})\sqrt{2(\mu_n+K)}\gamma\beta}{\sigma_{\mu_n}(\mathbf{A}_{(n)})^{2q-1}}.
	\end{equation*}
\end{theorem}
\begin{proof}
	By the fact that $a^2\leq b^2+c^2$ implies that $a\leq b+c$, where $a$, $b$ and $c$ are arbitrary positive numbers, it is deduced from Lemma \ref{RALMRA:lem7} that
	\begin{equation*}
 \begin{aligned}
	\|\mathbf{A}_{(n)}-\mathbf{Q}_n\mathbf{Q}_n^\top\mathbf{A}_{(n)}\|_F
	\leq &2\|(\mathbf{A}_{(n)}\mathbf{A}_{(n)}^\top)^q\mathbf{G}_n\mathbf{F}_n-\mathbf{A}_{(n)}\|_F\\
	&+2\|\mathbf{F}\|_2\|\mathbf{Q}_n\mathbf{R}_n-(\mathbf{A}_{(n)}\mathbf{A}_{(n)}^\top)^q\mathbf{G}_n\|_F.
 \end{aligned}
	\end{equation*}
	By combining Theorems \ref{RALMRA:thm3} and \ref{RALMRA:thm4} together, the result is obtained.
	\end{proof}

For $N$ given positive integers $\mu_n\leq I_n\ (n=1,2,\dots,N)$, when $\{\mathbf{Q}_1,\mathbf{Q}_2,\dots,\mathbf{Q}_N\}$ is acquired from Algorithm \ref{RALMRA:alg1}, the upper bounds of $\|\mathcal{A}-\widehat{\mathcal{A}}\|_F$ is given in the following theorem, which is derived from (\ref{RALMRA:eqn25}) and Theorem \ref{RALMRA:thm5}.
\begin{theorem}
	For each $n$, let $\mu_n$, $I_n$ and $K$ are positive integers such that $\mu_n< \mu_n+K< I_n'$. For $\beta>1$ and $\gamma>1$, suppose that
	\begin{equation}\label{RALMRA:eqn17}
	\begin{aligned}
	1-&\sum_{n=1}^N\left(\frac{1}{4(\gamma^2-1)\sqrt{\pi \max\{I_n'-\mu_n,\mu_n+K\}\gamma^2}}\left(\frac{2\gamma^2}{e^{\gamma^2-1}}\right)^2\right.\\
	&\left.+
	\frac{1}{\sqrt{2\pi(K+1)}}\left(\frac{e}{K+1}\right)^{K+1}+\frac{1}{4(\gamma^2-1)\sqrt{\pi I_n\gamma^2}}\left(\frac{2\gamma^2}{e^{\gamma^2-1}}\right)^2\right)
	\end{aligned}
	\end{equation}
	is nonnegative. For a given tensor $\mathcal{A}\in\mathbb{R}^{I_1\times I_2\times\dots \times  I_N}$ and an integer $q\geq 1$, $N$ orthonormal matrices $\mathbf{Q}_n$ are obtained by Algorithm {\rm\ref{RALMRA:alg1}}. Then
	\begin{equation*}
	\|\mathcal{A}-\widehat{\mathcal{A}}\|_F
	\leq 2\cdot\sum_{n=1}^N\left(\lambda_n\Delta_{\mu_n+1}(\mathbf{A}_{(n)},q)
	+\Delta_{\mu_n+1}(\mathbf{A}_{(n)})\right)
	\end{equation*}
	holds with probability at least the amount in {\rm (\ref{RALMRA:eqn17})}, where
	\begin{equation*}
	\lambda_n=\frac{(\sqrt{\max\{I_n'-\mu_n,\mu_n+K\}}+\sqrt{I_n})\sqrt{2(\mu_n+K)}\gamma\beta}{\sigma_{\mu_n}(\mathbf{A}_{(n)})^{2q-1}}.
	\end{equation*}
	\end{theorem}

When the matrix $\mathbf{G}_n$ in Theorems \ref{RALMRA:thm3} and \ref{RALMRA:thm4} is a randomly chosen sub-Gaussian matrix, these theorems should be rewritten as follows.
\begin{remark}
	For a given $n$, suppose that $\mathbf{A}_{(n)}\in\mathbb{R}^{I_n\times I_1\dots I_{n-1}I_{n+1}\dots I_{N}}$ and $\mathbf{G}_n\in\mathbb{R}^{I_n\times (\mu_n+K)}$ is a randomly chosen sub-Gaussian matrix with $\mu_n< \mu_n+K<I_n'$. For a given positive integer $q>0$, there exist an orthonormal matrix $\mathbf{Q}_n\in\mathbb{R}^{I_n\times \mu_n}$ and $\mathbf{R}\in\mathbb{R}^{\mu_n\times (\mu_n+K)}$ such that
	\begin{equation*}
	\begin{aligned}
	&\|\mathbf{Q}_n\mathbf{R}_n-(\mathbf{A}_{(n)}\mathbf{A}_{(n)}^\top)^q\mathbf{G}_n\|_F\leq
	a_{n,1}I_n\Delta_{\mu_n+1}(\mathbf{A}_{(n)},q),
	\end{aligned}
	\end{equation*}
	with a probability at least $1-e^{-a_{n,2}I_n}$, where $a_{n,1}=6s\sqrt{a_{n,2}+4}$ and $s>1$.
\end{remark}
\begin{remark}
	For a given $n$, suppose that $\mathbf{A}_{(n)}\in\mathbb{R}^{I_n\times I_1\dots I_{n-1}I_{n+1}\dots I_{N}}$ and $\mathbf{G}_n\in\mathbb{R}^{I_n\times (\mu_n+K)}$ is a randomly chosen sub-Gaussian matrix with $\mu_n< \mu_n+K<I_n'$ and $\mu_n+K>(1+1/\ln \mu_n )\mu_n$. We define $a_{n,1}$, $a_{n,2}$, $c_{n,1}$ and $c_{n,2}$ as in Theorems {\rm \ref{RALMRA:thm1}} and {\rm \ref{RALMRA:thm2}}. Then, there exists a matrix $\mathbf{F}\in\mathbb{R}^{(\mu_n+K)\times \mu_1\dots \mu_{n-1}I_{n+1}\dots I_{N}}$ such that
	\begin{equation*}
	\begin{aligned}
	&\|(\mathbf{A}_{(n)}\mathbf{A}_{(n)}^\top)^q
	\mathbf{G}_n\mathbf{F}_n-\mathbf{A}_{(n)}\|_F\\
	&\leq\frac{a_{n,1}c_{n,1}\sqrt{\max\{I_n'-\mu_n,\mu_n+K\}(\mu_n+K)}}{\sigma_{\mu_n}(\mathbf{A}_{(n)})^{2q-1}}\Delta_{\mu_n+1}(\mathbf{A}_{(n)},q)
	+\Delta_{\mu_n+1}(\mathbf{A}_{(n)}),
	\end{aligned}
	\end{equation*}
	and
	\begin{equation*}
	\|\mathbf{F}_n\|_2\leq\frac{c_{n,1}\sqrt{\mu_n+K}}{\sigma_{\mu_n}(\mathbf{A}_{(n)})^{2q-1}}
	\end{equation*}
	with a probability at least $1-e^{-c_{n,2}(\mu_n+K)}-e^{-a_{n,2}\max\{I_n'-\mu_n,\mu_n+K\}}$, where $a_{n,1}=6s\sqrt{a_{n,2}+4}$ and $s>1$.
\end{remark}
The following result is obtained by combining Lemma \ref{RALMRA:lem7}, and the above two remarks.
\begin{remark}
	For each $n$, let $\mu_n$, $I_n$ and $K$ are positive integers such that $\mu_n< \mu_n+K< I_n'$ and $\mu_n+K>(1+1/\ln \mu_n )\mu_n$. We define $a_{n,1}$, $a_{n,2}$, $c_{n,1}$ and $c_{n,2}$ as in Theorems {\rm \ref{RALMRA:thm1}} and {\rm \ref{RALMRA:thm2}} with $a_{n,1}=6s\sqrt{a_{n,2}+4}$ and $s>1$. For a given tensor $\mathcal{A}\in\mathbb{R}^{I_1\times I_2\times\dots \times  I_N}$ and an integer $q\geq 1$, $N$ orthonormal matrices $\mathbf{Q}_n$ are obtained by Algorithm {\rm\ref{RALMRA:alg1}}. Then
	\begin{equation*}
	\|\mathcal{A}-\widehat{\mathcal{A}}\|_F
	\leq 2\cdot\sum_{n=1}^N\left(\lambda_n\Delta_{\mu_n+1}(\mathbf{A}_{(n)},q)
	+\Delta_{\mu_n+1}(\mathbf{A}_{(n)})\right)
	\end{equation*}
	holds with a probability at least $$1-\sum_{n=1}^N\left(e^{-c_{n,2}(\mu_n+K)}+e^{-a_{n,2}\max\{I_n'-\mu_n,\mu_n+K\}}\right),$$ where
	\begin{equation*}
	\lambda_n=\frac{(\sqrt{\max\{I_n'-\mu_n,\mu_n+K\}}+I_n)\sqrt{\mu_n+K}a_{n,1}c_{n,1}}{\sigma_{\mu_n}(\mathbf{A}_{(n)})^{2q-1}}.
	\end{equation*}
\end{remark}
\subsection{Analyzing Algorithm \ref{RALMRA:alg4}}

For each $n$, when $\mathbf{Q}_n$ is obtained from Algorithm \ref{RALMRA:alg4}, we have
\begin{equation*}
\begin{aligned}
&\|\mathbf{A}_{(n)}-\mathbf{Q}_n\mathbf{Q}_n^\top\mathbf{A}_{(n)}\|_F^2\\
&\leq {\rm rank}(\mathbf{A}_{(n)})\|(\mathbf{A}_{(n)}-\mathbf{Q}_n\mathbf{Q}_n^\top\mathbf{A}_{(n)})(\mathbf{A}_{(n)}-\mathbf{Q}_n\mathbf{Q}_n^\top\mathbf{A}_{(n)})^\top\|_F\\
&={\rm rank}(\mathbf{A}_{(n)})\|\mathbf{A}_{(n)}\mathbf{A}_{(n)}^\top-\mathbf{Q}_n\mathbf{Q}_n^\top\mathbf{A}_{(n)}\mathbf{A}_{(n)}^\top\|_F\\
&\leq {\rm rank}(\mathbf{A}_{(n)})\left(\|\mathbf{A}_{(n)}\mathbf{A}_{(n)}^\top-\mathbf{C}_n\mathbf{C}_n'^\top\|_F+\|\mathbf{C}_n'\mathbf{C}_n'^\top-\mathbf{Q}_n\mathbf{Q}_n^\top\mathbf{C}_n'\mathbf{C}_n'^\top\|_F\right),
\end{aligned}
\end{equation*}
where $\mathbf{S}_n\in\mathbb{R}^{I_1\dots I_{n-1}I_{n+1}\dots I_N\times T_n}\sim{\rm RANDSAMPLE}(T_n,\mathbf{p}_n)$ and $\mathbf{C}_n'=\mathbf{A}_{(n)}\mathbf{S}_n$.
\begin{corollary}\label{RALMRA:cor1}
	For a given $n$, suppose that $\mathbf{A}_{(n)}\in\mathbb{R}^{I_n\times I_1\dots I_{n-1}I_{n+1}\dots I_N}$ and $\mathbf{C}_n'=\mathbf{A}_{(n)}\mathbf{S}_n$, where $\mathbf{S}_n\in\mathbb{R}^{I_1\dots I_{n-1}I_{n+1}\dots I_N\times T_n}\sim{\rm RANDSAMPLE}(T_n,\mathbf{p}_n)$. If the probabilities $\mathbf{p}_n$ are the nearly optimal probabilities, then for any $\delta_n\in(0,1)$, there exists $\eta_n=1+\sqrt{(8/\delta_n)\log(1/\delta_n)}$ such that
	\begin{equation*}
	\|\mathbf{A}_{(n)}\mathbf{A}_{(n)}^\top-\mathbf{C}_n'\mathbf{C}_n'^\top\|_F\leq\frac{\eta_n}{\sqrt{\beta T_n}}\|\mathbf{A}_{(n)}\|_F^2
	\end{equation*}
	with a probability at least $1-\delta_n$. If the probabilities $\mathbf{p}_n$ are the uniform probabilities, then for any $\delta_n\in(0,1)$, there exists $\gamma_n\leq 1+\frac{I_n'}{\sqrt{T_n}}\sqrt{8\log(1/\delta_n)}\|\mathbf{A}_{(n)}\|_F^2$ such that
	\begin{equation*}
	\|\mathbf{A}_{(n)}\mathbf{A}_{(n)}^\top-\mathbf{C}_n'\mathbf{C}_n'^\top\|_F\leq\left(1+\sqrt{\frac{I_n'}{T_n}}+\frac{I_n'}{\sqrt{T_n}}\sqrt{8\log(1/\delta_n)}\right)\|\mathbf{A}_{(n)}\|_F^2
	\end{equation*}
	with a probability at least $1-\delta_n$.
	\end{corollary}
\begin{proof}
	The first inequality is derived from Lemma \ref{RALMRA:lem5}. We now prove the second inequality. As shown in Lemma \ref{RALMRA:lem6}, we have that for any $\delta_n\in(0,1)$, there exists $$\gamma_n=1+\frac{I_n'}{\sqrt{I_n}}\sqrt{8\log(1/\delta)}\max_i\|\mathbf{A}_{(n)}(:,i)\|_2^2$$ such that
	\begin{equation*}
	\|\mathbf{A}_{(n)}\mathbf{A}_{(n)}^\top-\mathbf{C}_n'\mathbf{C}_n'^\top\|_F \leq\sqrt{\frac{I_n'}{K}}\left(\sum_{i=1}^{I_n'}\|\mathbf{A}_{(n)}(:,i)\|_2^4\right)^{1/2}+\gamma_n
	\end{equation*}
	holds with a probability at least $1-\delta_n$. It is obvious that $\max_i\|\mathbf{A}_{(n)}(:,i)\|_2^2\leq \|\mathbf{A}_{(n)}\|_F^2$ and
	\begin{equation*}
	\begin{aligned}
	&\sum_{i=1}^{I_n'}\|\mathbf{A}_{(n)}(:,i)\|_2^4\leq
	\left(\sum_{i=1}^{I_n'}\|\mathbf{A}_{(n)}(:,i)\|_2^2\right)^2=\|\mathbf{A}_{(n)}\|_F^4,
	\end{aligned}
	\end{equation*}
	which implies the second inequality.
	\end{proof}

We now estimate the upper bound for $\|\mathbf{C}_n'\mathbf{C}_n'^\top-\mathbf{Q}_n\mathbf{Q}_n^\top\mathbf{C}_n'\mathbf{C}_n'^\top\|_F$. When the matrix $\mathbf{Q}_n$ is obtained by Algorithm \ref{RALMRA:alg4} with a given $n$, the matrix $\mathbf{Q}_n\mathbf{Q}_n^\top\mathbf{C}_n'\mathbf{C}_n'^\top$ is a good approximation to the matrix $\mathbf{C}_n'\mathbf{C}_n'^\top$, provided that there exist matrices $\mathbf{G}_n\in\mathbb{R}^{I_n\times (\mu_n+K)}$ and $\mathbf{R}_n\in\mathbb{R}^{\mu_n\times (\mu_n+K)}$ such that: (a) $\mathbf{Q}_n$ is orthonormal; (b) $\mathbf{Q}_n\mathbf{R}_n$ is a good approximation of $\mathbf{C}_n'\mathbf{C}_n'^\top\mathbf{G}_n$; and (c) there exists a matrix $\mathbf{F}\in\mathbb{R}^{(\mu_n+K)\times I_n}$ such that $\|\mathbf{F}\|_2$ is not too large and $\mathbf{C}_n'\mathbf{C}_n'^\top\mathbf{G}_n\mathbf{F}_n$ is a good approximation of $\mathbf{C}_n'\mathbf{C}_n'^\top$.

Similar to Lemma \ref{RALMRA:lem7}, the following lemma can be easily proved.
\begin{lemma}
	\label{RALMRA:lem7add}
	For a given $n$, suppose that $\mathbf{C}_n'\in\mathbb{R}^{I_n\times T_n}$, $\mathbf{Q}_n\in\mathbb{R}^{I_n\times \mu_n}$ is orthonormal, $\mathbf{R}_n\in\mathbb{R}^{\mu_n\times (\mu_n+K)}$, $\mathbf{F}\in\mathbb{R}^{(\mu_n+K)\times I_n}$ and $\mathbf{G}_n\in\mathbb{R}^{I_n\times (\mu_n+K)}$ with $\mu_n< \mu_n+K< I_n$. Then, we have
	\begin{equation*}
	\begin{aligned}
	\|\mathbf{C}_n'\mathbf{C}_n'^\top-\mathbf{Q}_n\mathbf{Q}_n^\top\mathbf{C}_n'\mathbf{C}_n'^\top\|_F^2
	&\leq 2\|\mathbf{C}_n'\mathbf{C}_n'^\top\mathbf{G}_n\mathbf{F}_n-\mathbf{C}_n'\mathbf{C}_n'^\top\|_F^2\\
&\quad	+2\|\mathbf{F}\|_2^2\|\mathbf{Q}_n\mathbf{R}_n-\mathbf{C}_n'\mathbf{C}_n'^\top\mathbf{G}_n\|_F^2.
	\end{aligned}
	\end{equation*}
\end{lemma}

According to the properties of the matrix $\mathbf{G}_n$, for each $n$, we estimate $\|\mathbf{C}_n'\mathbf{C}_n'^\top\mathbf{G}_n\mathbf{F}_n-\mathbf{C}_n'\mathbf{C}_n'^\top\|_F$ and $\|\mathbf{F}_n\|_2$.
\begin{theorem}\label{RALMRA:thm4add}
	For a given $n$, suppose that $\mathbf{C}_n'\in\mathbb{R}^{I_n\times T_n}$ and $\mathbf{G}_n\in\mathbb{R}^{I_n\times (\mu_n+K)}$ is a randomly chosen standard Gaussian matrix with $\mu_n< \mu_n+K< I_n''$. For $\beta>1$ and $\gamma>1$, if
	\begin{equation}\label{RALMRA:eqn15add}
 \begin{aligned}
	&1-\frac{1}{4(\gamma^2-1)\sqrt{\pi \max\{I_n''-\mu_n,\mu_n+K\}\gamma^2}}\left(\frac{2\gamma^2}{e^{\gamma^2-1}}\right)^2\\
 &-
	\frac{1}{\sqrt{2\pi(K+1)}}\left(\frac{e}{K+1}\right)^{K+1}
 \end{aligned}
	\end{equation}
	is nonnegative, then, there exists a matrix $\mathbf{F}\in\mathbb{R}^{(\mu_n+K)\times \mu_1\dots \mu_{n-1}I_{n+1}\dots I_{N}}$ such that
	\begin{equation*}
	\begin{aligned}
	&\|\mathbf{C}_n'\mathbf{C}_n'^\top\mathbf{G}_n\mathbf{F}_n-\mathbf{C}_n'\mathbf{C}_n'^\top\|_F\\
	&\leq(\sqrt{2\max\{I_n''-\mu_n,\mu_n+K\}(\mu_n+K)}\gamma\beta+1)\Delta_{\mu_n+1}(\mathbf{C}_{n},1),
	\end{aligned}
	\end{equation*}
	and
	\begin{equation*}
	\|\mathbf{F}_n\|_2\leq\sqrt{\mu_n+K}\beta
	\end{equation*}
	with probability not less than the amount in {\rm (\ref{RALMRA:eqn15add})}, where
 \begin{equation*}
 \Delta_{\mu_n+1}(\mathbf{C}_{n},1)=\sqrt{\sum_{i=\mu_n+1}^{I_n''}\sigma_{i}(\mathbf{C}_n')^4}.
 \end{equation*}
\end{theorem}
\begin{proof}
	Suppose that $\mathbf{C}_{n}'=\mathbf{U}_n{\bf \Sigma}_n\mathbf{V}_n^\top$, where $\mathbf{U}_n\in\mathbb{R}^{I_n\times I_n''}$ and $\mathbf{V}_n\in\mathbb{R}^{I_n\times I_n''}$ are orthonornal, and the diagonal entries of the diagonal matrix ${\bf \Sigma}_n\in\mathbb{R}^{I_n''\times I_n''}$ are nonnegative and arranged in descending order. Then, we have
	\begin{equation*}
	\mathbf{C}_n'\mathbf{C}_n'^\top=\mathbf{U}_n{\bf \Sigma}_n^{2}\mathbf{U}_n^\top.
	\end{equation*}
	Assume that given $\mathbf{U}_n^\top$ and $\mathbf{G}_n$, suppose that
	\begin{equation*}
	\mathbf{U}_n^\top\mathbf{G}_n=\begin{pmatrix}
	\mathbf{H}\\
	\mathbf{R}
	\end{pmatrix},
	\end{equation*}
	where $\mathbf{H}\in\mathbb{R}^{\mu_n\times(\mu_n+K)}$ and $\mathbf{R}\in\mathbb{R}^{(I_n''-\mu_n)\times (\mu_n+K)}$. Since $\mathbf{G}_n$ is a standard Gaussian matrix, and $\mathbf{V}$ is an orthogonal matrix, then $\mathbf{U}_n^\top\mathbf{G}_n$ is also a standard Gaussian matrix. Therefore, $\mathbf{H}$ and $\mathbf{R}$ are also standard Gaussian matrices. Let $\mathbf{P}\in\mathbb{R}^{(\mu_n+K)\times I_n}$ be defined by
	\begin{equation*}
	\mathbf{P}=\begin{pmatrix}
	\mathbf{H}^\dag&\mathbf{0}_{(\mu_n+K)\times (I_n-\mu_n)}
	\end{pmatrix},
	\quad
	\bm{\Sigma}_{n,1}=\bm{\Sigma}_n(1:\mu_n,1:\mu_n).
	\end{equation*}
	Define $\mathbf{F}_n=\mathbf{P}\mathbf{U}_n^\top$. From Lemma \ref{RALMRA:lem2}, we estimate $\|\mathbf{F}_n\|_2$ as
	\begin{equation*}
	\|\mathbf{F}_n\|_2=\|\mathbf{P}\mathbf{U}_n^\top\|_2=
	\left\|\mathbf{H}^\dag\right\|_2
	\leq\frac{1}{\sigma_{\min}(\mathbf{H})}
	\leq\sqrt{\mu_n+K}\beta
	\end{equation*}
	with a probability at least
	\begin{equation*}
	1-\frac{1}{\sqrt{2\pi(K+1)}}\left(\frac{e}{K+1}\right)^{K+1},
	\end{equation*}
	where $\beta>1$.
	Now, we can bound $\|\mathbf{C}_n'\mathbf{C}_n'^\top\mathbf{G}_n\mathbf{F}_n-\mathbf{C}_n'\mathbf{C}_n'^\top\|_F$. Note that we have
	\begin{equation*}
	\mathbf{C}_n'\mathbf{C}_n'^\top\mathbf{G}_n\mathbf{F}_n-\mathbf{C}_n'\mathbf{C}_n'^\top
	=\mathbf{U}_n\bm{\Sigma}_n^2\left(
	\begin{pmatrix}
	\mathbf{H}\\
	\mathbf{R}\end{pmatrix}\begin{pmatrix}
	\mathbf{H}^\dag&\mathbf{0}_{(\mu_n+K)\times (I_n-\mu_n)}
	\end{pmatrix}-\mathbf{I}_{I_n}\right)\mathbf{U}_n^{\top}.
	\end{equation*}
	Let $\bm{\Sigma}_{n,2}={\bf \Sigma}_n((\mu_n+1):J_n,(\mu_n+1):J_n)$. Then
	\begin{equation*}
	\begin{aligned}
	&\bm{\Sigma}_n^2\left(
	\begin{pmatrix}
	\mathbf{H}\\
	\mathbf{R}\end{pmatrix}\begin{pmatrix}
	\mathbf{H}^\dag&\mathbf{0}_{(\mu_n+K)\times (I_n-\mu_n)}
	\end{pmatrix}-\mathbf{I}_{I_n}\right)=\begin{pmatrix}
	\mathbf{0}_{\mu_n\times \mu_n}&\mathbf{0}_{\mu_n\times (I_n-\mu_n)}\\
	\bm{\Sigma}_{n,2}^{2}\mathbf{R}\mathbf{H}^{\dag}&-\bm{\Sigma}_{n,2}^2
	\end{pmatrix}.
	\end{aligned}
	\end{equation*}
	The norm of the last term is:
	\begin{equation*}
	\left\|\begin{pmatrix}
	\mathbf{0}_{\mu_n\times \mu_n}&\mathbf{0}_{\mu_n\times (I_n-\mu_n)}\\
	\bm{\Sigma}_{n,2}^{2}\mathbf{R}\mathbf{H}^{\dag}&-\bm{\Sigma}_{n,2}^2
	\end{pmatrix}\right\|_F
	\leq \left\|\bm{\Sigma}_{n,2}^{2}\mathbf{R}\mathbf{H}^{\dag}\right\|_F
	+\|\bm{\Sigma}_{n,2}^2\|_F.
	\end{equation*}
	Moreover,
	\begin{equation*}
	\left\|\bm{\Sigma}_{n,2}^{2}\mathbf{R}\mathbf{H}^{\dag}\right\|_F\leq
        \left\|\mathbf{H}^{\dag}\right\|_2\|\mathbf{R}\|_2
        \left\|\bm{\Sigma}_{n,2}^{2}\right\|_F
        =\left\|\mathbf{H}^{\dag}\right\|_2\|\mathbf{R}\|_2\Delta_{\mu_n+1}(\mathbf{C}_{n},1),
	\end{equation*}
	with $\Delta_{\mu_n+1}(\mathbf{C}_{n},1)=\|\bm{\Sigma}_{n,2}^{2}\|_F$. We also know that
	$$\|\mathbf{R}\|_2\leq \sqrt{2\max\{I_n''-\mu_n,\mu_n+K\}}\gamma$$
	with a probability at least $$1-\frac{1}{4(\gamma^2-1)\sqrt{\pi \max\{I_n''-\mu_n,\mu_n+K\}\gamma^2}}\left(\frac{2\gamma^2}{e^{\gamma^2-1}}\right)^2.$$
	Therefore,
	\begin{equation*}
	\begin{aligned}
	&\|\mathbf{C}_n'\mathbf{C}_n'^\top\mathbf{G}_n\mathbf{F}_n-\mathbf{C}_n'\mathbf{C}_n'^\top\|_F
	\leq(\|\mathbf{H}^{\dag}\|_2\|\mathbf{R}\|_2+1)\Delta_{\mu_n+1}(\mathbf{C}_{n},1).
	\end{aligned}
	\end{equation*}
	with a probability at least the amount in (\ref{RALMRA:eqn15add}).
\end{proof}

By combining Lemmas \ref{RALMRA:lem3} and \ref{RALMRA:lem7add}, and Theorem \ref{RALMRA:thm4add}, we give an upper bound for $\|\mathbf{C}_n'\mathbf{C}_n'^\top-\mathbf{Q}_n\mathbf{Q}_n^\top\mathbf{C}_n\mathbf{C}_n'^\top\|_F$, as shown in the following corollary.
\begin{corollary}\label{RALMRA:cor2}
	For a given $n$, let $\mu_n$, $I_n$, $T_n$ and $K$ are positive integers such that $\mu_n< \mu_n+K< I_n''$. For $\beta>1$ and $\gamma>1$, suppose that
	\begin{equation}\label{RALMRA:eqn18}
	\begin{aligned}
	1-&\frac{1}{4(\gamma^2-1)\sqrt{\pi \max\{I_n''-\mu_n,\mu_n+K\}\gamma^2}}\left(\frac{2\gamma^2}{e^{\gamma^2-1}}\right)^2\\
	&-
	\frac{1}{\sqrt{2\pi(K+1)}}\left(\frac{e}{K+1}\right)^{K+1}-\frac{1}{4(\gamma^2-1)\sqrt{\pi I_n\gamma^2}}\left(\frac{2\gamma^2}{e^{\gamma^2-1}}\right)^2
	\end{aligned}
	\end{equation}
	is nonnegative. For a given tensor $\mathcal{A}\in\mathbb{R}^{I_1\times I_2\times\dots \times  I_N}$ and $q= 1$, the orthonormal matrix $\mathbf{Q}_n$ is obtained by Algorithm {\rm\ref{RALMRA:alg4}}. Then
	\begin{equation*}
	\begin{aligned}
	\|\mathbf{C}_n'\mathbf{C}_n'^\top-\mathbf{Q}_n\mathbf{Q}_n^\top\mathbf{C}_n\mathbf{C}_n'^\top\|_F
	&\leq 2\cdot(\lambda_n+1)\Delta_{\mu_n+1}(\mathbf{C}_n,1)
	\end{aligned}
	\end{equation*}
	holds with a probability at least the amount in {\rm (\ref{RALMRA:eqn18})}, where
	\begin{equation*}
	\lambda_n=(\sqrt{\max\{I_n''-\mu_n,\mu_n+K\}}+\sqrt{I_n})\sqrt{2(\mu_n+K)}\gamma\beta.
	\end{equation*}
\end{corollary}

\begin{remark}
	When $q>1$, the constants $\lambda_n'$ and $\lambda_n''$ in Corollary {\rm \ref{RALMRA:cor2}} will be given by
	\begin{equation*}
	\lambda_n=\frac{(\sqrt{\max\{I_n''-\mu_n,\mu_n+K\}}+\sqrt{I_n})\sqrt{2(\mu_n+K)}\gamma\beta}
{\sigma_{\mu_n}(\mathbf{C}_n'\mathbf{C}_n'^\top)^{q-1}},
	\end{equation*}
	which depends on the matrix $\mathbf{C}_n'$. From {\rm\cite{boutsidis2014randomized}}, when $\mathbf{A}_{(n)}^\top$ is an orthonormal matrix, then with a probability at least $1-\delta$, the singular values of $\mathbf{C}_n'\mathbf{C}_n'^\top$ belong to \begin{equation*}
	\left[1-\sqrt{\frac{4I_n\log(2I_n/\delta)}{T_n}},1+\sqrt{\frac{4I_n\log(2I_n/\delta)}{T_n}}\right],
	\end{equation*}
	where $0<\delta<1$ and $4I_n\log(2I_n/\delta)<T_n\leq I_n'$. However, the lower bounds for the singular values of $\mathbf{C}_n'\mathbf{C}_n'^\top$ is unclear for a general matrix $\mathbf{A}_{(n)}$. Hence, how to analyze Algorithm {\rm\ref{RALMRA:alg4}} with $q>1$ is an open issue, which will be considered in the future.
	\end{remark}

The following corollary states the upper bound of $\|\mathcal{A}-\widehat{\mathcal{A}}\|_F$, where for $N$ given positive integers $\mu_n\leq I_n\ (n=1,2,\dots,N)$, the optimal solution $\{\mathbf{Q}_1,\mathbf{Q}_2,\dots,\mathbf{Q}_N\}$ is obtained from Algorithm \ref{RALMRA:alg4} with $q=1$.
\begin{corollary}\label{RALMRA:cor3}
	For each $n$, let $\mu_n$, $I_n$, $T_n$ and $K$ are positive integers such that $\mu_n\leq \mu_n+K\leq I_n''$. For $\beta>1$, $\gamma>1$ and $0<\delta_n<1$, suppose that
	\begin{equation}\label{RALMRA:eqn19}
	\begin{aligned}
	1-&\sum_{n=1}^N\left(\frac{1}{4(\gamma^2-1)\sqrt{\pi \max\{I_n''-\mu_n,\mu_n+K\}\gamma^2}}\left(\frac{2\gamma^2}{e^{\gamma^2-1}}\right)^2\right.\\
	&\left.+
	\frac{1}{\sqrt{2\pi(K+1)}}\left(\frac{e}{K+1}\right)^{K+1}+\frac{1}{4(\gamma^2-1)\sqrt{\pi I_n\gamma^2}}\left(\frac{2\gamma^2}{e^{\gamma^2-1}}\right)^2+\delta_n\right)
	\end{aligned}
	\end{equation}
	is nonnegative. For a given tensor $\mathcal{A}\in\mathbb{R}^{I_1\times I_2\times\dots \times  I_N}$ and $q= 1$, the orthonormal matrices $\mathbf{Q}_n\ (n=1,2,\dots,N)$ are obtained by Algorithm {\rm\ref{RALMRA:alg4}}. Then
	\begin{equation*}
	\begin{aligned}
	\|\mathcal{A}-\widehat{\mathcal{A}}\|_F^2
	&\leq \sum_{n=1}^N{\rm rank}(\mathbf{A}_{(n)})\left(2\cdot(1+\lambda_n)I_n\Delta_{\mu_n+1}(\mathbf{A}_{(n)},1)+\phi_n\|\mathcal{A}\|_F^2\right)
	\end{aligned}
	\end{equation*}
	holds with a probability at least the amount in {\rm (\ref{RALMRA:eqn19})}, where
	\begin{equation*}
	\lambda_n=(\sqrt{\max\{I_n''-\mu_n,\mu_n+K\}}+\sqrt{I_n})\sqrt{2(\mu_n+K)}\gamma\beta.
	\end{equation*}
	For the nearly optimal probabilities, there exists $\eta_n=1+\sqrt{(8/\delta_n)\log(1/\delta_n)}$ such that
	\begin{equation*}
	\phi_n=\frac{\eta_n}{\sqrt{\beta T_n}},
	\end{equation*}
	and for the uniform probabilities, there exists $\gamma_n\leq 1+\frac{I_n'}{\sqrt{T_n}}\sqrt{8\log(1/\delta_n)}\|\mathbf{A}_{(n)}\|_F^2$ such that
	\begin{equation*}
	\phi_n=1+\sqrt{\frac{I_n'}{T_n}}+\frac{I_n'}{\sqrt{T_n}}\sqrt{8\log(1/\delta_n)}.
	\end{equation*}
\end{corollary}
\begin{proof}
	By combining (\ref{RALMRA:eqn5}), and Corollaries \ref{RALMRA:cor1} and \ref{RALMRA:cor2}, we have that
	\begin{equation*}
	\begin{aligned}
	\|\mathcal{A}-\widehat{\mathcal{A}}\|_F
	&\leq \sum_{n=1}^N{\rm rank}(\mathbf{A}_{(n)})\left(2\cdot(1+\lambda_n)\Delta_{\mu_n+1}(\mathbf{C}_n',1)+\phi_n\|\mathcal{A}\|_F^2\right)
	\end{aligned}
	\end{equation*}
	holds with a probability at least the amount in (\ref{RALMRA:eqn19}).
	
	Since $\mathbf{S}_n\in\mathbb{R}^{I_1\dots I_{n-1}I_{n+1}\dots I_N\times T_n}\sim{\rm RANDSAMPLE}(T_n,\mathbf{p}_n)$, it follows from \cite{aizenbud2016randomized} that $\|\mathbf{S}_n\|_2\leq \sqrt{I_n}$. From Lemma \ref{RALMRA:lem4}, we have
	\begin{equation*}
	\Delta_{\mu_n+1}(\mathbf{C}_n',1)\leq I_n\Delta_{\mu_n+1}(\mathbf{A}_{(n)},1),
	\end{equation*}
	which implies $\Delta_{\mu_n+1}(\mathbf{C}_n',1)\leq I_n\Delta_{\mu_n+1}(\mathbf{A}_{(n)},1)$. Hence, this theorem is completely proved.
	\end{proof}
\subsection{Analyzing Algorithm \ref{RALMRA:alg2}}
\label{RALMRA:sec4:sub4}
For each $n$, when $\mathbf{Q}_n$ is obtained from Algorithm \ref{RALMRA:alg2}, we assume that $\mathcal{B}=\mathcal{A}\times_1\mathbf{Q}_1^\top\dots\times_{n-1}\mathbf{Q}_{n-1}$,
which implies that $\mathbf{B}_{(n)}=\mathbf{A}_{(n)}(\mathbf{Q}_1\otimes \dots \otimes \mathbf{Q}_{n-1}\otimes \mathbf{I}_{I_{n-1}}\otimes\dots \otimes \mathbf{I}_{I_{N}})$. Hence we can rewrite (\ref{RALMRA:eqn26}) as
\begin{equation*}
\begin{aligned}
&\left\|\mathcal{A}-\mathcal{A}\times_1\left({\bf Q}_1{\bf Q}_1^\top\right)\times_2\left({\bf Q}_2{\bf Q}_2^\top\right)\dots\times_N\left({\bf Q}_N{\bf Q}_N^\top\right)\right\|_F^2\\
&\leq\sum_{n=1}^N\|\mathbf{B}_{(n)}-{\bf Q}_n{\bf Q}_n^\top\mathbf{B}_{(n)}\|_F^2.
\end{aligned}
\end{equation*}

When we replace the matrix $\mathbf{A}_{(n)}$ in Theorem \ref{RALMRA:thm5} as $\mathbf{B}_{(n)}$, the following remark establishes the upper bound of $\|\mathbf{B}_{(n)}-\mathbf{Q}_n\mathbf{Q}_n^\top\mathbf{B}_{(n)}\|_F$.
\begin{remark}
	For a given $n$, let $\mu_n$, $I_n$ and $K$ are positive integers such that $\mu_n< \mu_n+K< I_n'''$. For $\beta>1$ and $\gamma>1$, suppose that
	\begin{equation}\label{RALMRA:eqn20}
	\begin{aligned}
	1-&\frac{1}{4(\gamma^2-1)\sqrt{\pi \max\{I_n'''-\mu_n,\mu_n+K\}\gamma^2}}\left(\frac{2\gamma^2}{e^{\gamma^2-1}}\right)^2\\
	&-
	\frac{1}{\sqrt{2\pi(K+1)}}\left(\frac{e}{K+1}\right)^{K+1}-\frac{1}{4(\gamma^2-1)\sqrt{\pi I_n\gamma^2}}\left(\frac{2\gamma^2}{e^{\gamma^2-1}}\right)^2
	\end{aligned}
	\end{equation}
	is nonnegative. For a given tensor $\mathcal{A}\in\mathbb{R}^{I_1\times I_2\times\dots \times  I_N}$ and a integer $q\geq 1$, the orthonormal matrix $\mathbf{Q}_n$ is obtained by Algorithm {\rm\ref{RALMRA:alg2}}. Then
	\begin{equation*}
	\|\mathbf{B}_{(n)}-\mathbf{Q}_n\mathbf{Q}_n^\top\mathbf{B}_{(n)}\|_F
	\leq 2\cdot\left(\lambda_n\Delta_{\mu_n+1}(\mathbf{B}_{(n)},q)
	+\Delta_{\mu_n+1}(\mathbf{B}_{(n)})\right)
	\end{equation*}
	holds with a probability at least the amount in {\rm (\ref{RALMRA:eqn20})}, where
	\begin{equation*}
	\lambda_n=\frac{(\sqrt{\max\{I_n'''-\mu_n,\mu_n+K\}}+\sqrt{I_n})\sqrt{2(\mu_n+K)}\gamma\beta}
{\sigma_{\mu_n}(\mathbf{B}_{(n)})^{2q-1}}.
	\end{equation*}

Since all the $\mathbf{Q}_j\ (j=1,2,\dots,n-1)$ are orthonormal, the matrix $\mathbf{Q}_1\otimes \dots \otimes \mathbf{Q}_{n-1}\otimes \mathbf{I}_{I_{n-1}}\otimes\dots \otimes \mathbf{I}_{I_{N}}$ is also orthonormal. According to Lemma 3.9 in {\rm\cite{woolfe2008a}} and Corollary 5.1 in {\rm\cite{che2021randomized}}, we have that $\Delta_{\mu_n+1}(\mathbf{B}_{(n)})\leq \Delta_{\mu_n+1}(\mathbf{A}_{(n)})$ and $\Delta_{\mu_n+1}(\mathbf{B}_{(n)},q)\leq \Delta_{\mu_n+1}(\mathbf{A}_{(n)},q)$. However, the lower bound of $\sigma_{\mu_n}(\mathbf{B}_{(n)})$ is unclear. Hence, the framework of analyzing Algorithm {\rm\ref{RALMRA:alg1}} is not suitable for considering Algorithm {\rm\ref{RALMRA:alg2}}.
\end{remark}

Note that
\begin{equation*}
\begin{aligned}
&\|\mathbf{B}_{(n)}-{\bf Q}_n{\bf Q}_n^\top\mathbf{B}_{(n)}\|_F^2\\
&\leq {\rm rank}(\mathbf{B}_{(n)})
\|(\mathbf{B}_{(n)}-{\bf Q}_n{\bf Q}_n^\top\mathbf{B}_{(n)})(\mathbf{B}_{(n)}-{\bf Q}_n{\bf Q}_n^\top\mathbf{B}_{(n)})^\top\|_F\\
&\leq{\rm rank}(\mathbf{A}_{(n)})
\|(\mathbf{B}_{(n)}-{\bf Q}_n{\bf Q}_n^\top\mathbf{B}_{(n)})(\mathbf{B}_{(n)}-{\bf Q}_n{\bf Q}_n^\top\mathbf{B}_{(n)})^\top\|_F\\
&={\rm rank}(\mathbf{A}_{(n)})
\|(\mathbf{B}_{(n)}\mathbf{B}_{(n)}^\top-{\bf Q}_n{\bf Q}_n^\top\mathbf{B}_{(n)}\mathbf{B}_{(n)}^\top)(\mathbf{I}_{I_n}-{\bf Q}_n{\bf Q}_n^\top)\|_F\\
&\leq {\rm rank}(\mathbf{A}_{(n)})
\|\mathbf{B}_{(n)}\mathbf{B}_{(n)}^\top-{\bf Q}_n{\bf Q}_n^\top\mathbf{B}_{(n)}\mathbf{B}_{(n)}^\top\|_F,
\end{aligned}
\end{equation*}
where the first inequality holds for the fact that ${\rm rank}(\mathbf{B}_{(n)})\leq {\rm rank}(\mathbf{A}_{(n)})$ and the second inequality holds for the fact that $\|\mathbf{I}_{I_n}-{\bf Q}_n{\bf Q}_n^\top\|_2\leq 1$.

When we replace the matrix $\mathbf{C}_n'\mathbf{C}_n'^\top$ in Corollary \ref{RALMRA:cor2} as $\mathbf{B}_{(n)}\mathbf{B}_{(n)}^\top$, the following corollary establishes the upper bound of $\|\mathbf{B}_{(n)}\mathbf{B}_{(n)}^\top-\mathbf{Q}_n\mathbf{Q}_n^\top\mathbf{B}_{(n)}\mathbf{B}_{(n)}^\top\|_F$.
\begin{corollary}\label{RALMRA:cor4}
	For a given $n$, let $\mu_n$, $I_n$, and $K$ are positive integers such that $\mu_n< \mu_n+K< I_n'''$. For $\beta>1$ and $\gamma>1$, suppose that
	\begin{equation}\label{RALMRA:eqn21}
	\begin{aligned}
	1-&\frac{1}{4(\gamma^2-1)\sqrt{\pi \max\{I_n'''-\mu_n,\mu_n+K\}\gamma^2}}\left(\frac{2\gamma^2}{e^{\gamma^2-1}}\right)^2\\
	&-
	\frac{1}{\sqrt{2\pi(K+1)}}\left(\frac{e}{K+1}\right)^{K+1}-\frac{1}{4(\gamma^2-1)\sqrt{\pi I_n\gamma^2}}\left(\frac{2\gamma^2}{e^{\gamma^2-1}}\right)^2
	\end{aligned}
	\end{equation}
	is nonnegative. For a given tensor $\mathcal{A}\in\mathbb{R}^{I_1\times I_2\times\dots \times  I_N}$ and $q= 1$, the orthonormal matrix $\mathbf{Q}_n$ is obtained by Algorithm {\rm\ref{RALMRA:alg2}}. Then
	\begin{equation*}
	\begin{aligned}
	\|\mathbf{B}_{(n)}\mathbf{B}_{(n)}^\top
-\mathbf{Q}_n\mathbf{Q}_n^\top\mathbf{B}_{(n)}\mathbf{B}_{(n)}^\top\|_F
	&\leq 2\cdot(1+\lambda_n)\Delta_{\mu_n+1}(\mathbf{B}_{(n)},1)
	\end{aligned}
	\end{equation*}
	holds with a probability at least the amount in {\rm (\ref{RALMRA:eqn21})}, where
	\begin{equation*}
	\lambda_n=(\sqrt{\max\{I_n'''-\mu_n,\mu_n+K\}}+\sqrt{I_n})\sqrt{2(\mu_n+K)}\gamma\beta.
	\end{equation*}
\end{corollary}

We illustrate the relationship between the singular values of $\mathbf{A}_{(n)}$ and $\mathbf{B}_{(n)}$ in the following lemma.
\begin{lemma}\label{RALMRA:lem8}
	For each $n$, assume that $\mu_n< I_n'''$. Then we have
        \begin{equation*}
	\sum_{k=\mu_n+1}^{I_n'''}\sigma_k(\mathbf{B}_{(n)})^2\leq
	\sum_{k=\mu_n+1}^{I_n'}\sigma_k(\mathbf{A}_{(n)})^2.
	\end{equation*}
	\end{lemma}
\begin{proof}
Note that $\mathbf{B}_{(n)}=\mathbf{A}_{(n)}(\mathbf{Q}_1\otimes \dots \otimes \mathbf{Q}_{n-1}\otimes \mathbf{I}_{I_{n-1}}\otimes\dots \otimes \mathbf{I}_{I_{N}})$, then for each $k=1,2,\dots,I_n'''-1,I_n'''$, we have
	\begin{equation*}
	\sigma_k(\mathbf{B}_{(n)})\leq \sigma_k(\mathbf{A}_{(n)}).
	\end{equation*}
	By combining Lemma 3.9 in \cite{woolfe2008a} and the fact that $$\|\mathbf{Q}_1\otimes \dots \otimes \mathbf{Q}_{n-1}\otimes \mathbf{I}_{I_{n-1}}\otimes\dots \otimes \mathbf{I}_{I_{N}}\|_2= 1,$$
	the proof is completed.
	\end{proof}

\begin{corollary}
	For each $n$, let $\mu_n$, $I_n$ and $K$ are positive integers such that $\mu_n\leq \mu_n+K\leq I_n'''$. For $\beta>1$, $\gamma>1$ and $0<\delta_n<1$, suppose that
	\begin{equation}\label{RALMRA:eqn22}
	\begin{aligned}
	1-&\sum_{n=1}^N\left(\frac{1}{4(\gamma^2-1)\sqrt{\pi \max\{I_n'''-\mu_n,\mu_n+K\}\gamma^2}}\left(\frac{2\gamma^2}{e^{\gamma^2-1}}\right)^2\right.\\
	&\left.+
	\frac{1}{\sqrt{2\pi(K+1)}}\left(\frac{e}{K+1}\right)^{K+1}+\frac{1}{4(\gamma^2-1)\sqrt{\pi I_n\gamma^2}}\left(\frac{2\gamma^2}{e^{\gamma^2-1}}\right)^2\right)
	\end{aligned}
	\end{equation}
	is nonnegative. For a given tensor $\mathcal{A}\in\mathbb{R}^{I_1\times I_2\times\dots \times  I_N}$ and $q= 1$, the orthonormal matrices $\mathbf{Q}_n\ (n=1,2,\dots,N)$ are obtained by Algorithm {\rm\ref{RALMRA:alg2}}. Then
	\begin{equation*}
	\begin{aligned}
	\|\mathcal{A}-\widehat{\mathcal{A}}\|_F^2
	&\leq 2\cdot \sum_{n=1}^N{\rm rank}(\mathbf{A}_{(n)})(1+\lambda_n)I_n\Delta_{\mu_n+1}(\mathbf{A}_{(n)},1)
	\end{aligned}
	\end{equation*}
	holds with probability at least the amount in {\rm (\ref{RALMRA:eqn21})}, where
	\begin{equation*}
	\lambda_n=(\sqrt{\max\{I_n'''-\mu_n,\mu_n+K\}}+\sqrt{I_n})\sqrt{2(\mu_n+K)}\gamma\beta.
	\end{equation*}
\end{corollary}
\begin{proof}
This result is derived by combining (\ref{RALMRA:eqn23}), Corollary \ref{RALMRA:cor4} and Lemma \ref{RALMRA:lem8}.
\end{proof}

\subsection{Analyzing Algorithm \ref{RALMRA:alg5}}

For each $n$, when $\mathbf{Q}_n$ is obtained from Algorithm \ref{RALMRA:alg5}, we have
\begin{equation*}
\begin{aligned}
&\|\mathbf{B}_{(n)}-\mathbf{Q}_n\mathbf{Q}_n^\top\mathbf{B}_{(n)}\|_F^2\\
&\leq {\rm rank}(\mathbf{B}_{(n)})\|(\mathbf{B}_{(n)}-\mathbf{Q}_n\mathbf{Q}_n^\top\mathbf{B}_{(n)})(\mathbf{B}_{(n)}-\mathbf{Q}_n\mathbf{Q}_n^\top\mathbf{B}_{(n)})^\top\|_F\\
&={\rm rank}(\mathbf{B}_{(n)})\|\mathbf{B}_{(n)}\mathbf{B}_{(n)}^\top-\mathbf{Q}_n\mathbf{Q}_n^\top\mathbf{B}_{(n)}\mathbf{B}_{(n)}^\top\|_F\\
&\leq {\rm rank}(\mathbf{B}_{(n)})\left(\|\mathbf{B}_{(n)}\mathbf{B}_{(n)}^\top-\mathbf{C}_n\mathbf{C}_n'^\top\|_F+\|\mathbf{C}_n'\mathbf{C}_n'^\top-\mathbf{Q}_n\mathbf{Q}_n^\top\mathbf{C}_n'\mathbf{C}_n'^\top\|_F\right),
\end{aligned}
\end{equation*}
with $\mathbf{B}_{(n)}=\mathbf{A}_{(n)}(\mathbf{Q}_1\otimes \dots \otimes \mathbf{Q}_{n-1}\otimes \mathbf{I}_{I_{n-1}}\otimes\dots \otimes \mathbf{I}_{I_{N}})$ and $\mathbf{C}_n'=\mathbf{A}_{(n)}\mathbf{S}_n$, where $\mathcal{B}=\mathcal{A}\times_1\mathbf{Q}_1^\top\dots\times_{n-1}\mathbf{Q}_{n-1}$ and $\mathbf{S}_n\in\mathbb{R}^{\mu_1\dots \mu_{n-1}I_{n+1}\dots I_N\times T_n}\sim{\rm RANDSAMPLE}(T_n,\mathbf{p}_n)$.

Similar to Corollary \ref{RALMRA:cor1}, the upper bound for $\|\mathbf{B}_{(n)}\mathbf{B}_{(n)}^\top-\mathbf{C}_n\mathbf{C}_n'^\top\|_F$ is given in the following corollary.
\begin{corollary}\label{RALMRA:cor1add}
	For a given $n$, let $\mathbf{S}_n\in\mathbb{R}^{\mu_1\dots \mu_{n-1}I_{n+1}\dots I_N\times T_n}\sim{\rm RANDSAMPLE}(T_n,\mathbf{p}_n)$. Suppose that $\mathbf{B}_{(n)}\in\mathbb{R}^{\mu_n\times \mu_1\dots I_{n-1}I_{n+1}\dots I_N}$ and $\mathbf{C}_n'=\mathbf{B}_{(n)}\mathbf{S}_n$. If the probabilities $\mathbf{p}_n$ are the nearly optimal probabilities, then for any $\delta_n\in(0,1)$, there exists $\eta_n=1+\sqrt{(8/\delta_n)\log(1/\delta_n)}$ such that
	\begin{equation*}
	\|\mathbf{B}_{(n)}\mathbf{B}_{(n)}^\top-\mathbf{C}_n'\mathbf{C}_n'^\top\|_F\leq\frac{\eta_n}{\sqrt{\beta T_n}}\|\mathbf{B}_{(n)}\|_F^2
	\end{equation*}
	with a probability at least $1-\delta_n$. If the probabilities $\mathbf{p}_n$ are the uniform probabilities, then for any $\delta_n\in(0,1)$, there exists $\gamma_n\leq 1+\frac{I_n''''}{\sqrt{T_n}}\sqrt{8\log(1/\delta_n)}\|\mathbf{B}_{(n)}\|_F^2$ such that
	\begin{equation*}
	\|\mathbf{B}_{(n)}\mathbf{B}_{(n)}^\top-\mathbf{C}_n'\mathbf{C}_n'^\top\|_F\leq\left(1+\sqrt{\frac{I_n''''}{T_n}}+\frac{I_n''''}{\sqrt{T_n}}\sqrt{8\log(1/\delta_n)}\right)\|\mathbf{B}_{(n)}\|_F^2
	\end{equation*}
	with a probability at least $1-\delta_n$.
	\end{corollary}

For $N$ given positive integers $\{\mu_1,\mu_2,\dots,\mu_N\}$ with $\mu_n<I_n$, suppose that the optimal solution $\{\mathbf{Q}_1,\mathbf{Q}_2,\dots,\mathbf{Q}_N\}$ is obtained from Algorithm \ref{RALMRA:alg5} with $q=1$. The upper bound for $\|\mathcal{A}-\widehat{\mathcal{A}}\|_F$ is obtained by combining Corollaries \ref{RALMRA:cor2} and \ref{RALMRA:cor1add}, and the fact that $\sigma_k(\mathbf{B}_{(n)})\leq \sigma_k(\mathbf{A}_{(n)})$ with $k=1,2,\dots,I_n'''-1,I_n'''$.
\begin{corollary}\label{RALMRA:cor3add}
	For each $n$, let $\mu_n$, $I_n$, $T_n$ and $K$ are positive integers such that $\mu_n< \mu_n+K< \min\{I_n'',I_n''''\}$. For $\beta>1$, $\gamma>1$ and $0<\delta_n<1$, suppose that
	\begin{equation}\label{RALMRA:eqn19add}
	\begin{aligned}
	1-&\sum_{n=1}^N\left(\frac{1}{4(\gamma^2-1)\sqrt{\pi \max\{I_n''-\mu_n,\mu_n+K\}\gamma^2}}\left(\frac{2\gamma^2}{e^{\gamma^2-1}}\right)^2\right.\\
	&\left.+
	\frac{1}{\sqrt{2\pi(K+1)}}\left(\frac{e}{K+1}\right)^{K+1}+\frac{1}{4(\gamma^2-1)\sqrt{\pi I_n\gamma^2}}\left(\frac{2\gamma^2}{e^{\gamma^2-1}}\right)^2+\delta_n\right)
	\end{aligned}
	\end{equation}
	is nonnegative. For a given tensor $\mathcal{A}\in\mathbb{R}^{I_1\times I_2\times\dots \times  I_N}$ and $q= 1$, the orthonormal matrices $\mathbf{Q}_n\ (n=1,2,\dots,N)$ are obtained by Algorithm {\rm\ref{RALMRA:alg5}}. Then
	\begin{equation*}
	\begin{aligned}
	\|\mathcal{A}-\widehat{\mathcal{A}}\|_F^2
	&\leq \sum_{n=1}^N{\rm rank}(\mathbf{A}_{(n)})\left(2\cdot(1+\lambda_n)I_n\Delta_{\mu_n+1}(\mathbf{A}_{(n)},1)+\phi_n\|\mathcal{A}\|_F^2\right)
	\end{aligned}
	\end{equation*}
	holds with a probability at least the amount in {\rm (\ref{RALMRA:eqn19add})}, where
	\begin{equation*}
	\lambda_n=(\sqrt{\max\{I_n''-\mu_n,\mu_n+K\}}+\sqrt{I_n})\sqrt{2(\mu_n+K)}\gamma\beta.
	\end{equation*}
	For the nearly optimal probabilities, there exists $\eta_n=1+\sqrt{(8/\delta_n)\log(1/\delta_n)}$ such that
	\begin{equation*}
	\phi_n=\frac{\eta_n}{\sqrt{\beta T_n}},
	\end{equation*}
	and for the uniform probabilities, there exists $\gamma_n\leq 1+\frac{I_n''''}{\sqrt{T_n}}\sqrt{8\log(1/\delta_n)}\|\mathbf{A}_{(n)}\|_F^2$ such that
	\begin{equation*}
	\phi_n=1+\sqrt{\frac{I_n''''}{T_n}}+\frac{I_n''''}{\sqrt{T_n}}\sqrt{8\log(1/\delta_n)}.
	\end{equation*}
\end{corollary}
\begin{remark}
The theoretical analysis for Algorithms {\rm\ref{RALMRA:alg4}}, {\rm\ref{RALMRA:alg2}} and {\rm\ref{RALMRA:alg5}} with $q>1$ is unclear and will be studied in the future.
\end{remark}

\subsection{Related forms for Algorithms \ref{RALMRA:alg1} and \ref{RALMRA:alg2}}

One difference between Algorithms \ref{RALMRA:alg1} and \ref{RALMRA:alg2} is that for each $n$, the tensor $\mathcal{A}$ in Algorithm \ref{RALMRA:alg1} remains unchanged, but the tensor $\mathcal{A}$ in Algorithm \ref{RALMRA:alg2} is updated by $\mathcal{A}=\mathcal{A}\times_n\mathbf{Q}_n^\top$. Another difference is that for Algorithm \ref{RALMRA:alg1}, after obtaining all the matrices $\mathbf{Q}_n$, we need to compute the core tensor $\mathcal{G}$ as $\mathcal{G}=\mathcal{A}\times_1\mathbf{Q}_1^\top\times_2\mathbf{Q}_2^\top\dots\times_N\mathbf{Q}_N^\top$, and for Algorithm \ref{RALMRA:alg2}, the core tensor $\mathcal{G}$ is directly obtained after finding all the matrices $\mathbf{Q}_n$.

Algorithms \ref{RALMRA:alg6} and \ref{RALMRA:alg7} are obtained by computing the matrix $\mathbf{Q}_n$ in Algorithms \ref{RALMRA:alg1} and \ref{RALMRA:alg2} as this way: a) to obtain the thin QR factorization of $\mathbf{C}_n$ as $\mathbf{C}_n=\mathbf{P}_n\mathbf{R}$, where $\mathbf{P}_n\in\mathbb{R}^{I_n\times (\mu_n+K)}$ is orthonormal and $\mathbf{R}\in\mathbb{R}^{(\mu_n+K)\times (\mu_n+K)}$ is upper triangular; b) to form $\mathbf{C}_n'=\mathbf{P}_n^\top\mathbf{A}_{(n)}$; c) to compute an orthonormal matrix $\mathbf{U}_{n}\in\mathbb{R}^{I_n\times \mu_n}$ such that its columns are the left singular vectors corresponding to the top $\mu_n$ singular values of $\mathbf{C}_n'$; and d) to form $\mathbf{Q}_n=\mathbf{P}_n\mathbf{U}_{n}$.
\begin{algorithm}[htb]
	\caption{The related form for Algorithm \ref{RALMRA:alg1}}
	\label{RALMRA:alg6}
	\begin{algorithmic}[1]
		\STATEx {\bf Input}: A tensor $\mathcal{A}\in \mathbb{R}^{I_1\times I_2\times\dots\times I_N}$, the desired multilinear rank $\{\mu_1,\mu_2,\dots,\mu_N\}$, the oversampling parameter $K$ and an integer $q\geq 1$.
		\STATEx {\bf Output}: $N$ orthonormal matrices $\mathbf{Q}_n$ and a core tensor $\mathcal{G}$ such that $\mathcal{A}\approx\mathcal{G}\times_1\mathbf{Q}_1\times_2\mathbf{Q}_2\dots\times_N\mathbf{Q}_N$.
		\STATE For $n=1,2,\dots,N$
		\STATE Generate a standard Gaussian matrix $\mathbf{G}\in\mathbb{R}^{I_n \times (\mu_n+K)}$.
		\STATE Arrange the mode-$n$ unfolding $\mathbf{A}_{(n)}$ of $\mathcal{A}$.
            \STATE Compute $\mathbf{C}_n=(\mathbf{A}_{(n)}\mathbf{A}_{(n)}^\top)^q\mathbf{G}_n$.
		\STATE Compute the thin QR factorization of $\mathbf{C}_n$ as $\mathbf{C}_n=\mathbf{P}_n\mathbf{R}$.
            \STATE Form $\mathbf{C}_n'=\mathbf{P}_n^\top\mathbf{A}_{(n)}$ and let $\mathbf{U}_{n}$ be the left singular vectors corresponding to top $\mu_n$
           singular values of $\mathbf{C}_n'$.
            \STATE Compute $\mathbf{Q}_n=\mathbf{P}_n\mathbf{U}_{n}$.
		\STATE End for
            \STATE Compute the core tensor $\mathcal{G}=\mathcal{A}\times_1\mathbf{Q}_1^\top\times_2\mathbf{Q}_2^\top\dots\times_N\mathbf{Q}_N^\top$.
	\end{algorithmic}
\end{algorithm}

\begin{algorithm}
	\caption{The related form for Algorithm \ref{RALMRA:alg2}}
	\label{RALMRA:alg7}
	\begin{algorithmic}[1]
		\STATEx {\bf Input}: A tensor $\mathcal{A}\in \mathbb{R}^{I_1\times I_2\times\dots\times I_N}$, the desired multilinear rank $\{\mu_1,\mu_2,\dots,\mu_N\}$, the oversampling parameter $K$ and an integer $q\geq 1$.
		\STATEx {\bf Output}: $N$ orthonormal matrices $\mathbf{Q}_n$ and a core tensor $\mathcal{G}$ such that $\mathcal{A}\approx\mathcal{G}\times_1\mathbf{Q}_1\times_2\mathbf{Q}_2\dots\times_N\mathbf{Q}_N$.
		\STATE For $n=1,2,\dots,N$
		\STATE Generate a standard Gaussian matrix $\mathbf{G}\in\mathbb{R}^{I_n \times (\mu_n+K)}$.
		\STATE Arrange the mode-$n$ unfolding $\mathbf{A}_{(n)}$ of $\mathcal{A}$.
            \STATE Compute $\mathbf{C}_n=(\mathbf{A}_{(n)}\mathbf{A}_{(n)}^\top)^q\mathbf{G}_n$.
		\STATE Compute the thin QR factorization of $\mathbf{C}_n$ as $\mathbf{C}_n=\mathbf{P}_n\mathbf{R}$.
            \STATE Form $\mathbf{C}_n'=\mathbf{P}_n^\top\mathbf{A}_{(n)}$ and let $\mathbf{U}_{n}$ be the left singular vectors corresponding to top $\mu_n$
           singular values of $\mathbf{C}_n'$.
            \STATE Compute $\mathbf{Q}_n=\mathbf{P}_n\mathbf{U}_{n}$.
            \STATE Update $\mathcal{A}=\mathcal{A}\times_n\mathbf{Q}_n^\top$.
		\STATE End for
            \STATE Set the core tensor $\mathcal{G}$ as $\mathcal{G}=\mathcal{A}$.
	\end{algorithmic}
\end{algorithm}

%For each $n$, by setting $\mathbf{Q}_n=\mathbf{Q}_n'$ and whether to update the tensor $\mathcal{B}$, Algorithms \ref{RALMRA:alg6} and \ref{RALMRA:alg7} are derived from Algorithm \ref{RALMRA:alg1} or \ref{RALMRA:alg2}. Let  $\widehat{\mathcal{A}}
%=\mathcal{A}\times_1(\mathbf{Q}_1\mathbf{Q}_1^\top)\times_2(\mathbf{Q}_2\mathbf{Q}_2^\top)
%\dots\times_N(\mathbf{Q}_N\mathbf{Q}_N^\top)$. If $\{\mathbf{Q}_1,\mathbf{Q}_2,\dots,\mathbf{Q}_N\}$ is obtained by Algorithm \ref{RALMRA:alg1} or \ref{RALMRA:alg2}, then for each $n$, we have ${\rm rank}(\widehat{\mathbf{A}}_{(n)})\leq \mu_n$. If $\{\mathbf{Q}_1,\mathbf{Q}_2,\dots,\mathbf{Q}_N\}$ is obtained by Algorithm \ref{RALMRA:alg6} or \ref{RALMRA:alg7}, then for each $n$, we have ${\rm rank}(\widehat{\mathbf{A}}_{(n)})\leq \mu_n+K$. In Algorithms \ref{RALMRA:alg1} and \ref{RALMRA:alg2}, the matrix $\mathbf{Q}_n$ is set to $\mathbf{Q}_n=\mathbf{Q}_n'(:,1:\mu_n)$, and in Algorithms \ref{RALMRA:alg6} and \ref{RALMRA:alg7}, the matrix $\mathbf{Q}_n$ is obtained by $\mathbf{Q}_n=\mathbf{Q}_n'$. In Algorithms \ref{RALMRA:alg1} and \ref{RALMRA:alg6}, the tensor $\mathcal{B}$ remains unchanged, and in Algorithms \ref{RALMRA:alg2} and \ref{RALMRA:alg7}, the tensor $\mathcal{B}$ is updated as $\mathcal{B}=\mathcal{B}=\mathcal{B}\times_n\mathbf{Q}_n^\top$. Note that the sizes of $\mathcal{B}$ in these two algorithms are not the same.

For each $n$, we set $J_n'=I_1\dots I_{n-1}I_{n+1}\dots I_N$ and $\mathcal{B}=\mathcal{A}$ in Algorithm \ref{RALMRA:alg6}, and we set $J_n'=\mu_1\dots \mu_{n-1}I_{n+1}\dots I_N$ and $\mathcal{B}=\mathcal{A}\times_1\mathbf{Q}_1^\top\dots\times_{n-1}\mathbf{Q}_{n-1}^\top$ in Algorithm \ref{RALMRA:alg7}. Suppose that $\mathbf{B}_{(n)}=\mathbf{U}_n{\bf \Sigma}_n\mathbf{V}_n^\top$, where $\mathbf{U}_n\in\mathbb{R}^{I_n\times J_n}$ and $\mathbf{V}_n\in\mathbb{R}^{J_n'\times J_n}$ are orthonornal, and the diagonal entries of the diagonal matrix ${\bf \Sigma}_n\in\mathbb{R}^{J_n\times J_n}$ are nonnegative and arranged in descending order with $J_n=\min\{I_n,J_n'\}$. Then $(\mathbf{B}_{(n)}\mathbf{B}_{(n)}^\top)^q=\mathbf{U}_n{\bf \Sigma}_n^{2q}\mathbf{U}_n^\top$. Let $\mathbf{U}_{n,1}=\mathbf{U}_n(:,1:\mu_n)$, $\mathbf{U}_{n,2}=\mathbf{U}_n(:,\mu_n+1:J_n)$, ${\bf \Sigma}_{n,1}={\bf \Sigma}_n(1:\mu_n,1:\mu_n)$, and ${\bf \Sigma}_{n,2}={\bf \Sigma}_n((\mu_n+1):J_n,(\mu_n+1):J_n)$. Define ${\bf S}_{n,1}=\mathbf{U}_{n,1}^\top\mathbf{G}_n$ and ${\bf S}_{n,2}=\mathbf{U}_{n,2}^\top\mathbf{G}_n$.

For each $n$, when the matrix $\mathbf{Q}_n$ is obtained by Algorithm \ref{RALMRA:alg6}. then, from Theorem 3.5 in \cite{gu2015subspace} and Theorem 9 in \cite{zhang2018a}, we have
\begin{equation*}
\begin{aligned}
\|\mathbf{B}_{(n)}-\mathbf{Q}_n\mathbf{Q}_n^\top\mathbf{B}_{(n)}\|_F^2\leq \|{\bf \Sigma}_{n,2}\|_F^2+\tau_{\mu_n}^{4q-1}\|{\bf \Sigma}_{n,2}{\bf S}_{n,2}{\bf S}_{n,1}^\dag\|_F^2,
\end{aligned}
\end{equation*}
where $\tau_{\mu_n}=\sigma_{\mu_n+1}(\mathbf{B}_{(n)})/\sigma_{\mu_n}(\mathbf{B}_{(n)})$. When the matrix $\mathbf{Q}_n$ is obtained by Algorithm \ref{RALMRA:alg7}, by using the same notations in Section \ref{RALMRA:sec4:sub4}, we have
\begin{equation*}
\begin{aligned}
\|\mathbf{B}_{(n)}-\mathbf{Q}_n\mathbf{Q}_n^\top\mathbf{B}_{(n)}\|_F^2\leq \|{\bf \Sigma}_{n,2}\|_F^2+\tau_{\mu_n}^{4q-1}\|{\bf \Sigma}_{n,2}{\bf S}_{n,2}{\bf S}_{n,1}^\dag\|_F^2,
\end{aligned}
\end{equation*}
which implies that
\begin{equation*}
\begin{aligned}
\|\mathbf{B}_{(n)}-\mathbf{Q}_n\mathbf{Q}_n^\top\mathbf{B}_{(n)}\|_F^2\leq \|{\bf \Sigma}_{n,2}\|_F^2+\|{\bf \Sigma}_{n,2}{\bf S}_{n,2}{\bf S}_{n,1}^\dag\|_F^2,
\end{aligned}
\end{equation*}
where $\tau_{\mu_n}=\sigma_{\mu_n+1}(\mathbf{B}_{(n)})/\sigma_{\mu_n}(\mathbf{B}_{(n)})\leq 1$.

It is worth to note that: for Algorithm \ref{RALMRA:alg6}, the singular values of $\mathbf{B}_{(n)}$ are equivalent to that of $\mathbf{A}_{(n)}$, and for Algorithm \ref{RALMRA:alg7}, the relationship between the singular values of $\mathbf{B}_{(n)}$ and $\mathbf{A}_{(n)}$ are in Lemma \ref{RALMRA:lem4}. Since $\mathbf{G}_n$ is a standard Gaussian matrix, ${\bf S}_{n,1}=\mathbf{U}_{n,1}^\top\mathbf{G}_n$ and ${\bf S}_{n,2}=\mathbf{U}_{n,2}^\top\mathbf{G}_n$ are also standard Gaussian matrices. Hence, the upper bounds of $\|{\bf S}_{n,2}\|_2$ and $\|{\bf S}_{n,1}^\dag\|_2$ can be obtained from Lemmas \ref{RALMRA:lem1} and \ref{RALMRA:lem2}. Then, the upper bound of of $\|\mathcal{A}-\widehat{\mathcal{A}}\|_F^2$ can be sequentially obtained by combining (\ref{RALMRA:eqn5}), and Lemmas \ref{RALMRA:lem1} and \ref{RALMRA:lem2}.

On the other hand, suppose that for each $n$, the matrix $\mathbf{Q}_n$ is obtained by Algorithm \ref{RALMRA:alg6} or \ref{RALMRA:alg7}. From Theorem 9.2 in \cite{halko2011finding}, we have
\begin{equation*}
\begin{aligned}
\|\mathbf{B}_{(n)}-{\bf P}_n{\bf P}_n^\top\mathbf{B}_{(n)}\|_F^2
&\leq
\|(\mathbf{B}_{(n)}\mathbf{B}_{(n)}^\top)^q-{\bf P}_n{\bf P}_n^\top(\mathbf{B}_{(n)}\mathbf{B}_{(n)}^\top)^q\|_F^{1/2q}\\
&\leq
(\|{\bf \Sigma}_{n,2}^{2q}\|_F+\|{\bf \Sigma}_{n,2}^{2q}{\bf S}_{n,2}{\bf S}_{n,1}^\dag\|_F)^{1/2q}\\
&\leq
(\|{\bf \Sigma}_{n,2}^{2q}\|_F+\|{\bf \Sigma}_{n,2}^{2q}\|_F\|{\bf S}_{n,2}\|_2\|{\bf S}_{n,1}^\dag\|_2)^{1/2q}.
\end{aligned}
\end{equation*}
Hence, we can obtain the upper bound for $\|\mathcal{A}-\widetilde{\mathcal{A}}\|_F$ by combining (\ref{RALMRA:eqn5}), and Lemmas \ref{RALMRA:lem1} and \ref{RALMRA:lem2}, where $\widetilde{\mathcal{A}}=\mathcal{A}\times_1({\bf P}_1{\bf P}_1^\top)\times_2({\bf P}_2{\bf P}_2^\top)\dots\times_N({\bf P}_N{\bf P}_N^\top)$. By the relationship between $\mathbf{P}_n$ and $\mathbf{Q}_n$, it is clear to see that each term of the multilinear rank of $\widetilde{\mathcal{A}}$ may be larger than that of $\widehat{\mathcal{A}}$.

Suppose that the parameters $K$ and $q$ in Algorithms \ref{RALMRA:alg1}, \ref{RALMRA:alg2}, \ref{RALMRA:alg6} and \ref{RALMRA:alg7} are the same in Section \ref{RALMRA:sec3:3}. By setting $\gamma=0.01$ and $\mu=5,10,\dots,50$, when we apply these algorithms to the tensor $\mathcal{A}$ in (\ref{RALMRA:eqn9}), the results are shown in Figure \ref{RALMRA:fig5}. In terms of RLNE, these algorithms are comparable; and in terms of CPU time, Algorithm \ref{RALMRA:alg2} is faster than Algorithm \ref{RALMRA:alg7}, Algorithm \ref{RALMRA:alg1} is faster than Algorithm \ref{RALMRA:alg6}, and Algorithms \ref{RALMRA:alg2} and \ref{RALMRA:alg7} are much faster than Algorithms \ref{RALMRA:alg1} and \ref{RALMRA:alg6}.
\begin{figure}[htb]
	\centering
	\includegraphics[width=3.8in]{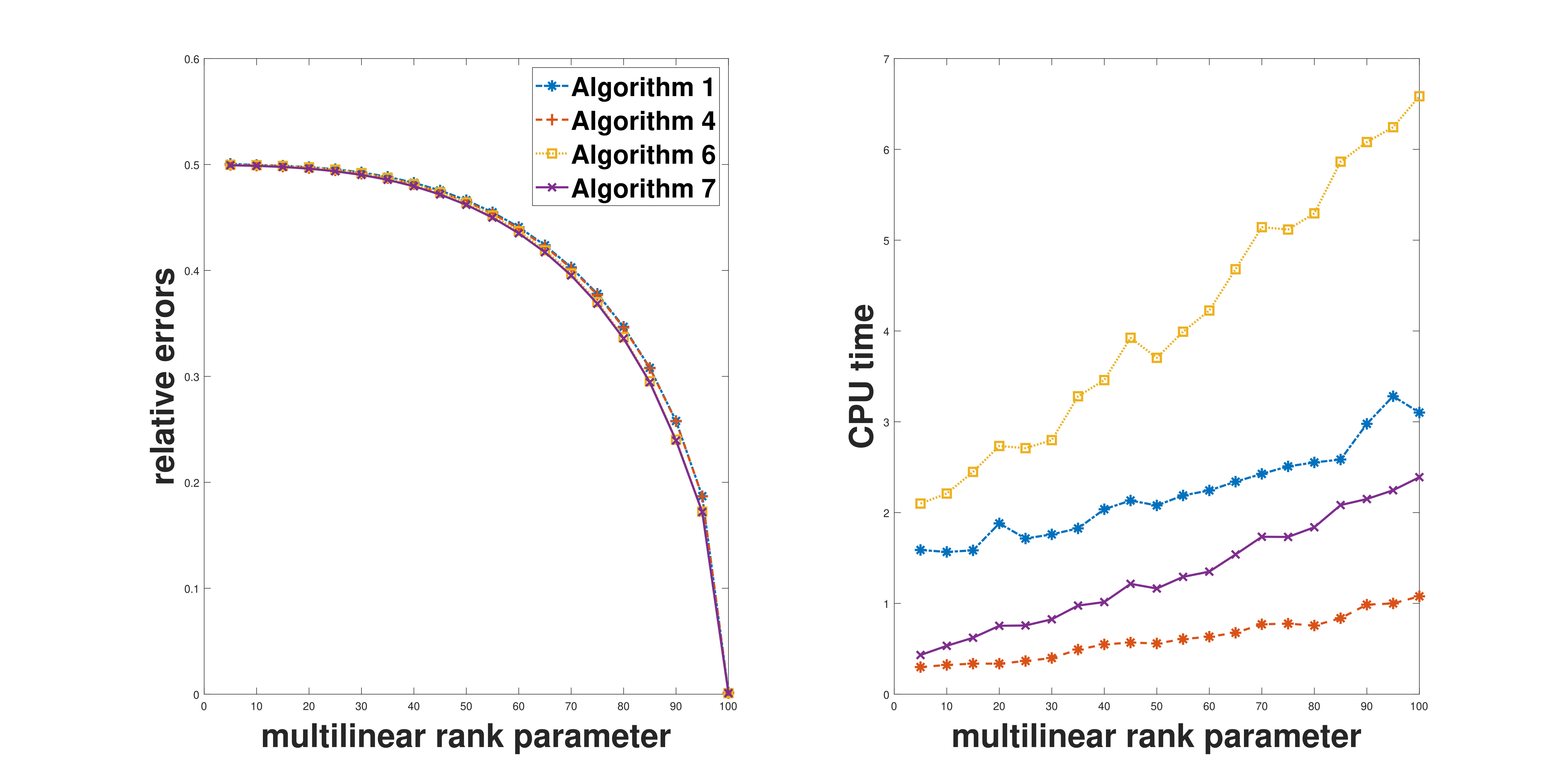}\\
	\caption{Results of applying Algorithms \ref{RALMRA:alg1}, \ref{RALMRA:alg2}, \ref{RALMRA:alg6} and \ref{RALMRA:alg7} to $\mathcal{A}$ with $\gamma=0.01$ and $\mu=5,10,\dots,50$.}\label{RALMRA:fig5}
\end{figure}

\begin{remark}
These two modifications are also suitable for Algorithms {\rm\ref{RALMRA:alg4}} and {\rm\ref{RALMRA:alg5}}. The theoretical results for Tucker-SRFT and Tucker-SRFT-power in {\rm \cite{che2021randomized}} have been derived by changing the way to obtain the unitary matrix $\mathbf{Q}_n\in\mathbb{C}^{I_n\times \mu_n}$ as the same as that in Algorithm {\rm\ref{RALMRA:alg6}} or Algorithm {\rm\ref{RALMRA:alg7}}.
\end{remark}
\section{Numerical examples}
\label{RALMRA:sec6}

In this section, we use the numerical computation software MATLAB to develop computer programs and implement the calculations on a laptop with Intel Core i5-4200M CPU (2.50GHz) and 8GB RAM. We set MATLAB  maxNumCompThreads to 1 and use ``tic" and ``toc" to measure running time when applying all of the algorithms to the test tensors. The unit of CPU time is second. The tensor-matrix product and tensor contraction can be implemented by the functions ``ttm'' and ``ttt'' in the MATLAB Tensor Toolbox \cite{bader2006algorithm,tensortool}.

We compare the efficiency and accuracy of Algorithms \ref{RALMRA:alg2} and \ref{RALMRA:alg5} with HOOI \cite{lathauwer2000on}, T-HOSVD \cite{kolda2019tensor}, ST-HOSVD \cite{vannieuwenhoven2012new}, randomized T-HOSVD \cite{minster2020randomized}, randomized ST-HOSVD \cite{minster2020randomized}, Adap-Tucker \cite{che2019randomized}, Tucker-SVD \cite{che2021randomized} and Tucker-SVD-power \cite{che2020computation}. Note that randomized T-HOSVD, Adap-Tucker, Tucker-SVD, Tucker-SVD-power, and Algorithms \ref{RALMRA:alg2} and \ref{RALMRA:alg5} are implemented by MATLAB codes, HOOI is implemented by the function tucker\_als \cite{tensortool}, T-HOSVD and ST-HOSVD are implemented by the function mlsvd \cite{tensorlab} with different parameters, and randomized ST-HOSVD is implemented by the mlsvd\_rsi \cite{tensorlab}. We also compare Algorithms \ref{RALMRA:alg2} and \ref{RALMRA:alg5} with randomized T-HOSVD and randomized ST-HOSVD under subspace iterations. The number of subspace iterations is set to $q=1$. For all algorithms, the processing order is $\{1,2,\dots,N\}$ and the factor matrices are orthonormal. For Algorithms \ref{RALMRA:alg2} and \ref{RALMRA:alg5}, we set $K=10$ and $q=1$, and for Algorithm \ref{RALMRA:alg5}, we set $T_n={\rm ceil}(\alpha \mu_1\dots \mu_{n-1}I_{n+1}\dots I_N)$ with $\alpha=0.2$.

Let ${\bf Q}_n\in\mathbb{R}^{I_n\times \mu_n}$ be derived using any numerical algorithm to $\mathcal{A}\in\mathbb{R}^{I_1\times I_2\times\dots\times I_N}$. For a given multilinear rank $\{\mu_1,\mu_2,\dots,\mu_N\}$, the relative error (RE) is defined in (\ref{RALMRA:eqn8}), and $\widehat{\mathcal{A}}=\mathcal{A}\times_{1}({\bf Q}_{1}{\bf Q}_{1}^\top)\times_{2}({\bf Q}_{2}{\bf Q}_{2}^\top)\dots\times_{N}({\bf Q}_{N}{\bf Q}_{N}^\top)$ is a low  rank approximation in the Tucker format to $\mathcal{A}$.

\begin{figure}[htb]
	\setlength{\tabcolsep}{4pt}
	\renewcommand\arraystretch{1}
	\centering
	\begin{tabular}{c}
		\includegraphics[width=3.8in]{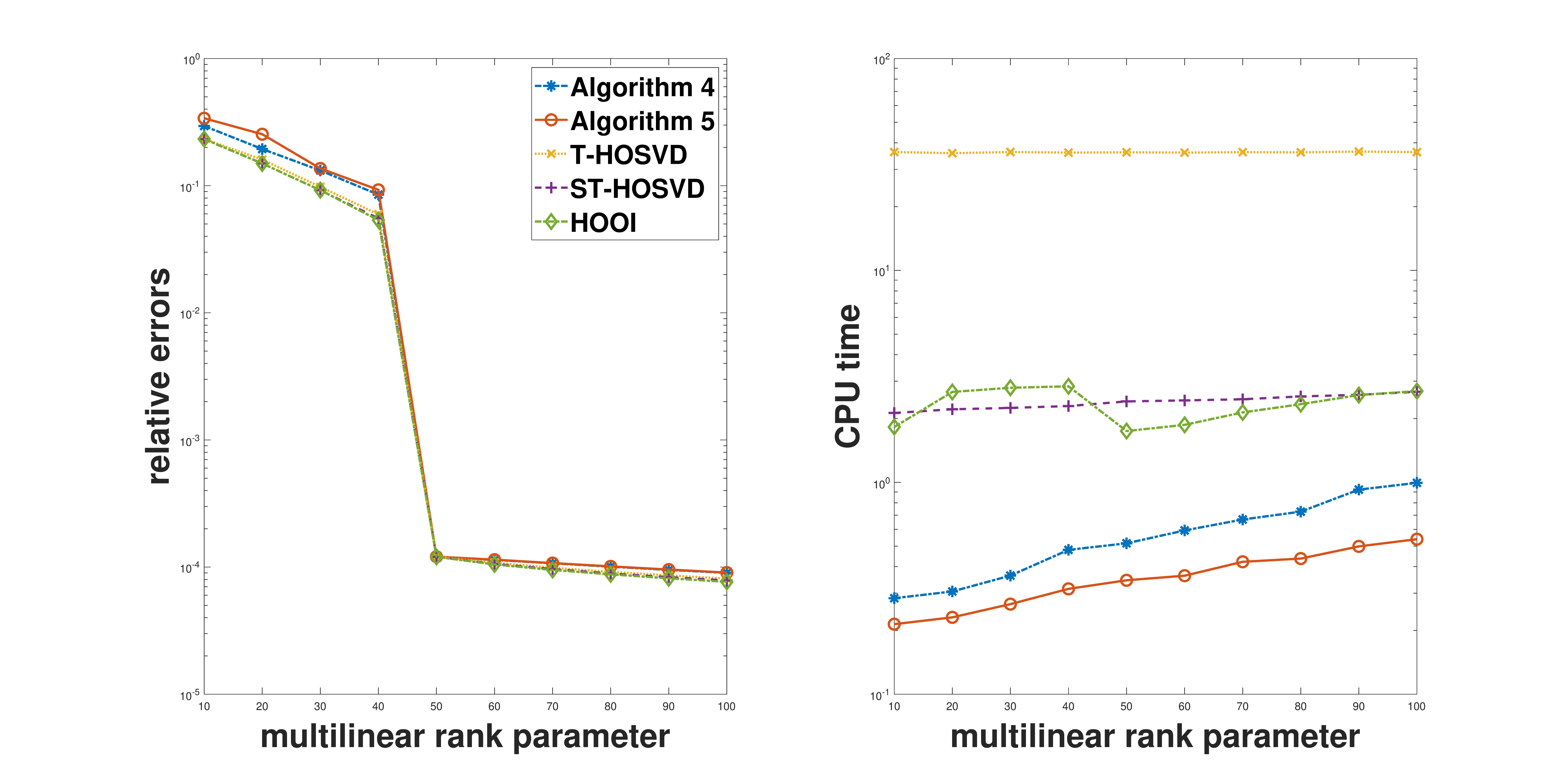}\\
		\includegraphics[width=3.8in]{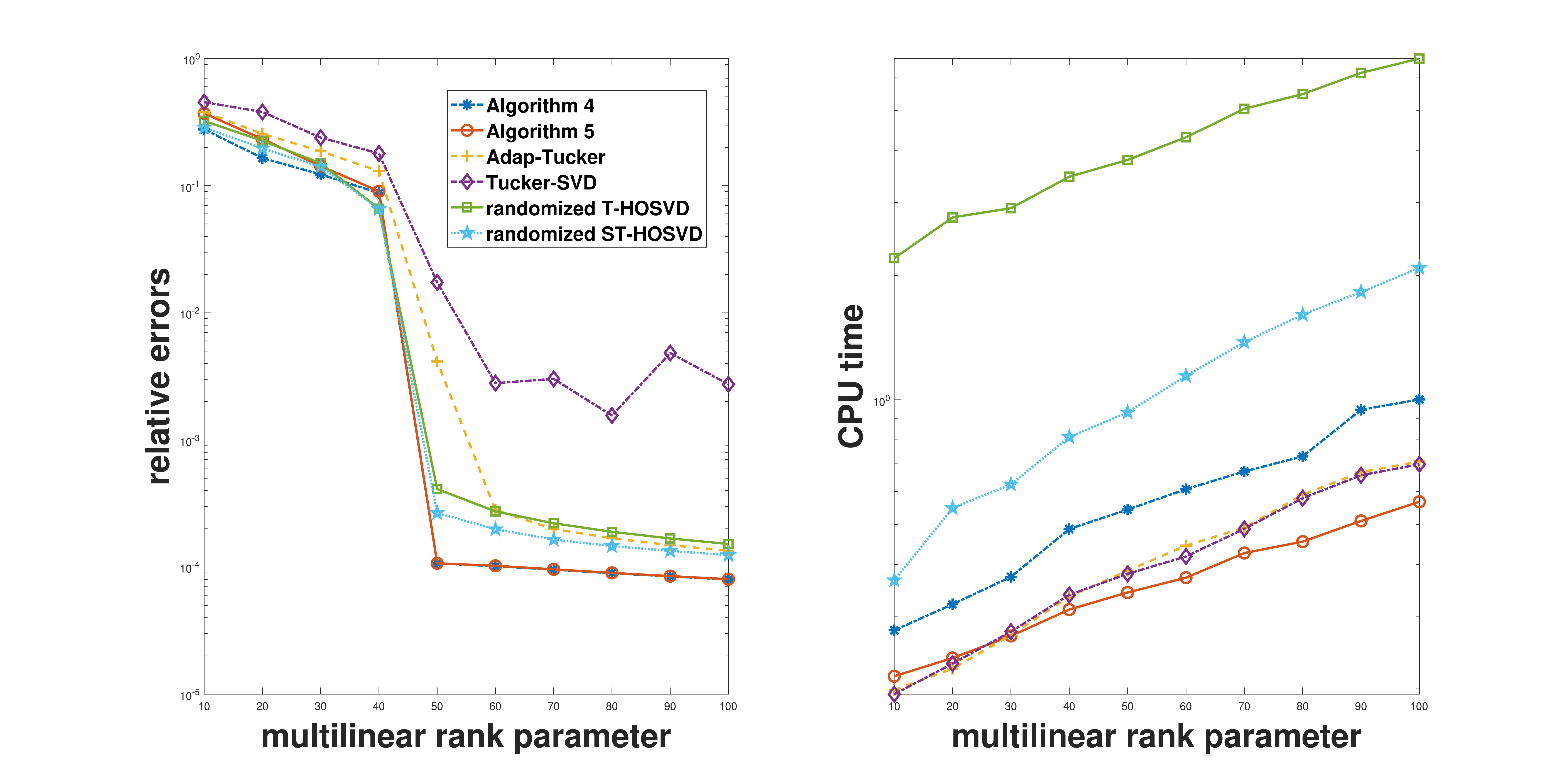}\\
		\includegraphics[width=3.8in]{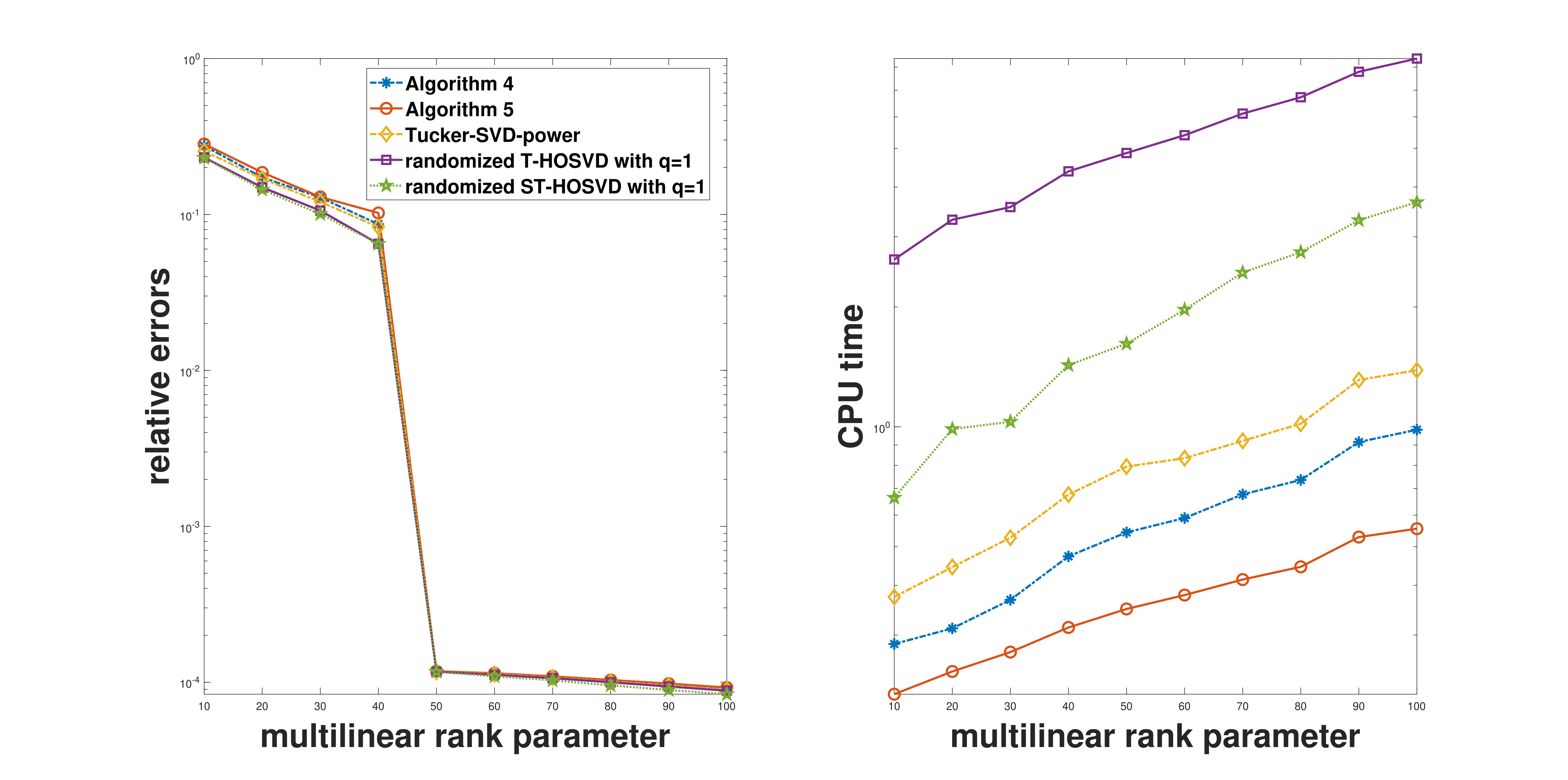}\\
	\end{tabular}
	\caption{Results of applying Algorithms \ref{RALMRA:alg2} and \ref{RALMRA:alg5}, HOOI, T-HOSVD, ST-HOSVD, randomized T-HOSVD, randomized ST-HOSVD, Adap-Tucker, Tucker-SVD and Tucker-SVD-power to the sparse tensor $\mathcal{A}$ with $\mu=5,10,\dots,100$.}\label{RALMRA:fig6}
\end{figure}

\subsection{A sparse tensor}
Consider a sparse tensor $\mathcal{A}\in\mathbb{R}^{600\times 600\times 600}$ such that
\begin{equation*}
\mathcal{A}=\sum_{i=1}^{50}\frac{\gamma}{i^2}\mathbf{x}_i\circ\mathbf{y}_i\circ\mathbf{z}_i+\sum_{i=51}^{600}\frac{1}{i^2}\mathbf{x}_i\circ\mathbf{y}_i\circ\mathbf{z}_i
\end{equation*}
where $\mathbf{x}_i$, $\mathbf{y}_i$ and $\mathbf{z}_i$ are sparse vectors in $\mathbb{R}^{600}$ with only 0.05 sparsity.

In this section, we set $\gamma=1000$. For different multilinear ranks $\{\mu,\mu,\mu\}$, when we apply Algorithms \ref{RALMRA:alg2} and \ref{RALMRA:alg5}, and the existing deterministic and randomized algorithms for finding a low multilinear rank approximation of the sparse tensor $\mathcal{A}$, the values of RE and the CPU time are shown in Figure \ref{RALMRA:fig6}.

In terms of RE, Algorithms \ref{RALMRA:alg2} and \ref{RALMRA:alg5} are similar to HOOI, T-HOSVD, ST-HOSVD, Tucker-SVD-power, randomized T-HOSVD with $q=1$ and randomized ST-HOSVD with $q=1$, and are better than Adap-Tucker, Tucker-SVD, randomized T-HOSVD and randomized ST-HOSVD; in terms of CPU time, Algorithm \ref{RALMRA:alg5} is the fastest one, and Algorithm \ref{RALMRA:alg2} is slower than Adap-Tucker and Tucker-SVD and faster than other algorithms.

\begin{figure}[htb]
	\setlength{\tabcolsep}{4pt}
	\renewcommand\arraystretch{1}
	\centering
	\begin{tabular}{c}
		\includegraphics[width=3.8in]{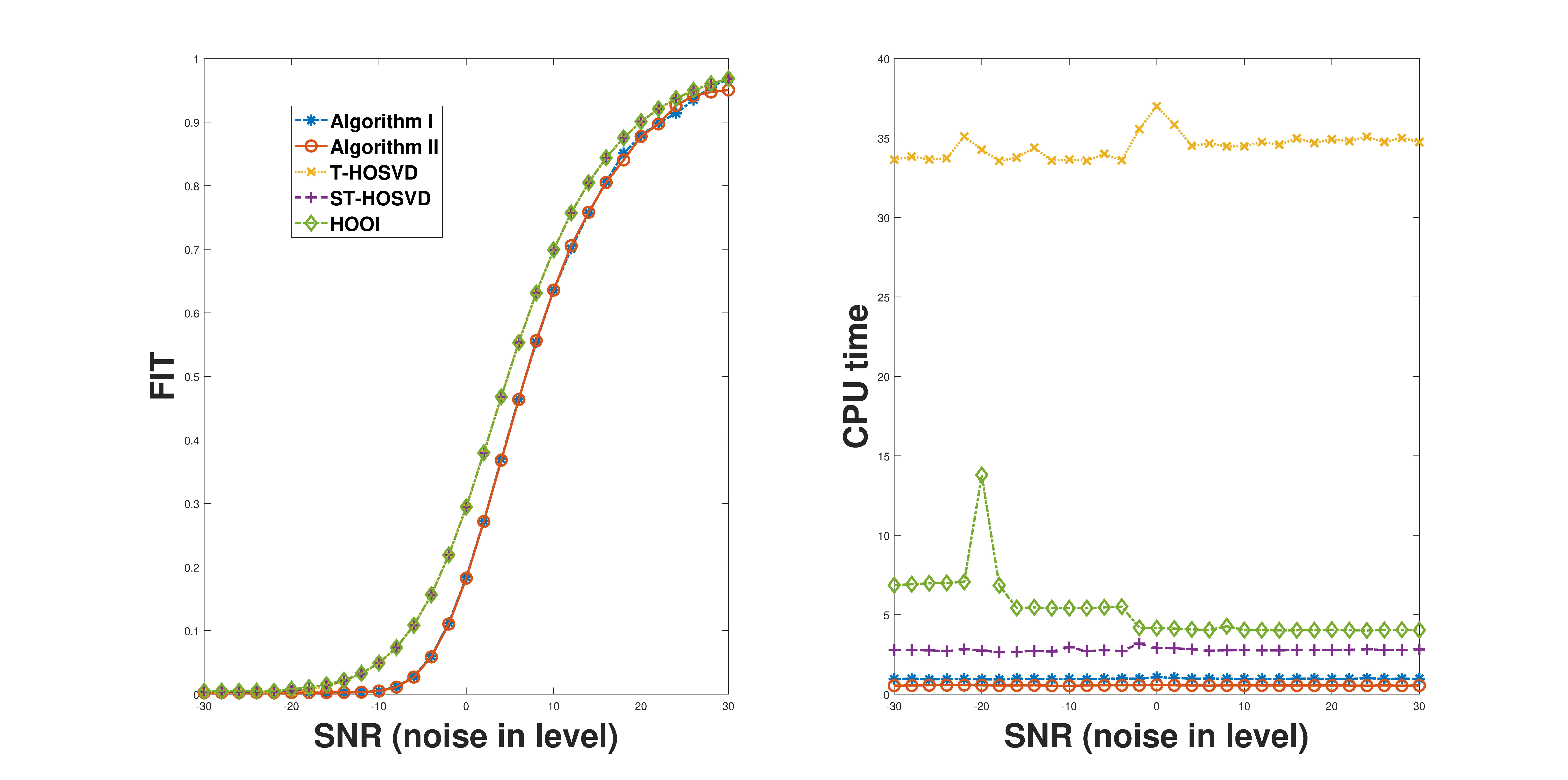}\\
		\includegraphics[width=3.8in]{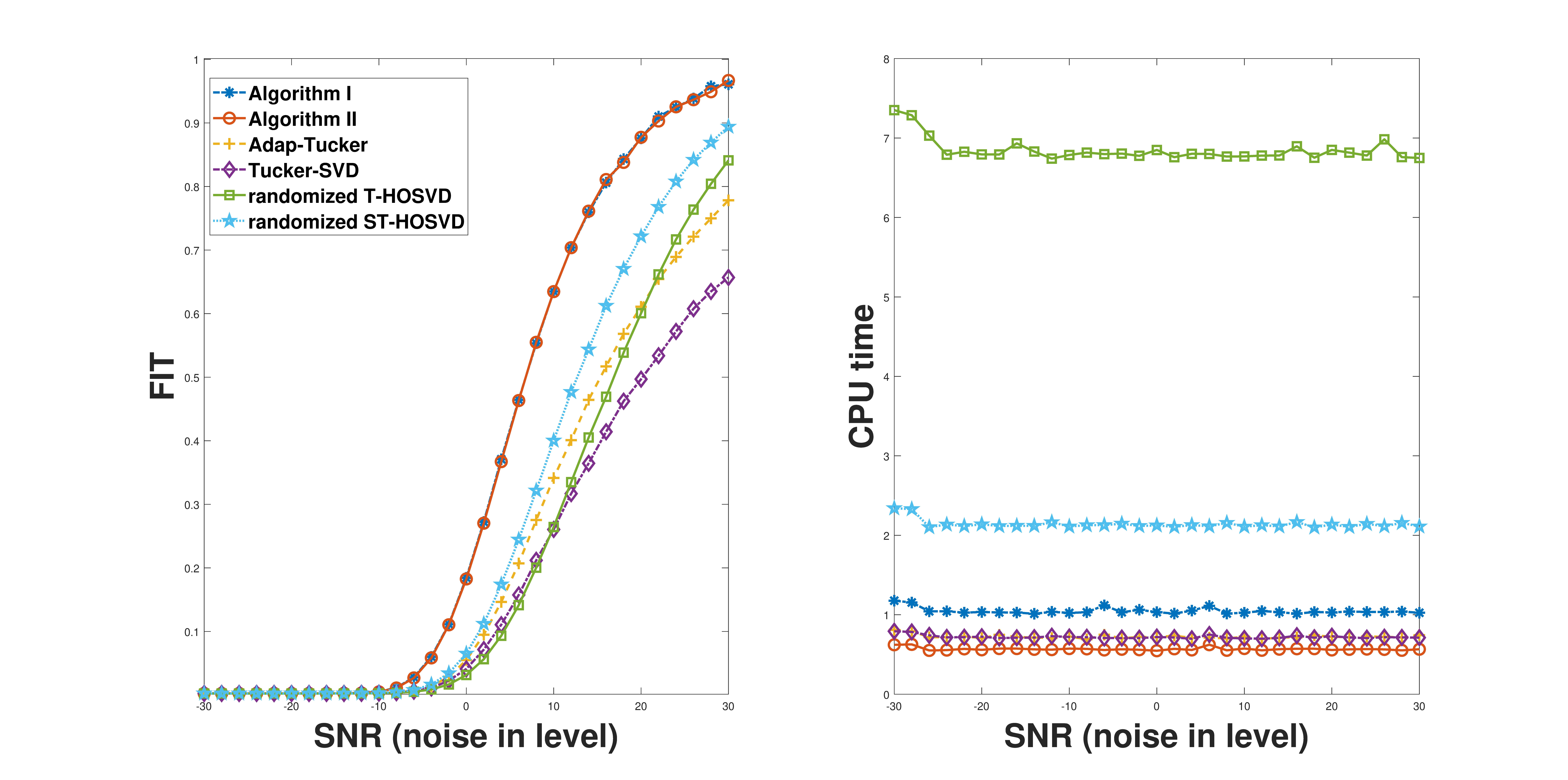}\\
		\includegraphics[width=3.8in]{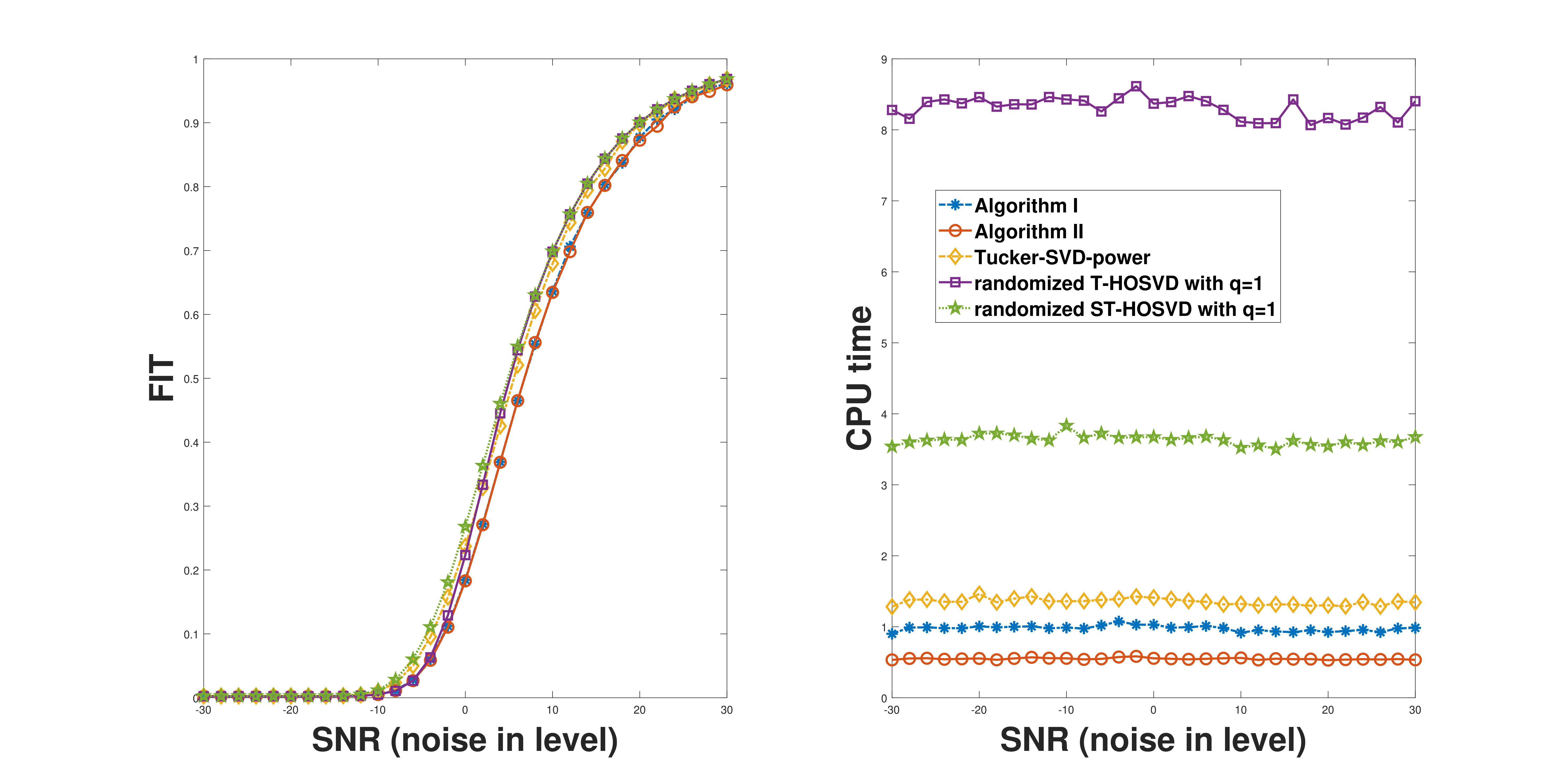}\\
	\end{tabular}
	\caption{Results of applying Algorithms \ref{RALMRA:alg2} and \ref{RALMRA:alg5}, HOOI, T-HOSVD, ST-HOSVD, randomized T-HOSVD, randomized ST-HOSVD, Adap-Tucker, Tucker-SVD and Tucker-SVD-power to the tensor $\mathcal{A}$ with ${\rm SNR}=-30,-28,\dots,28,30$.}\label{RALMRA:fig7}
\end{figure}

\subsection{A low multilinear rank tensor plus the white noise}

Let $\mathcal{B}\in\mathbb{R}^{600\times 600\times 600}$ be given in the Tucker form \cite{caiafa2010generalizing}: $\mathcal{B}=\mathcal{G}\times_1\mathbf{B}_1\times_2\mathbf{B}_2\times_3\mathbf{B}_3$, where the entries of $\mathcal{G}\in\mathbb{R}^{100\times 100\times 100}$ and $\mathbf{B}_n\in\mathbb{R}^{600\times 100}$ are i.i.d. Gaussian variables with zero mean and unit variance. The form of this test tensor $\mathcal{A}$ is given as $\mathcal{A}=\mathcal{B}+\beta\mathcal{N}$, where $\mathcal{N}\in\mathbb{R}^{600\times 600\times 600}$ is an unstructured perturbation tensor with different noise level $\beta$. The following signal-to-noise ratio (SNR) measure will be used:
	\begin{equation}\label{RALMRA:eqn24}
	{\rm SNR}\ [{\rm dB}]=10\log\left(\frac{\|\mathcal{A}\|_F^2}{\|\beta\mathcal{N}\|_F^2}\right).
	\end{equation}
	The FIT value for approximating the tensor $\mathcal{A}$ is defined by
	${\rm FIT}=1-{\rm RE}$,
	where ${\rm RE}$ is given in (\ref{RALMRA:eqn8}).

Here we assume that the desired multilinear rank is given by $\{100,100,100\}$. By applying Algorithms \ref{RALMRA:alg2} and \ref{RALMRA:alg5}, and the existing deterministic and randomized algorithms to the tensor $\mathcal{A}$ with different values of SNR, the related results are given in Figure \ref{RALMRA:fig7}.

It follows from this figure that a) the values of CPU time obtained by Algorithms \ref{RALMRA:alg2} and \ref{RALMRA:alg5} are less than that obtained by other algorithms, b) the values of FIT of Algorithms \ref{RALMRA:alg2} and \ref{RALMRA:alg5} are similar to that obtained by HOOI, T-HOSVD, ST-HOSVD, Tucker-SVD-power, randomized T-HOSVD with $q=1$ and randomized ST-HOSVD with $q=1$, and c) the values of FIT of Algorithms \ref{RALMRA:alg2} and \ref{RALMRA:alg5} are larger than that obtained by Adap-Tucker, Tucker-SVD, randomized T-HOSVD  and randomized ST-HOSVD.

\begin{figure}[htb]
	\setlength{\tabcolsep}{4pt}
	\renewcommand\arraystretch{1}
	\centering
	\begin{tabular}{c}
		\includegraphics[width=3.8in]{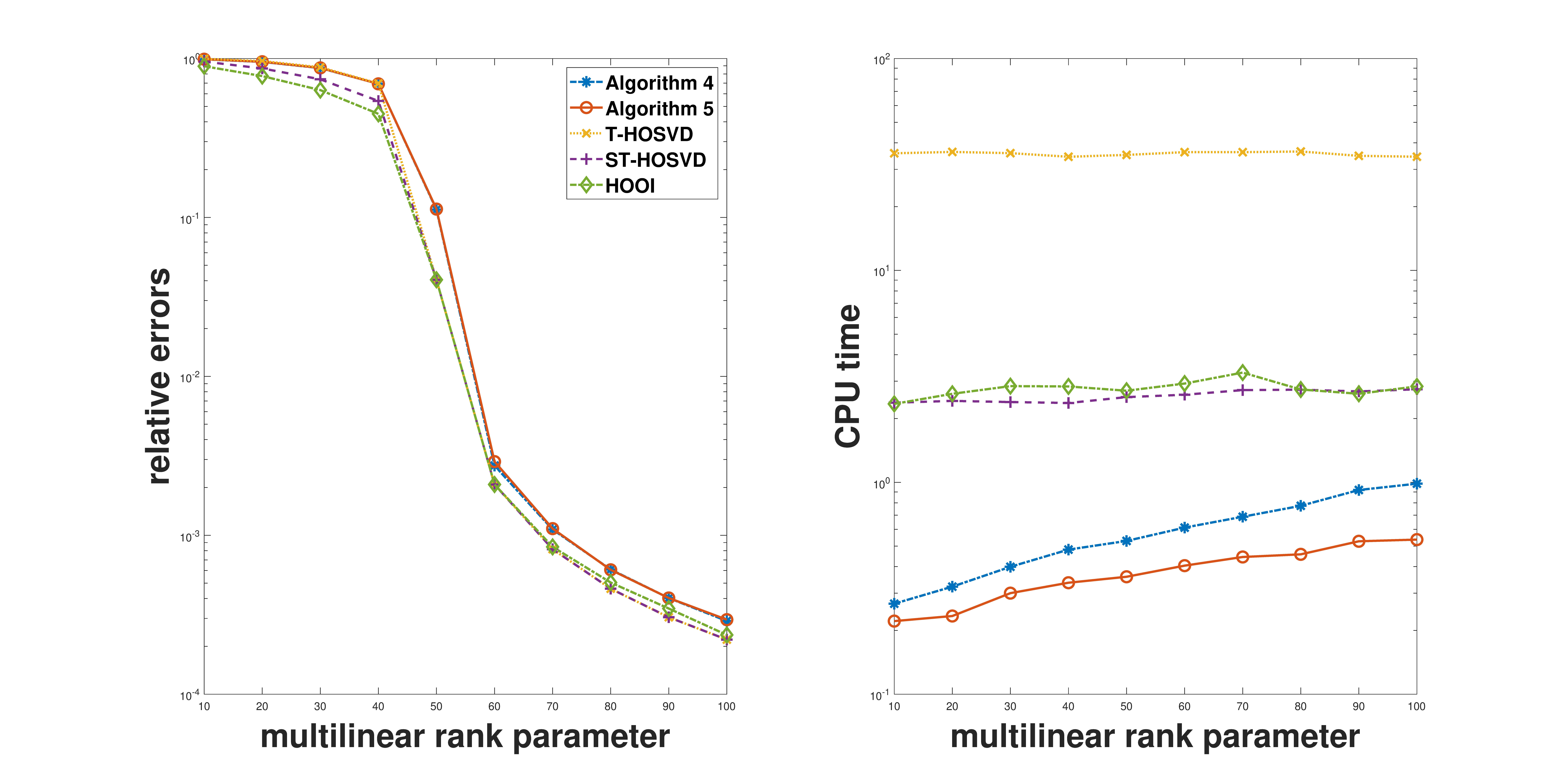}\\
		\includegraphics[width=3.8in]{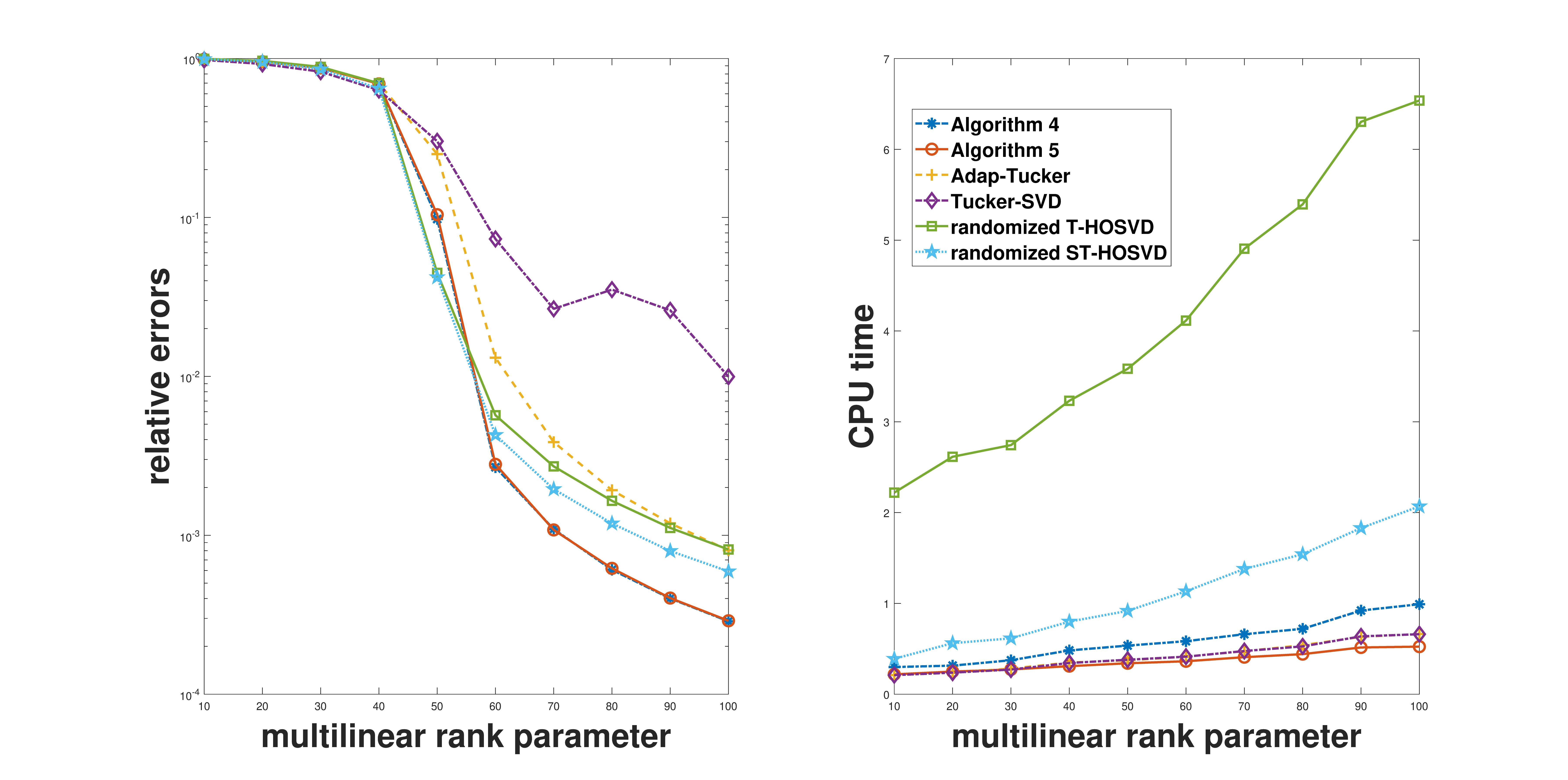}\\
		\includegraphics[width=3.8in]{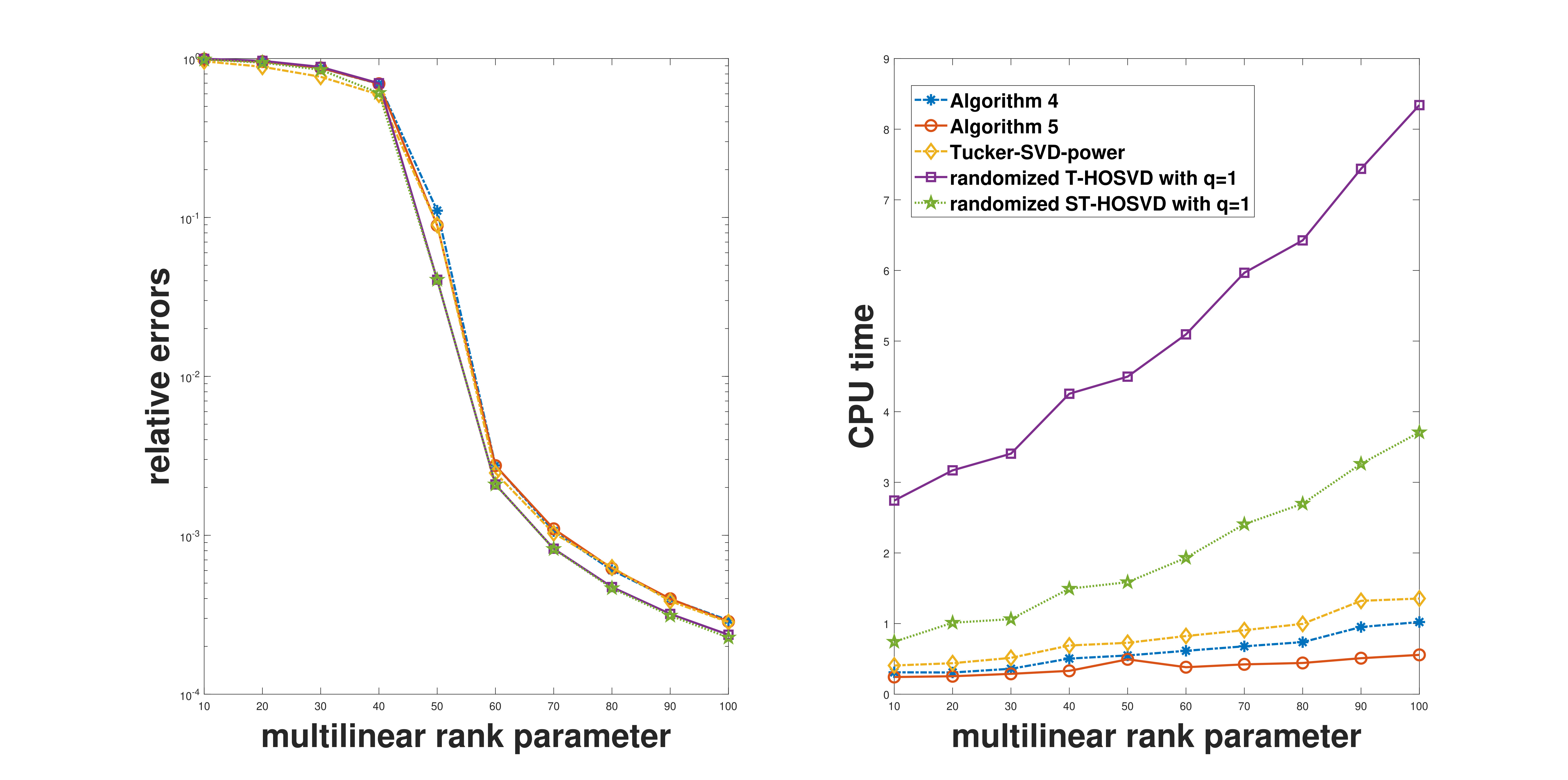}\\
	\end{tabular}
	\caption{Results of applying Algorithms \ref{RALMRA:alg2} and \ref{RALMRA:alg5}, HOOI, T-HOSVD, ST-HOSVD, randomized T-HOSVD, randomized ST-HOSVD, Adap-Tucker, Tucker-SVD and Tucker-SVD-power to the tensor $\mathcal{A}$ with $\mu=5,10,\dots,100$.}\label{RALMRA:fig8}
\end{figure}
\subsection{A more synthetic tensor}
We construct the input tensor $\mathcal{A}\in \mathbb{R}^{600\times 600\times 600}$ as $$\mathcal{A}={\rm tendiag}(\mathbf{v},[600,600,600])\times_1\mathbf{A}_1\times_2\mathbf{A}_2\times_3\mathbf{A}_3,$$
where ``{\rm tendiag}'' is a function in the MATLAB Tensor Toolbox ({\rm \cite{tensortool}}), which converts $\mathbf{v}$ to a diagonal tensor in $\mathbb{R}^{600\times 600\times 600}$, and $\mathbf{A}_n$ is an orthogonal basis for the column space of a standard Gaussian matrix $\mathbf{B}_n\in\mathbb{R}^{600\times 600}$. The vector $\mathbf{v}\in\mathbb{R}^{600}$ is given by
\begin{equation*}
v_i=
\begin{cases}
\begin{aligned}
1,\quad & i=1,2,\dots,50;\\
(i-49)^{-2},\quad & i=51,52,\dots,600.
\end{aligned}
\end{cases}
\end{equation*}

Figure \ref{RALMRA:fig8} illustrates that for different multilinear rank $\{\mu,\mu,\mu\}$ with $\mu=5,10,\dots,100$, in terms of RE, Algorithms \ref{RALMRA:alg2} and \ref{RALMRA:alg5} are similar to and slightly worse than HOOI, T-HOSVD, ST-HOSVD, Tucker-SVD-power, randomized T-HOSVD with $q=1$ and randomized ST-HOSVD with $q=1$, and better than Adap-Tucker, Tucker-SVD, randomized T-HOSVD and randomized ST-HOSVD; in terms of CPU time, Algorithm \ref{RALMRA:alg5} is the fastest one and Algorithms \ref{RALMRA:alg2} and \ref{RALMRA:alg5} are faster than other deterministic and randomized algorithms.

\subsection{Compression of tensors from real-world applications}
In this section, we consider the Extended Yale Face Dataset B\footnote{The Yale Face Database is at \url{http://vision.ucsd.edu/~iskwak/ExtYaleDatabase/ExtYaleB.html}.} \cite{georghiades2001from} and the Washington DC Mall dataset\footnote{ The Washington DC Mall dataset is at \url{https://engineering.purdue.edu/biehl/MultiSpec/hyperspectral.html}.}. The Extended Yale Face Dataset B has $560$ images that contain the first $12$ possible illuminations of $28$ different people under the first poss, where each image has $480\times 640$ pixels in a grayscale range. After selecting 400 images and extracting a part of $400\times 400$ pixels from each image, there is collected into a $400\times 400\times 400$ tensor with real numbers.
	
The sensor system used in this case measured pixel response in 210 bands in the 0.4 to 2.4 um region of the visible and infrared spectrum. Bands in the 0.9 and 1.4 um region where the atmosphere is opaque have been omitted from the data set, leaving 191 bands. The data set contains 1280 scan lines with 307 pixels in each scan line. When we choose the first 600 scan lines, this dataset is collected as a tensor of size $600\times 307\times 191$.

For the first dataset, we set $\mu_1=\mu_2=\mu_3=\mu=20,40,\dots,200$. We also set $\mu_1=40,80,\dots,400$, $\mu_2=20,40,\dots,200$ and $\mu_3=10,20,\dots,100$ for the second dataset. The values of other parameters are the same as that in the above two sections. By applying these algorithms to the tensors from two datasets, the related results are shown in Figures \ref{RALMRA:fig9} and \ref {RALMRA:fig10}, respectively.

For both the Extended Yale Face Dataset B and the Washington DC Mall dataset, in terms of CPU time, Algorithm \ref{RALMRA:alg5} is the fastest one and Algorithms \ref{RALMRA:alg2} and \ref{RALMRA:alg5} are faster than other algorithms; and in terms of RE, Algorithms \ref{RALMRA:alg2} and \ref{RALMRA:alg5} are slightly less than HOOI, T-HOSVD, ST-HOSVD, Tucker-SVD-power, randomized T-HOSVD with $q=1$ and randomized ST-HOSVD with $q=1$, and better than Adap-Tucker, Tucker-SVD, randomized T-HOSVD and randomized ST-HOSVD. It is worthy noting that in terms of RE, for both the Extended Yale Face Dataset B, Algorithm \ref{RALMRA:alg5} is less than Algorithm \ref{RALMRA:alg2}, and for the Washington DC Mall dataset, Algorithm \ref{RALMRA:alg5} is similar to Algorithm \ref{RALMRA:alg2}.

\begin{figure}[htb]
	\setlength{\tabcolsep}{4pt}
	\renewcommand\arraystretch{1}
	\centering
	\begin{tabular}{c}
		\includegraphics[width=3.8in]{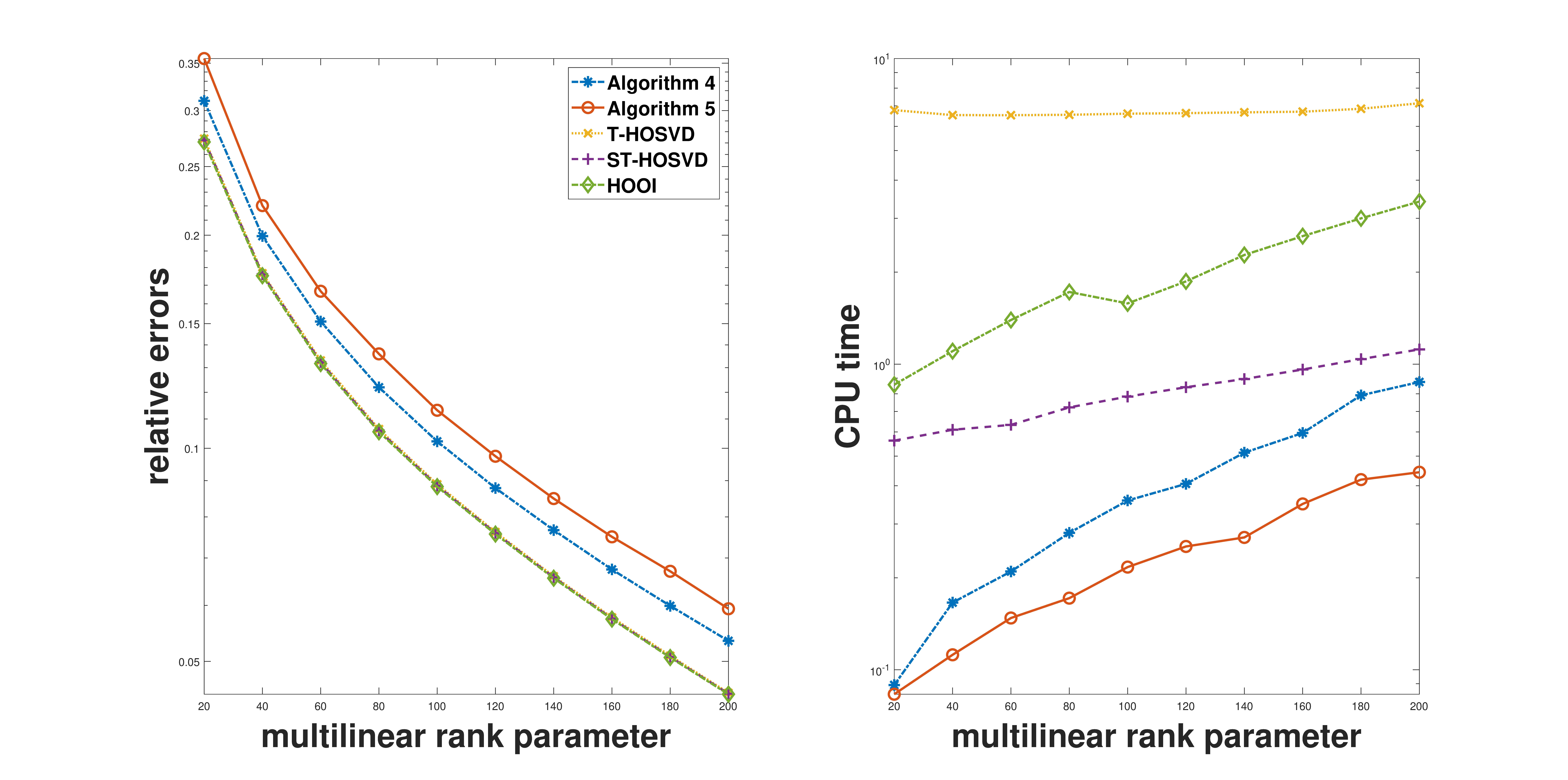}\\
		\includegraphics[width=3.8in]{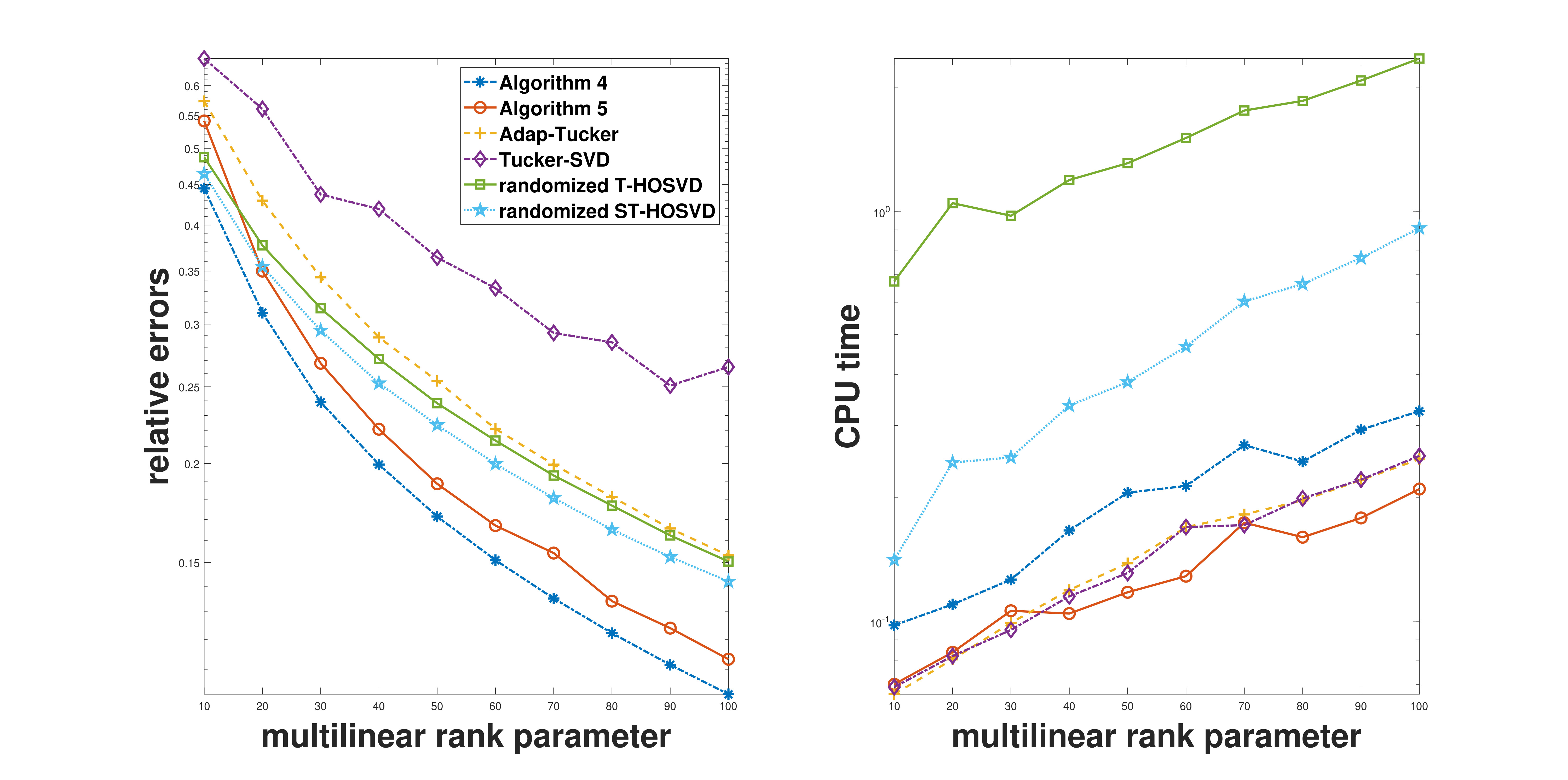}\\
		\includegraphics[width=3.8in]{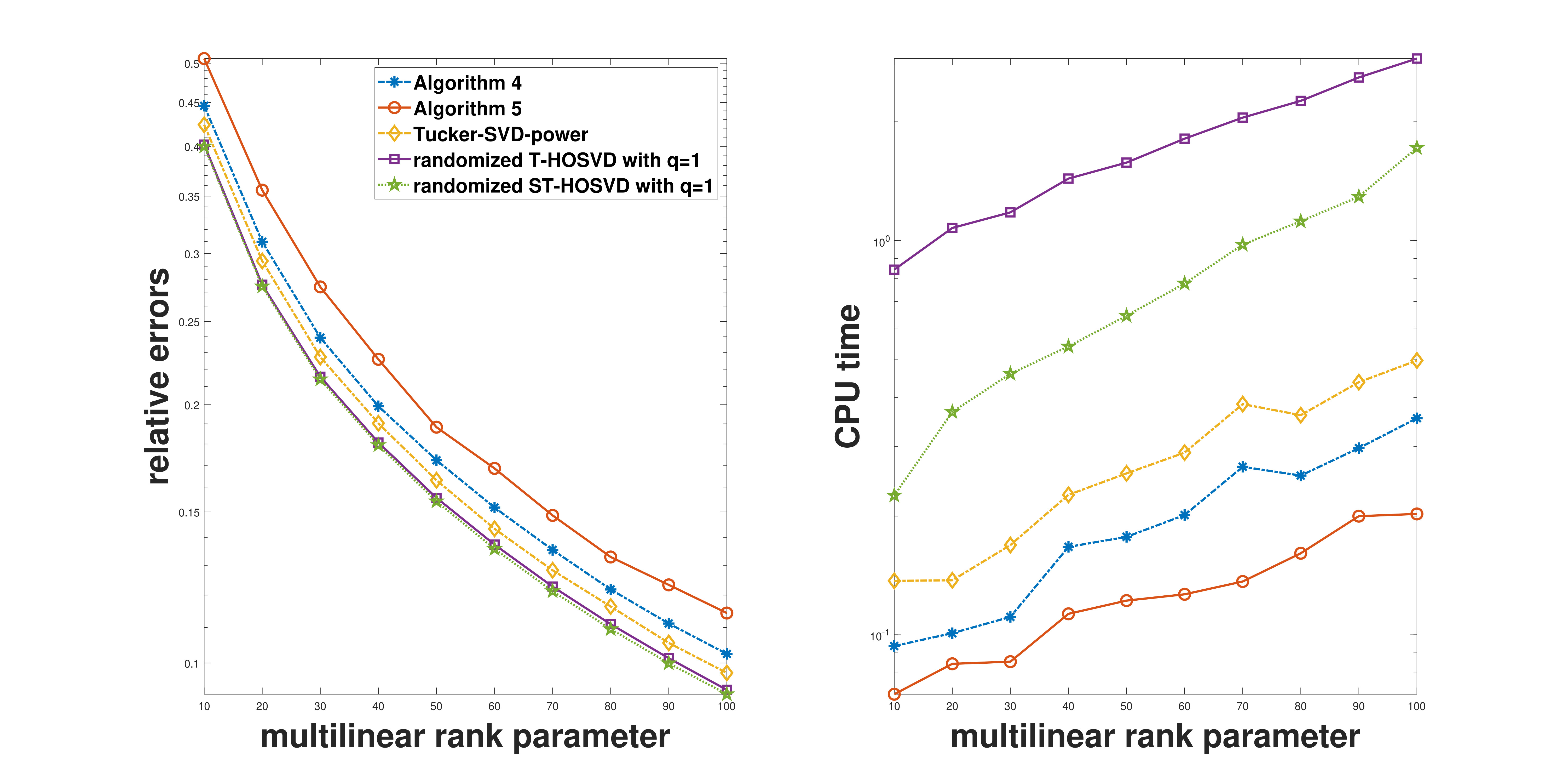}\\
	\end{tabular}
	\caption{Results of applying Algorithms \ref{RALMRA:alg2} and \ref{RALMRA:alg5}, HOOI, T-HOSVD, ST-HOSVD, randomized T-HOSVD, randomized ST-HOSVD, Adap-Tucker, Tucker-SVD and Tucker-SVD-power to the tensor from the Extended Yale Face Dataset B with $\mu=20,40,\dots,200$.}\label{RALMRA:fig9}
\end{figure}

\begin{figure}[htb]
	\setlength{\tabcolsep}{4pt}
	\renewcommand\arraystretch{1}
	\centering
	\begin{tabular}{c}
		\includegraphics[width=3.8in]{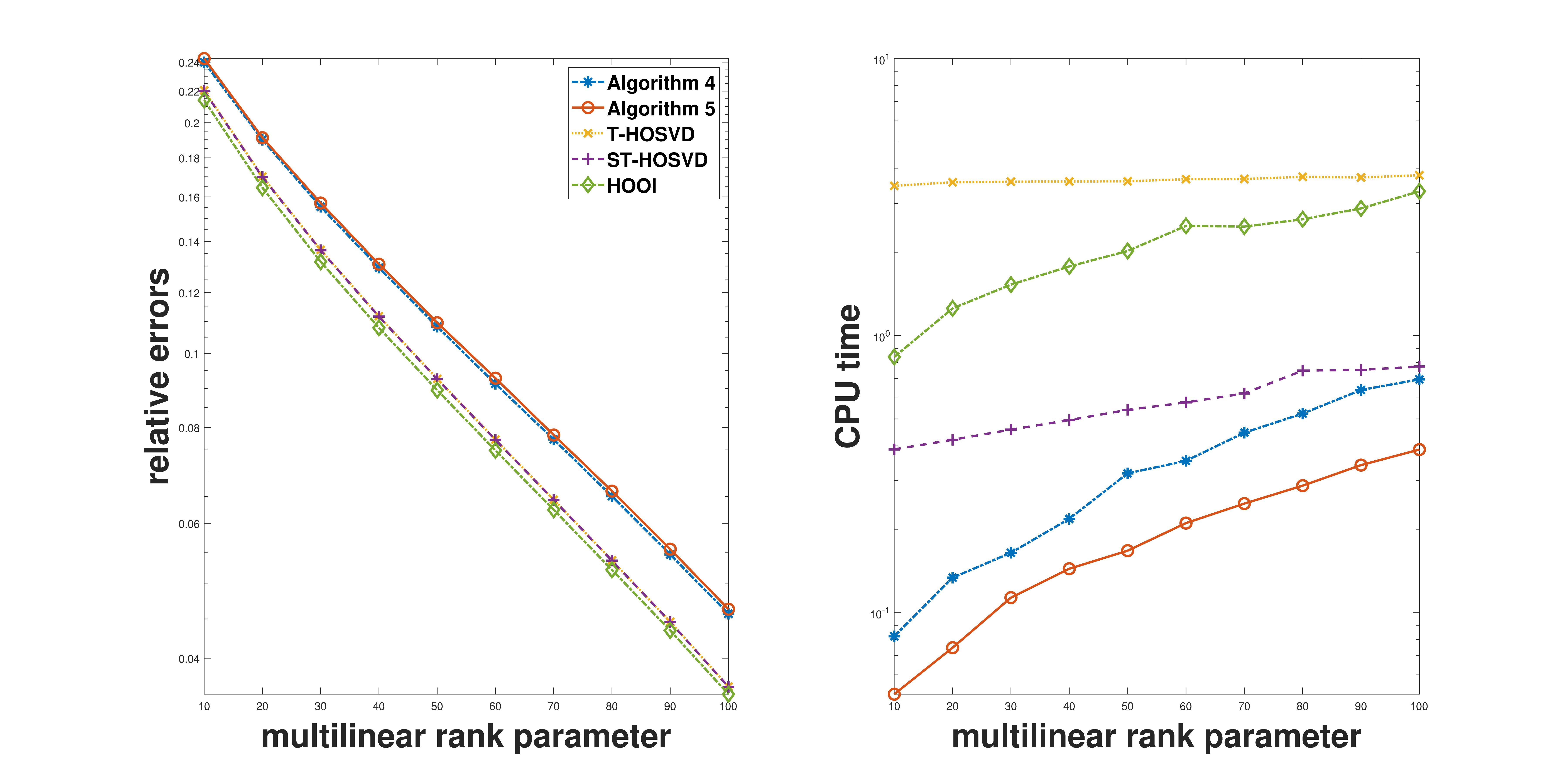}\\
		\includegraphics[width=3.8in]{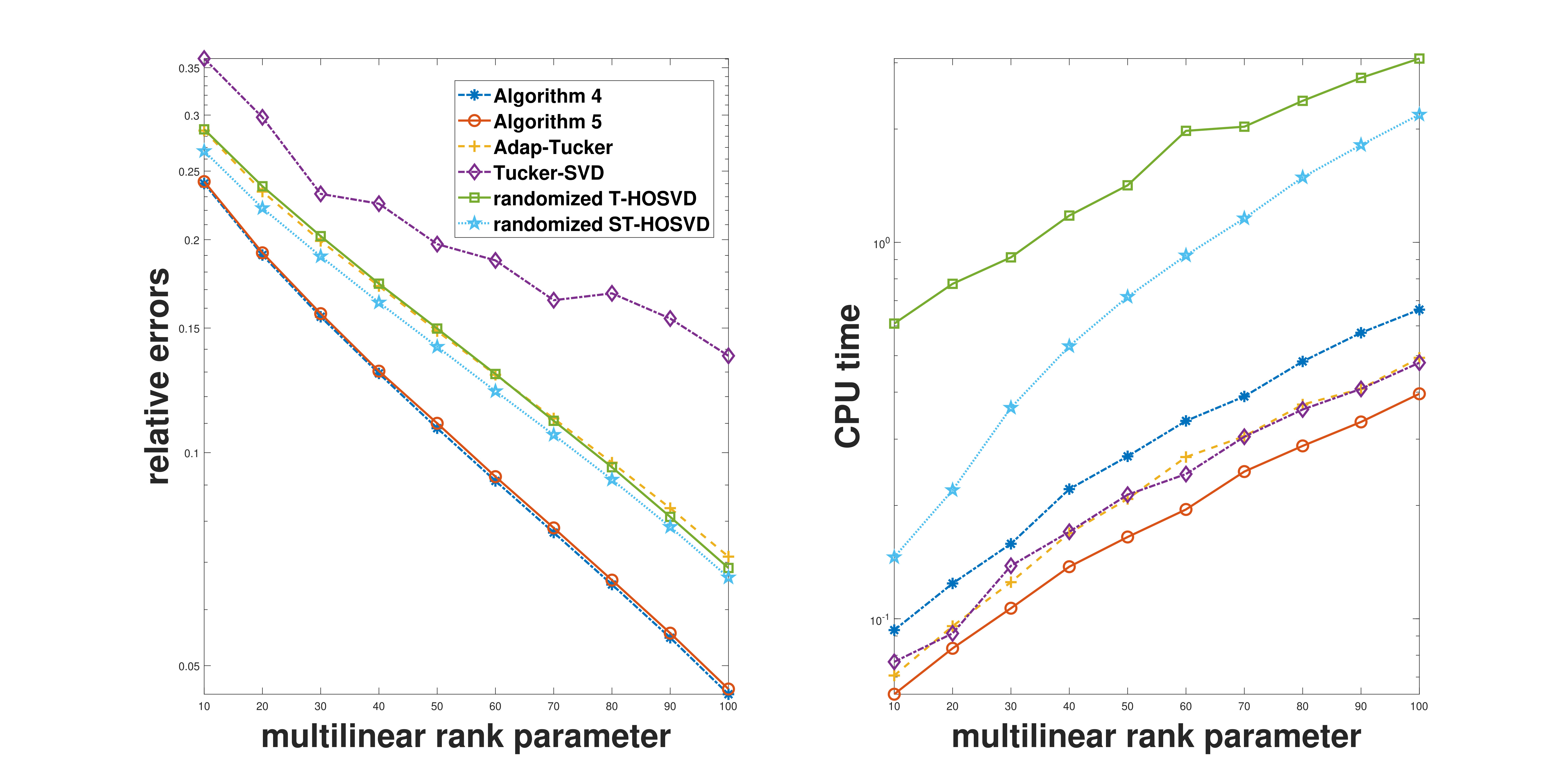}\\
		\includegraphics[width=3.8in]{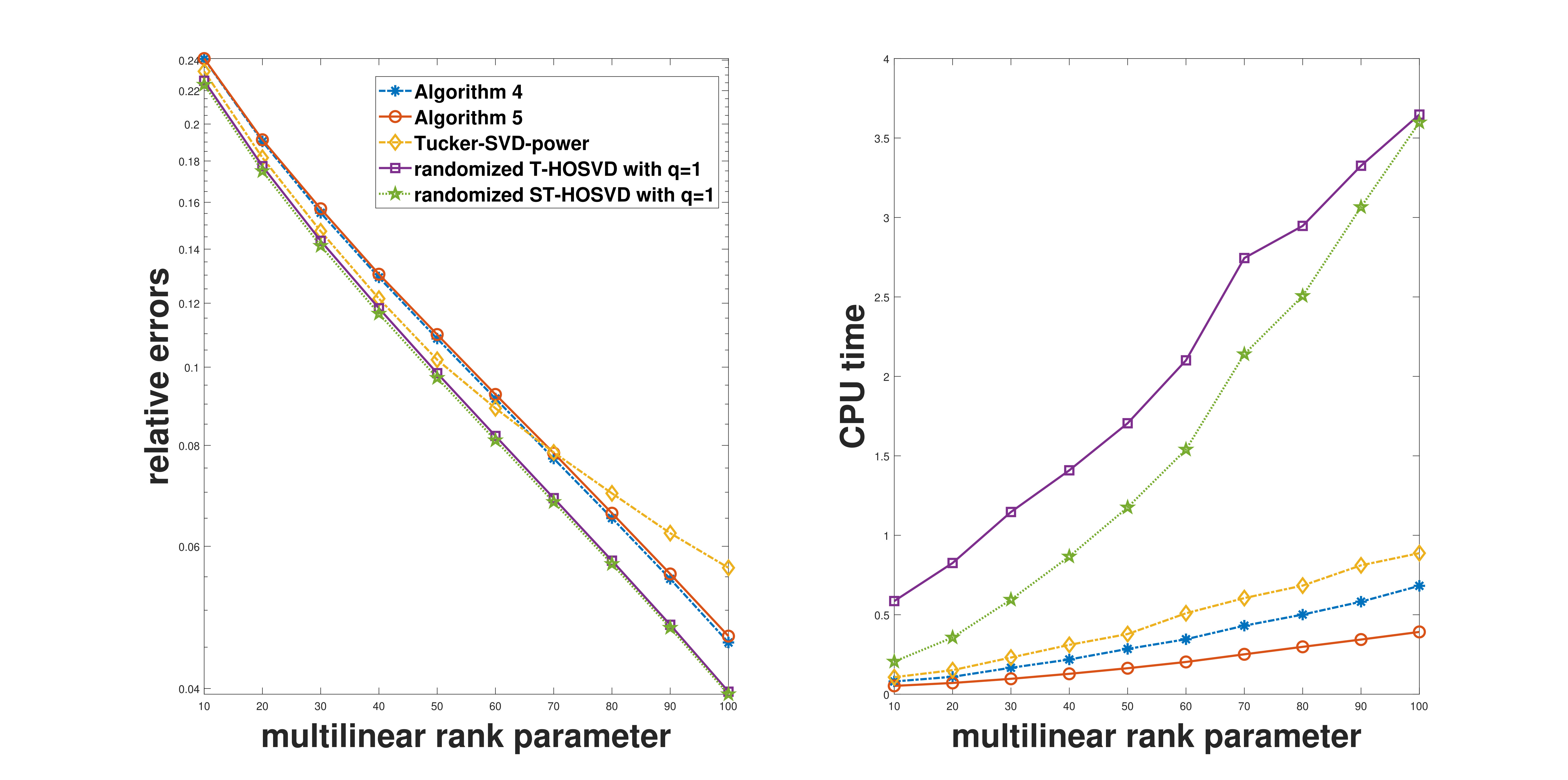}\\
	\end{tabular}
	\caption{Results of applying Algorithms \ref{RALMRA:alg2} and \ref{RALMRA:alg5}, HOOI, T-HOSVD, ST-HOSVD, randomized T-HOSVD, randomized ST-HOSVD, Adap-Tucker, Tucker-SVD and Tucker-SVD-power to the tensor from the Washington DC Mall dataset with $\mu_1=40,80,\dots,400$, $\mu_2=20,40,\dots,200$ and $\mu_3=10,20,\dots,100$.}\label{RALMRA:fig10}
\end{figure}

\subsection{Two examples for fourth-order tensors}

In the above four sections, we compare Algorithms \ref{RALMRA:alg2} and \ref{RALMRA:alg5} with the existing deterministic and randomized algorithms via some third-order tensors. We now introduce two fourth-order tensor to further illustrate the proposed algorithms.

The first tensor $\mathcal{A}\in\mathbb{R}^{200\times 200\times 200\times 200}$ is given by
\begin{equation*}
\mathcal{A}=\mathcal{G}\times_1\mathbf{B}_1\times_2\mathbf{B}_2\times_3\mathbf{B}_3\times_4\mathbf{B}_4+\beta\mathcal{N},
\end{equation*}
where the entries of $\mathcal{G}\in\mathbb{R}^{50\times 50\times 50\times 50}$  and $\mathbf{B}_n\in\mathbb{R}^{200\times 50}$ are i.i.d. Gaussian variables with zero mean and unit variance, and $\mathcal{N}\in\mathbb{R}^{200\times 200\times 200\times 200}$ is an unstructured perturbation tensor with different noise level $\beta$. The signal-to-noise ratio (SNR) measure is given in (\ref{RALMRA:eqn24}) and the definition of RLNE is similar to (\ref{RALMRA:eqn8}) for the case of third-order tensors.

The second tensor $\mathcal{B}\in\mathbb{R}^{200\times 200\times 200\times 200}$ is a sparse tensor such that
\begin{equation*}
\mathcal{B}=\sum_{i=1}^{50}\frac{1000}{i^2}\mathbf{x}_i\circ\mathbf{y}_i\circ\mathbf{z}_i\circ\mathbf{w}_i+\sum_{i=51}^{200}\frac{1}{i^2}\mathbf{x}_i\circ\mathbf{y}_i\circ\mathbf{z}_i\circ\mathbf{w}_i
\end{equation*}
where $\mathbf{x}_i$, $\mathbf{y}_i$, $\mathbf{z}_i$ and $\mathbf{w}_i$ are sparse vectors in $\mathbb{R}^{200}$ with only 0.05 sparsity.

For the tensor $\mathcal{A}$, the desired multilinear rank is set to $\{50,50,50,50\}$. The desired multilinear rank for the tensor $\mathcal{B}$ is denoted by $\{\mu,\mu,\mu,\mu\}$. In Algorithms \ref{RALMRA:alg2} and \ref{RALMRA:alg5}, we assume that $K=10$ and $q=1$. The processing order in Algorithm \ref{RALMRA:alg5} is $\{1,2,3,4\}$. From the description in Section \ref{RALMRA:sec3:3}, in Algorithm \ref{RALMRA:alg5}, we set $T_1={\rm ceil}(\alpha I^3)$, $T_2={\rm ceil}(\alpha \mu I^2)$, $T_3={\rm ceil}(\alpha \mu^2I)$ and $T_4={\rm ceil}(\alpha \mu^3)$ with $\alpha=0.01$.

By applying Algorithms \ref{RALMRA:alg2} and \ref{RALMRA:alg5}, ST-HOSVD, HOOI, randomized ST-HOSVD and randomized ST-HOSVD with $q=1$ to the tensor $\mathcal{A}$ with ${\rm SNR}=-30,-26,\dots,26,30$ and the tensor $\mathcal{B}$ with $\mu=5,10,\dots,50$, the related results are shown in Figures \ref{RALMRA:fig11} and \ref{RALMRA:fig12}, respectively. From these two figures, in terms of CPU time, Algorithm \ref{RALMRA:alg5} is the fastest one and Algorithms \ref{RALMRA:alg2} and \ref{RALMRA:alg5} are faster than ST-HOSVD, HOOI, randomized ST-HOSVD and randomized ST-HOSVD with $q=1$.

In terms of FIT, a) it follows from Figure \ref{RALMRA:fig11} that randomized ST-HOSVD is the worst one, Algorithm \ref{RALMRA:alg5} is worse than Algorithm \ref{RALMRA:alg2}, ST-HOSVD, HOOI and randomized ST-HOSVD with $q=1$, and Algorithm \ref{RALMRA:alg2} is similar to ST-HOSVD, HOOI and randomized ST-HOSVD with $q=1$; and b) it follows from Figure \ref{RALMRA:fig12} that Algorithms \ref{RALMRA:alg2} and \ref{RALMRA:alg5} are slightly worse than ST-HOSVD, HOOI and randomized ST-HOSVD with $q=1$.
\begin{figure}[htb]
	\centering
	\includegraphics[width=3.8in]{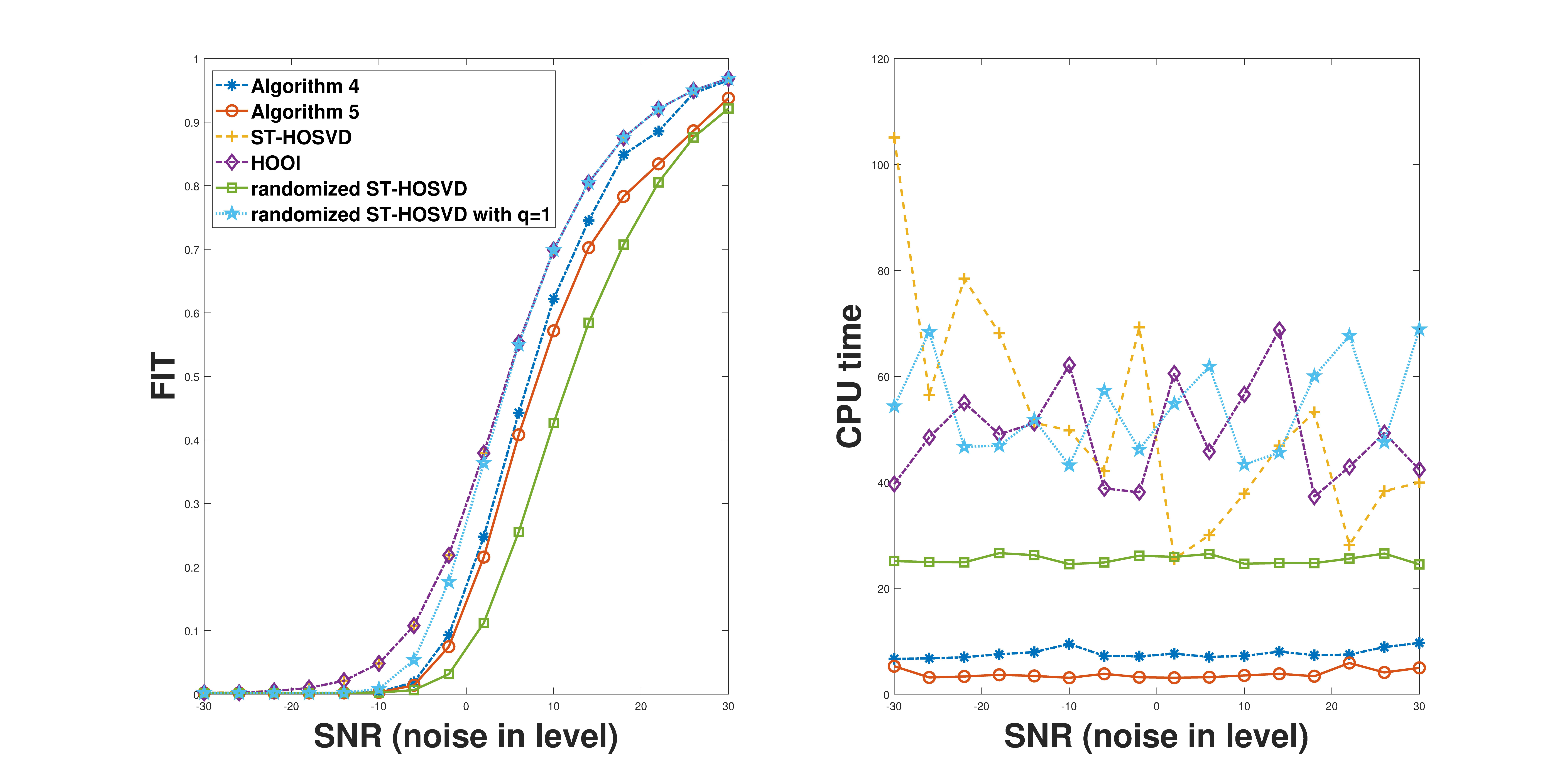}\\
	\caption{Results of applying Algorithms \ref{RALMRA:alg2} and \ref{RALMRA:alg5}, ST-HOSVD, HOOI, randomized ST-HOSVD and randomized ST-HOSVD with $q=1$ to $\mathcal{A}$ with ${\rm SNR}=-30,-28,\dots,28,30$.}\label{RALMRA:fig11}
\end{figure}

\begin{figure}[htb]
	\centering
	\includegraphics[width=3.8in]{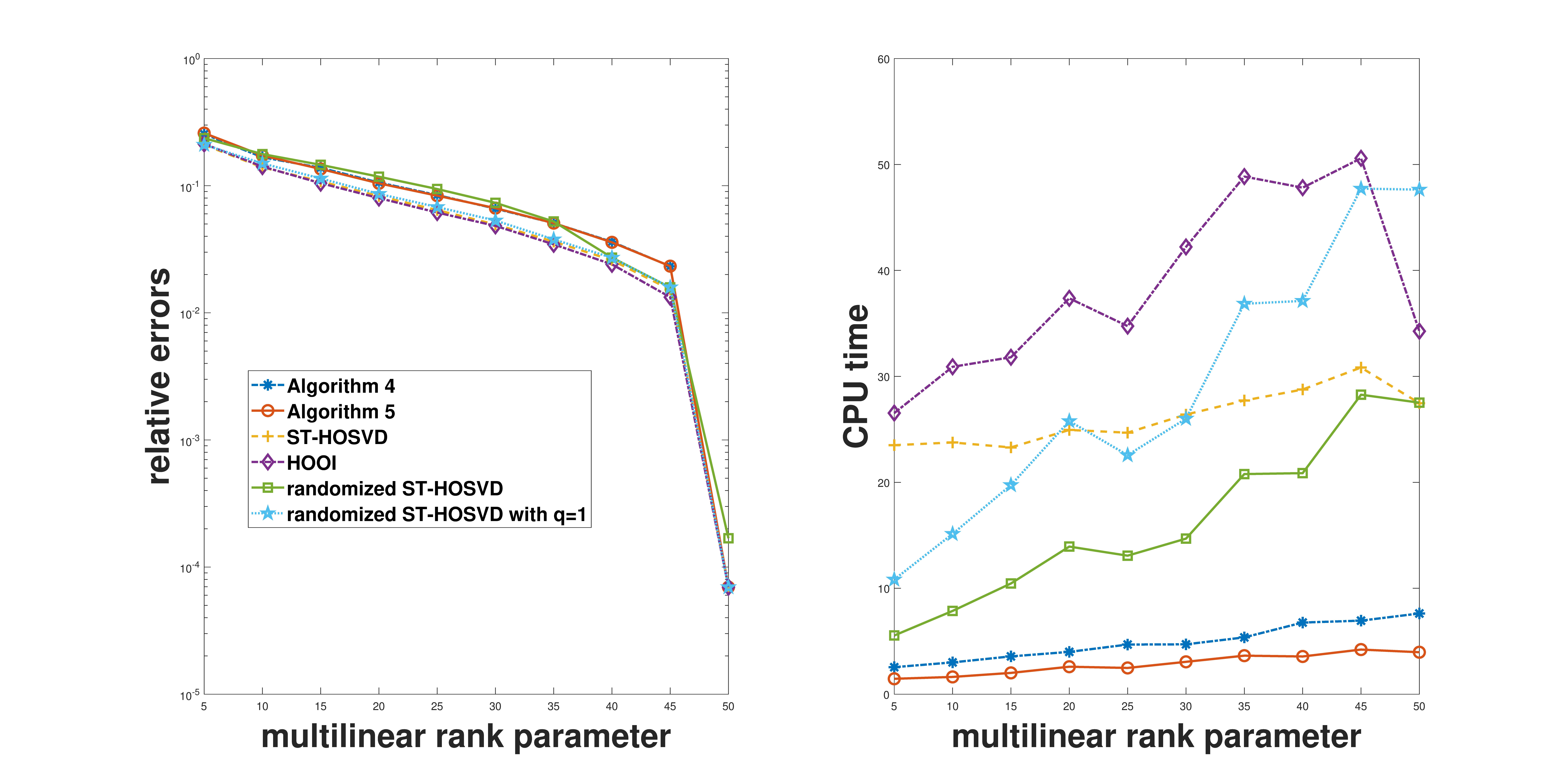}\\
	\caption{Results of applying Algorithms \ref{RALMRA:alg2} and \ref{RALMRA:alg5}, ST-HOSVD, HOOI, randomized ST-HOSVD and randomized ST-HOSVD with $q=1$ to $\mathcal{B}$ with $\mu=5,10,\dots,50$.}\label{RALMRA:fig12}
\end{figure}

\section{Conclusions}
\label{RALMRA:sec7}

For a given multilinear rank, two efficient randomized algorithms for computing the low multilinear rank approximation were designed by combining random projection and power scheme. To reduce the complexities of these two algorithms, we derived two other fast algorithms by applying approximate matrix multiplication into these two algorithms. We analyzed their theoretical results by using the bounds of singular values of standard Gaussian matrices and the theoretical results of approximate matrix multiplication. The comparisons of the proposed algorithms with the existing deterministic and randomized algorithms were displayed by high-dimensional image analysis.

We call Problem \ref{RALMRA:pro1} as the fixed multilinear rank problem for low multilinear rank approximations. We now introduce the fixed-precision problem for low multilinear rank approximations, which summarized in the following problem.
\begin{problem}
	\label{RALMRA:pro2}
	Suppose that $\mathcal{A}\in\mathbb{R}^{I_1\times I_2\times \dots\times I_N}$. For a given $\varepsilon>0$, the goal is to find $N$ positive integers $\mu_n$ and orthonormal matrices ${\bf Q}_{n}\in\mathbb{R}^{I_{n}\times \mu_{n}}$ such that
	\begin{equation*}
	\|\mathcal{A}-\mathcal{A}\times_1(\mathbf{Q}_1\mathbf{Q}_1^\top)\times_2(\mathbf{Q}_2\mathbf{Q}_2^\top)\dots
	\times_N(\mathbf{Q}_N\mathbf{Q}_N^\top)\|_F\leq \varepsilon.
	\end{equation*}
\end{problem}

Che and Wei \cite{che2019randomized} proposed an adaptive randomized algorithm for solving Problem \ref{RALMRA:pro2}. By this algorithm, we can derive a desired multilinear rank for $\mathcal{A}$ with a given $\varepsilon>0$. By using the efficient randomized algorithms for the fixed-precision problem of low rank approximations (\cite{yu2018efficient}) to each mode unfolding matrix, Minster {\it et al.} \cite{minster2020randomized} studied adaptive randomized algorithms to compute low multilinear rank approximations for use in applications where the target multilinear rank is not known beforehand. Hence, one issue is to design fast and efficient algorithms for solving Problem \ref{RALMRA:pro2}.

\section*{Acknowledgements}

 This work is supported by the Hong Kong Innovation and Technology Commission (InnoHK Project CIMDA), the Hong Kong Research Grants Council (Project 11204821),  City University of Hong Kong (Projects 9610034 and 9610460) and the National Natural Science Foundation of China under the grant 12271108.

\end{document}